\documentclass[reqno, 12pt, oneside]{amsart}

\usepackage{mathtools, bm}
\usepackage[all]{xy}
\usepackage{xcolor}
\usepackage{MnSymbol}
\usepackage{microtype}

\DeclareFontFamily{OT1}{pzc}{}
\DeclareFontShape{OT1}{pzc}{m}{it}{<-> s * [1.200] pzcmi7t}{}
\DeclareMathAlphabet{\mathpzc}{OT1}{pzc}{m}{it}

\usepackage{hyperref}

\hypersetup{colorlinks}
\definecolor{darkred}{rgb}{0.5,0,0}
\definecolor{darkgreen}{rgb}{0,0.5,0}
\definecolor{darkblue}{rgb}{0,0,0.5}

\makeatletter 
\@addtoreset{equation}{section}
\makeatother  

\numberwithin{equation}{section}

\usepackage{tikz-cd}
\usepackage{tikz}
\usetikzlibrary{arrows}
\usepackage{float}

\newcommand{\mb}{\mathbb}
\newcommand{\mf}{\mathfrak}
\newcommand{\mc}{\mathcal}

\renewcommand{\i}{{\bf i}}
\newcommand{\Gammait}{{\mathit{\Gamma}}}
\newcommand{\Piit}{{\mathit{\Pi}}}

\newcommand{\ov}{\overline}

\newcommand{\mz}{\mathpzc}
\DeclareMathAlphabet{\mathbbmsl}{U}{bbm}{m}{sl}

\newcommand{\beq}{\begin{equation}}
\newcommand{\eeq}{\end{equation}}
\newcommand{\beqn}{\begin{equation*}}
\newcommand{\eeqn}{\end{equation*}}

\newcommand{\uds}[1]{\underline{\smash{#1}}}

\renewcommand{\setminus}{\smallsetminus}

\newcommand{\wh}{\widehat}

\newcommand{\preq}{\preccurlyeq}

\usepackage[left= 3 cm, right= 3 cm, top= 2.7 cm, bottom=2.9 cm, foot= 1.2 cm, marginparwidth=2.3 cm, marginparsep=0.3 cm]{geometry}

\author[Tian]{Gang Tian}
\address{
Beijing International Center for Mathematical Research\\
Peking University\\
Beijing, China\\
and\\
Department of Mathematics \\
Princeton University\\
Fine Hall, Washington Road\\
Princeton, NJ 08544 USA
}
\email{tian@math.princeton.edu}

\author[Xu]{Guangbo Xu}
\address{
Department of Mathematics \\
Princeton University\\
Fine Hall, Washington Road\\
Princeton, NJ 08544 USA
}
\email{guangbox@math.princeton.edu}

\date{\today}

\newtheorem{thm}{Theorem}[section]
\newtheorem{lemma}[thm]{Lemma}
\newtheorem{cor}[thm]{Corollary}
\newtheorem{prop}[thm]{Proposition}

\theoremstyle{definition}
\newtheorem{defn}[thm]{Definition}
\newtheorem{hyp}[thm]{Hypothesis}
	
\theoremstyle{remark}
\newtheorem{rem}[thm]{Remark}

\newtheorem{notation}[thm]{Notation}

\title[Virtual cycles of gauged Witten equation]{Virtual cycles of gauged Witten equation}

\begin{document}

\thanks{G. T. is supported by DMS-1607091 and a grant from NSFC. G. X. is supported by AMS-Simons Travel Grant.}

\begin{abstract}
We construct virtual cycles on moduli spaces of perturbed gauged Witten equation over a fixed smooth $r$-spin curve, under the framework of \cite{Tian_Xu}. Together with the wall-crossing formula proved in the companion paper \cite{Tian_Xu_4}, it completes the construction of the correlation function for the gauged linear $\sigma$-model announced in \cite{Tian_Xu_2} as well as the proof of its invariance.
\end{abstract}

\maketitle

\setcounter{tocdepth}{1}
\tableofcontents 

\section{Introduction}

This paper, together with a companion paper \cite{Tian_Xu_4}, is the third input in a series of papers aiming at a mathematically rigorous theory of the gauged linear $\sigma$-model (GLSM) as Witten proposed in \cite{Witten_LGCY}, following the two papers \cite{Tian_Xu, Tian_Xu_2}. GLSM provides a fundamental framework in many studies of mathematics and physics related to string theory; it is of crucial importance in Hori--Vafa's physical proof of mirror symmetry \cite{Hori_Vafa}. It is also closely related to Gromov--Witten theory. However, despite many efforts, a mathematical foundation of GLSM has not been completely established\footnote{Here we mean the ${\mc N} = (2, 2)$, A-model closed-string theory, with a nontrivial superpotential.} and is needed for many important applications, for example, a mathematical proof of the Landau--Ginzburg/Calabi--Yau correspondence, and a geometric understanding of mirror symmetry.

In \cite{Tian_Xu, Tian_Xu_2}, we initiated our project on constructing a mathematical theory of GLSM. The central objects are the so-called gauged Witten equation and its moduli space. The gauged Witten equation is a system combining both the Witten equation, which was used to construct the A-side theory of orbifold Landau--Ginzburg models \cite{FJR1, FJR2, FJR3}, and the symplectic vortex equation, which was used to construct the Hamiltonian--Gromov--Witten invariants \cite{Mundet_thesis, Mundet_2003}, \cite{Cieliebak_Gaio_Salamon_2000} and \cite{Mundet_Tian_2009}. Roughly speaking, given a noncompact K\"ahler manifold $X$, an action on $X$ by a reductive Lie group $G^{\mb C}$ with a moment map $\mu: X \to {\mf g}$, and a holomorphic function $W: X \to {\mb C}$ which is homogeneous with respect to the $G^{\mb C}$-action, the gauged Witten equation reads
\beq\label{eqn11}
\left\{ \begin{aligned} \ov\partial_A u + \nabla W (u) &\ = 0,\\
                       * F_A + \mu(u) &\ = 0,
\end{aligned}
\right.
\eeq
where $A$ is a connection on a $G$-bundle $P$ over a Riemann surface $\Sigma$, and $u$ is a section $u: P \to X$. 

The moduli space of gauge equivalence classes of solutions contains the information of GLSM that can be extracted mathematically. In \cite{Tian_Xu} we have obtained crucial analytical results about the gauged Witten equation and its moduli space. These results will be recalled in due course in the main body of this paper.

As in Gromov--Witten theory, the correlation function of GLSM is the most important numerical output, which leads to many interesting algebraic structures. The (symplectic) Gromov--Witten invariants, which can naively be interpreted as ``curve counting,'' were constructed on any symplectic manifolds (cf. \cite{Li_Tian}, \cite{Fukaya_Ono}, also see \cite{Ruan_Tian}, \cite{Ruan_Tian_97} on semi-positive symplectic manifolds). One major difficulty of the construction is the lack of transversality, namely, the moduli spaces of holomorphic curves may not be smooth or of expected dimensions. One has to perturb the equation in certain way to obtain well-defined countings. This is a highly non-trivial problem, which becomes more sophisticated when involved with maps defined on singular domains and nontrivial automorphism groups.

The method of constructing a system of consistent perturbations for the stratified moduli space and extracting topological information (the virtual cycle) is often  called the virtual technique. After \cite{Li_Tian} and \cite{Fukaya_Ono}, different versions of the virtual technique have been developed (in the non-algebraic case). We refer the readers to \cite{Liu_Tian_Floer}, \cite{HWZ1, HWZ2, HWZ3}, \cite{Chen_Li_Wang_1}, \cite{Joyce_07, Joyce_14}, \cite{Pardon_virtual}, \cite{MW_1, MW_2, MW_3}, \cite{FOOO_2012, FOOO_2016, FOOO_2015, FOOO_2017} for more recent developments. 

In \cite{Tian_Xu_2} we outlined the definition of the correlation function assuming the existence and good properties of a virtual cycle. In this paper we give the details of the virtual cycle construction; in the companion paper \cite{Tian_Xu_4} we prove an important property of the virtual cycles. Here we briefly explain our results. Let $X$ be a K\"ahler manifold and $Q: X \to {\mb C}$ be a holomorphic function. Suppose there is a ${\mb C}^*$-action on $X$ making $Q$ homogeneous. Let $\tilde{X} = {\mb C} \times X$ have an induced action by $G = T^2$. Let $W: \tilde{X} \to {\mb C}$ be $W (p, x) = pQ(x)$. There is also a moment map $\mu: \tilde{X} \to {\bf Lie} G$. Then one can write down the gauged Witten equation \eqref{eqn11} over a so-called $r$-spin curve, which is an ordinary punctured Riemann surface $\Sigma$ with an orbifold line bundle $L$ whose $r$-th tensor power is isomorphic to the canonical bundle twisted with the orbifold data. Moreover, we assume that there is a vector space ${\bf V}$ of holomorphic functions with nice properties in order to properly perturb the gauged Witten equation. Our first main theorem is the following (the same as Theorem \ref{thm31}).
\begin{thm}\label{thm11}
Let ${\mc C}$ be a smooth $r$-spin curve with broad punctures ${\rm z}_1, \ldots, {\rm z}_b$. For an equivariant curve class $B$ (see Subsection \ref{subsection25}), a collection of strongly regular perturbations $\uds P = (P_1, \ldots, P_b)$, and a choice of asymptotic constrains at broad punctures ${\uds \kappa} = (\kappa_1, \ldots, \kappa_b)$, the moduli space ${\mc M}_{\uds P} ({\mc C}, B, {\uds \kappa})$ of perturbed gauged Witten equation over ${\mc C}$ admits an oriented virtual orbifold atlas. In particular, it has a virtual cardinality $\# {\mc M}_{\uds P}({\mc C}, B, {\uds \kappa})\in {\mb Q}$. 
\end{thm}

Here a perturbation is strongly regular if for each broad puncture ${\rm z}_i \in \Sigma$, a perturbed function $\tilde{W}_i$ is a holomorphic Morse function and all its critical values have distinct imaginary parts. These type of perturbations form a subset of the space ${\bf V}_i$ which is a complement of a real analytic hypersurface ${\bf V}_i^{\rm wall}\subset {\bf V}_i$. Theorem \ref{thm11} relies crucially on the results about compactness and Fredholm property proved in \cite{Tian_Xu}. 

The correlation function is supposed to be a multilinear function on the state space associated to the GLSM space (see Section \ref{section3}). As in \cite{Tian_Xu_2}, we define the correlation function by taking certain linear combination of the virtual cardinality $\# {\mc M}_{\uds P}({\mc C}, B, {\uds \kappa})$. Since the virtual cardinality does depend on the choice of strongly regular perturbations, it is a nontrivial procedure to prove the following invariance property of the correlation function (the same as Theorem \ref{thm41}).
\begin{thm}\label{thm12}
The correlation function defined by \eqref{eqn35} and linear extension is independent of the choice of strongly regular perturbations at broad punctures.
\end{thm}

The main idea of proving Theorem \ref{thm12} is to connect two strongly regular perturbations at a broad puncture ${\rm z}_i$ via a generic path in ${\bf V}_i$, where the path may cross the wall ${\bf V}_i^{\rm wall}$ at isolated places, and to construct a virtual chain on the  universal moduli space over this path. When crossing the wall, the so-called BPS soliton solutions form codimension one boundary of the universal moduli space. The contribution of the BPS soliton solutions gives a Picard--Lefschetz type wall-crossing formula (Theorem \ref{thm46}), which is a crucial step in deriving the invariance of the correlation function. The idea and the method are inspired by a similar argument in the Landau--Ginzburg A-model theory (see \cite{FJR3}) and the work of the second named author and S. Schecter in finite-dimensional Morse theory \cite{Lagrange_multiplier}. The detailed proof of the wall-crossing formula, though, is deferred to \cite{Tian_Xu_4}.

Both Theorem \ref{thm11} and Theorem \ref{thm12} rely on the construction of corresponding virtual cycle/chain in moduli spaces. Our construction uses a version of the virtual technique which originated in \cite{Li_Tian}. The method from \cite{Li_Tian} is topological, in the sense that the charts are only topological and the smoothness of coordinate changes is not needed. This differs from other methods, such as, the Kuranishi method or polyfold method, for which certain weak smoothness of coordinate changes has to be established. Usually, it is rather technical to establish the smoothness property required in those methods. On the other hand, in a recent work \cite{Pardon_virtual}, J. Pardon introduced a virtual technique which does not require smoothness of coordinate changes. In contrast, we get the virtual cycles by constructing perturbed sections in a more classical and topological way, while Pardon has his virtual cycles constructed via homological algebra operations.

This paper is organized as follows. In Section \ref{section2} we recall the basic set-up in \cite{Tian_Xu} of the gauged Witten equation and perturbations. In Section \ref{section3} we recall the linear Fredholm theory and then construct the virtual cycle in the case of strongly regular perturbations. This completes the definition of the correlation function. Sections \ref{section4}---\ref{section6} discuss the proof of the invariance of the correlation function: in Section \ref{section4} we consider variations of strongly regular perturbations and state the wall-crossing formula. In Section \ref{section6} we discuss the orientation issue and (re)prove the Picard--Lefschetz formula in our context. The key wall-crossing formula is proved in the companion paper \cite{Tian_Xu_4}. 

During the preparation of this paper, there appeared \cite{CLLL_15} and \cite{FJR_GLSM} which use algebraic methods to construct virtual cycles for GLSM in the absence of broad punctures. During the preparation and revision of this paper, we developed a different setting in which we can construct a cohomological field theory for much more general GLSM spaces in the so-called {\it geometric phases} (see \cite{Tian_Xu_2017, Tian_Xu_geometric}). Recently there also appeared \cite{CFGKS} which uses categorical construction of cohomological field theories for certain GLSM spaces. 

\subsection*{Acknowledgements}

The second named author would like to thank Mohammad Tehrani for many helpful discussions, Huai-Liang Chang for discussions about the algebraic geometry of GLSM, Mauricio Romo and Wei Gu for discussions about the physics of GLSM.

\section{Gauged Witten Equations and Perturbations}\label{section2}

In this section we recall the basic setup of gauged Witten equation given in \cite{Tian_Xu}. We apologize for many changes of notations  since we think the new notations are more convenient to use in this paper.

\subsection{The target space}\label{subsection21}

Let $X$ be a K\"ahler manifold and $Q: X \to {\mb C}$ be a holomorphic function, with a single critical point $\star \in X$. Suppose there is a holomorphic ${\mb C}^*$-action on $X$ such that $Q$ is homogeneous of degree $r$ ($r\geq 2$), meaning that $Q (g x) = g^r Q(x)$ for $x \in X$ and $g \in {\mb C}^*$. We assume that the ${\mb C}^*$-action restricts to a Hamiltonian $S^1$-action, and let $\mu': X \to \i {\mb R}$ denote a moment map. The pair $(X, Q)$ is called an {\it LG (Landau--Ginzburg) space} and the ${\mb C}^*$-action is referred to as the {\it R-symmetry}. 

Denote the group $S^1$ for the R-symmetry by $G'$. We call $(\tilde{X}, W, G)$ the {\it GLSM (gauged linear $\sigma$-model) space} whose elements are given as follows.
\begin{enumerate}
\item $\tilde{X} = {\mb C} \times X$. Its coordinates are denoted by $\tilde{x}= (p, x)$ and the factor ${\mb C}$ is equipped with the standard K\"ahler structure.

\item $G = G' \times G'' \simeq G' \times S^1$ acts on $\tilde X$ where the $G'$-factor acts on the $x$ coordinate and the $G'' \simeq S^1$-factor acts by $g'' (p, x) = ( (g'')^{-r} p, g'' x)$. The $G$-action is Hamiltonian with a moment map
\beq\label{eqn21}
\mu (p, x) = ( \mu', \mu''):=\Big( \mu' (x),\ \mu' (x) + \frac{\i r}{2} |p|^2 - \tau \Big).
\eeq
Here $\tau \in \i {\mb R}$ is a constant, which we fix from now on. 

\item $W: \tilde{X} \to {\mb C}$ is defined as $W( p, x) = p Q(x)$, which is $(G'')^{\mb C}$-invariant where $(G'')^{\mb C}$ is the complexification of $G''$.
\end{enumerate}

We make the same assumptions on the LG space as listed in \cite[Hypothesis 2.2, 2.4]{Tian_Xu} in order to guarantee the compactness of the moduli space. We do not need these assumptions explicitly in this paper, except for the following properness property.
\begin{hyp}\label{hyp21}
The functions $|\nabla Q|$ and $\mu'$ are both proper on $X$. 
\end{hyp}

We identify the Lie algebras of $G'$ and $G''$ with $\i {\mb R}$ and denote a vector of ${\mf g} = {\bf Lie} G$ as $\xi = ( \xi', \xi'') \in {\mf g}$. Denote by ${\mc X}_\xi \in \Gamma (TX)$ the infinitesimal action of $\xi$. To define the vortex equation, one also needs to specify a metric on ${\mf g}$. We do not simply take the product metric on ${\mf g} = {\mf g}' \times {\mf g}''$ but use the norm defined by
\beq\label{eqn22}
\big| (\xi', \xi'') \big|^2 = \lambda^{-1} \big| \xi' \big|^2 + \big| \xi'' \big|^2
\eeq
for a sufficiently large $\lambda$. The reason is explained in \cite{Tian_Xu}.

\subsection{Rigidified $r$-spin curves}

We recall the notion of $r$-spin curves. Since we will only consider a fixed domain curve, it is more convenient to rephrase the definitions without using the language of orbifolds.

\begin{defn}\label{defn22} Fix $r \in {\mb Z}$, $r \geq 3$. An {\it $r$-spin curve} with $k$ markings is a tuple $(\Sigma, {\bf z}, {\bf m}, L', \phi')$ where $\Sigma$ is a compact Riemann surface, ${\bf z} = \{ {\rm z}_1, \ldots, {\rm z}_k\}$ is a set of marked points, ${\bf m}$ is a collection of integers $m_1, \ldots, m_k \in \{0, 1, \ldots, r-1\}$, $L' \to \Sigma$ is a holomorphic line bundle, and $\phi'$ is a line bundle isomorphism
\beq\label{eqn23}
\phi': (L')^{\otimes r} \simeq K_\Sigma \otimes {\mc O}( (1-m_1) {\rm z}_1) \otimes \cdots \otimes {\mc O}((1-m_k) {\rm z}_k).
\eeq
\end{defn}

Let ${\mc C} = (\Sigma, {\bf z}, {\bf m}, L', \phi')$ be an $r$-spin curve. Denote $\Sigma^* = \Sigma \setminus {\bf z}$. It follows from the definition that if $z$ is a local coordinate on $\Sigma^*$, then there exists a local holomorphic frame $e'$ of $L'$, unique up to ${\mb Z}_r$-action, such that $\phi'((e')^{\otimes r}) = dz$. Moreover, if $w$ is a local coordinate centered at ${\rm z}_i\in {\bf z}$, then there exists local frames $e_i' $ of $L'$ near ${\rm z}_i$, unique up to ${\mb Z}_r$-action, such that 
\beq\label{eqn24}
\phi'( (e_i')^{\otimes r})  = w^{m_i} \frac{dw }{ w}.
\eeq
We call $\gamma_i = \exp ( 2 m_i \pi \i /r)$ the {\it monodromy} of the $r$-spin structure at ${\rm z}_i$.

Now we introduce the important notions of broadness/narrowness. Notice that since $Q$ is homogeneous with respect to the $G'$-action, for any $\gamma \in G'$, the critical point $\star$ is contained in the fixed point set $X_\gamma \subset X$ of $\gamma:X \to X$. 

\begin{defn}
Given $\gamma \in {\mb Z}_r \subset G'$, let $X_\gamma$ be the fixed point set of $\gamma$ in the LG space. $\gamma$ is called {\it broad} (resp. {\it narrow}) if $X_\gamma \neq \{\star\}$ (resp. $X_\gamma = \{\star\}$). A marking ${\rm z}_i$ of an $r$-spin curve is broad/narrow if the monodromy $\gamma_i$ is broad/narrow.
\end{defn}

\begin{defn}
Given an $r$-spin curve ${\mc C} = (\Sigma, {\bf z}, {\bf m}, L', \phi')$. A {\it rigidification} of ${\mc C}$ at a broad marking ${\rm z}_i$ is a choice of local frame $e_i$ of $L$ near ${\rm z}_i$ satisfying \eqref{eqn24}. An $r$-spin curve with rigidifications at all its broad markings is called a {\it rigidified $r$-spin curve}.
\end{defn}

The purpose of rigidifying the $r$-spin curve is the same as in \cite{FJR3}, because to perturb the equation we have to break the symmetry $Q$ or $W$ has. From now on we fix an $r$-spin curve $(\Sigma, {\bf z}, {\bf m}, L', \phi')$ and a rigidification. 

\subsection{Perturbations}

Now we state the assumption that allow us to perturb the gauged Witten equation. 

\begin{hyp}\label{hyp25}
For each broad $\gamma\in {\mb Z}_r$, there exists a nonzero, finite dimensional complex vector space ${\bf V}_\gamma$ parametrizing certain $\gamma$-invariant holomorphic functions $F_\gamma : X \to {\mb C}$ satisfying \cite[Hypothesis 2.8]{Tian_Xu}, particularly the following conditions. 
\begin{enumerate}

\item Each $F_\gamma\in {\bf V}_\gamma$ is the sum of homogeneous functions and each summand has uniformly bounded derivative on $X$.

\item For each $a_\gamma \in {\mb C}^*$, there is a complex analytic subset ${\bf V}_{\gamma, a_\gamma}^{\rm sing} \subset {\bf V}_\gamma$ such that for each $F_\gamma \in {\bf V}_\gamma \setminus {\bf V}_{\gamma, a_\gamma}^{\rm sing} $, the restriction of the function
\beqn
\tilde{W}_\gamma (x, p) = W(x, p) - a_\gamma p + F_\gamma (x)
\eeqn	
to $\tilde{X}_\gamma	$ is a holomorphic Morse function with finitely many critical points.

\item For each $a_\gamma \in {\mb C}^*$ and $F_\gamma \in {\bf V}_{\gamma, a_\gamma}^{\rm sing}$, $|\nabla \tilde W_\gamma|$ is bounded away from zero outside a compact subset. 

\item The set ${\bf V}_{\gamma, a_\gamma}^{\rm wall} \subset {\bf V}_\gamma \setminus {\bf V}_{\gamma, a_\gamma}^{\rm sing}$ defined by the coincidence of the imaginary parts of two critical values of $\tilde{W}_\gamma |_{\tilde{X}_\gamma}$ is a locally finite union of real analytic hypersurfaces.
\end{enumerate}
\end{hyp}

For $j= 1, \ldots, b$, denote ${\bf V}_i = {\bf V}_{\gamma_i}$ and $\tilde{\bf V}_i = {\mb C}^* \times {\bf V}_i$. Denote
\beqn
\tilde{\bf V}_i^{\rm sing} = \big\{ (a_i, F_i)\ |\  F_i \in {\bf V}_{\gamma_i, a_i}^{\rm sing} \big\},\ \tilde{\bf V}_i^{\rm wall} = \big\{ (a_i, F_i)\ |\ F_i \in {\bf V}_{\gamma_i, a_i}^{\rm wall} \big\}.
\eeqn

\begin{defn}\label{defn26}
$P_i = (a_i, F_i)\in  \tilde{\bf V}_i$ is called a {\it perturbation} of $W$ at the $j$-th broad puncture. It is called {\it regular} (resp. {\it strongly regular}) if $P_i \notin \tilde{\bf V}_i^{\rm sing}$ (resp. $P_i \notin \tilde{\bf V}^{\rm wall}_i \cup \tilde {\bf V}^{\rm sing}$). It is called {\it small} if it is close to the origin of $\tilde{\bf V}_i$. A perturbation to the gauged Witten equation over ${\mc C}$ is a collection
\beqn
\uds P = \big( P_i \big)_{j=1}^b  = \big( a_i, F_i \big)_{j=1}^b
\eeqn
where $P_i$ is a perturbation at the $j$-th broad puncture. It is called regular (resp. strongly regular) if each $P_i$ is regular (resp. strongly regular).
\end{defn}

Let $P_i = (a_i, F_i)$ be a perturbation at ${\rm z}_i$. We write
\begin{align*}
&\ \tilde{W}_i = W - a_i p + F_i,\ &\ \tilde{W}_{\tilde{X}_i}:= \tilde{W}_i|_{\tilde{X}_i}.
\end{align*}
We introduce the following notations. For ${\mz t} \in (G')^{\mb C} = {\mb C}^*$,\footnote{We were using the symbol $\delta$ instead of ${\mz t}$ in the previous papers \cite{Tian_Xu, Tian_Xu_2}.} we denote 
\beq\label{eqn25}
\tilde{W}_i^{\mz t} (x, p) = {\mz t}^r \tilde{W}_i ( {\mz t}^{-1} x, p) = W - {\mz t}^r a_i p + {\mz t}^r F_i( {\mz t}^{-1} x).
\eeq
Then we have
\beq\label{eqn26}
(x, p)\in {\bf Crit} \big[ \tilde{W}_i|_{{\tilde{X}_i}} \big] \Longleftrightarrow  ({\mz t} x, p) \in {\bf Crit} \big[  \tilde{W}_i^{\mz t}|_{{\tilde{X}_i}} \big];\ \ \  \tilde{W}_i^{{\mz t}}({\mz t} x, p) = {\mz t}^r \tilde{W}_i(x, p).
\eeq

\begin{rem}
The purpose of using perturbation is the same as in finite dimensional Morse theory of functions with degenerate singularities. There are alternative ways to develop the mathematical GLSM without perturbation. One possibility is to adapt the approach of Jake Solomon \cite{Solomon_preprint}, by treating the degenerate function $W$ directly. Another possibility is similar to the approach of Venugopalan \cite{Venugopalan_quasi} of defining quasimap invairants, when we are in the {\it geometric phase} (see \cite{Tian_Xu_2017} and \cite{Tian_Xu_geometric}). 
\end{rem}

\subsection{Gauged Witten equation}\label{subsection24}

The gauged Witten equation is the combination of the Witten equation and the symplectic vortex equation, where the latter requires an area form on the Riemann surface. Choose from now on a smooth area form on the closed surface $\Sigma$, which determines a Hodge star operator
\beqn
*: \Omega^2(\Sigma) \to \Omega^0(\Sigma).
\eeqn

On the other hand, we use a cylindrical metric to define Sobolev norms. For each ${\rm z}_i$, we have fixed a holomorphic coordinate $w$ centered at ${\rm z}_i$. We identify a cylindrical end $C_i$ with the cylinder $[0, +\infty) \times S^1$. The latter has coordinates $s + \i t = - \log w$. Choose an area form on $\Sigma^*$ whose restriction to $C_i$ coincides with $ds \wedge dt$. It induces a cylindrical metric on $\Sigma^*$. For $\tau>0$, $p>0$ and $k \geq 0$, for any open subset $U \subset \Sigma^*$, let $W_\tau^{k, p}(U)$ denote the weighted Sobolev space where the norm is defined as $\| f \|_{W_\tau^{k, p}(U)} := \| h_s^{- \tau /2} f \|_{W^{k, p}(U)}$. Here the latter Sobolev norm is taken with respect to the cylindrical metric on $U \subset \Sigma^*$. 

From now on we fix $p>2$.

The last auxiliary data we need to specify is a Hermitian metric on $L'$. Fix a reference Hermitian metrics $H'$ on $L'|_{\Sigma^*}$, which is of class $W^{1, p}_{\rm loc}$, such that for each puncture ${\rm z}_i$ with $m_i$ and $e_i'$ given in \eqref{eqn24}, for some $\tau > 0$, 
\beqn
\log ( |w|^{- {m_i/r}} \| e_i' \|_{H'} ) \in W_\tau^{2, p}(C_i).
\eeqn
Let $P' \to \Sigma^*$ be the unit circle bundle of $(L', H')$. On the other hand, we choose a smooth $G''$-bundle $P'' \to \Sigma$. Define 
\beqn
P:= P' \underset{\ \Sigma^*}{\times} P'' \to \Sigma^*.
\eeqn
This is a principal bundle with structure group $G = G' \times G''$. Since $G$ acts on $\tilde X$, we take the associated bundle over $\Sigma^*$, denoted by $Y$. 

\subsubsection{The variables}

The variables of the gauged Witten equation include connections on $P$, sections of $Y$, and a finite degrees of freedom called ``frames'' at broad markings. These variables transform via gauge transformations. 

Let us first specify the group of gauge transformations. Notice that $G$ is abelian. For $\tau>0$, let ${\mc G}_\tau'$ (resp. ${\mc G}_\tau''$) be the space of maps $g': \Sigma^* \to G'$ (resp. $g'': \Sigma^* \to G''$) of class $W^{2, p}_{\rm loc}$ such that the restrictions to each cylindrical end $C_i$ can be written as $g' = \exp (\i h')$ (resp. $g'' = \exp (\i h'')$) with $h'\in W_\tau^{2, p}$ (resp. $h'' \in W_\tau^{2, p}$). Define ${\mc G}_\tau = {\mc G}_\tau' \times{\mc G}_\tau''$. Then ${\mc G}_\tau$, ${\mc G}_\tau'$, ${\mc G}_\tau''$ are Banach Lie groups. Let ${\mc G}$ (resp. ${\mc G}'$, ${\mc G}''$) be the union of ${\mc G}_\tau$ (resp. ${\mc G}_\tau'$, ${\mc G}_\tau''$) for all $\tau>0$ which is also a group.

Now we discuss the class of connections. First, let $A_0'$ be the Chern connection on $P'$ associated to the metric $H'$. Let ${\mc A}_\tau'$ be the space of connections on $P'$ that can be obtained from $A_0'$ by a complex gauge transformation of regularity $W_\tau^{2,p}$, and ${\mc A}'$ be the union of ${\mc A}_\tau'$ for all $\tau>0$. On the other hand, fix from now on trivializations of $P''|_{C_i}$. Let ${\mc A}_\tau''$ be the space of connections on $P''$ of regularity $W^{1, p}_{\rm loc}$ whose restrictions to each $C_i$ is of the form $d + \alpha''$ with $\alpha''$ a 1-form of regularity $W_\tau^{1, p}$, and ${\mc A}''$ be the union of ${\mc A}_\tau''$ for all $\tau>0$. Denote ${\mc A}_\tau = {\mc A}_\tau' \times {\mc A}_\tau''$ and ${\mc A} = {\mc A}' \times {\mc A}''$. Then ${\mc G}_\tau$ acts on ${\mc A}_\tau$ and ${\mc G}$ acts on ${\mc A}$ via usual gauge transformations. 

Another variable of the gauge Witten equation is a section of $Y$, which is usually denoted by $u$. We allow arbitrarily weak regularity of $u$ at this moment. Later we will consider more specific asymptotic constrains on $u$. 

The last set of variables are frames of $L''$ at markings. Let ${\bf Fr}_i:= \big\{ \phi_i: G'' \to P_{{\rm z}_i}'' \big\}$ be the set of trivializations of $P_{{\rm z}_i}''$, which is an $S^1$-torsor. There is a left $G''$-action $e^{\i \theta} \cdot \phi_i = \phi_i \circ e^{\i \theta}$. Moreover, using the trivializations of $P''|_{C_i}$ fixed from the last paragraph, we extend each $\phi_i \in {\bf Fr}_i$ to a trivialization $\phi_i^\bullet: G'' \times C_i \to P''|_{C_i}$ such that $(e^{\i \theta} \phi_i)^\bullet = \phi_i^\bullet \circ e^{\i \theta}$. From now on we identify $\phi_i\in {\bf Fr}_i$ with $\phi_i^\bullet$. Denote $\uds{\bf Fr} = \prod_{j=1}^b {\bf Fr}_i$ whose elements are denoted by $\uds \phi = (\phi_1, \ldots, \phi_b)$. 

The variables of the gauged Witten equation are ${\bm u} = (A, u, \uds\phi)$ as described above. 

\begin{notation}\label{notation28}
We will often need to consider the expression of a section $u \in \Gamma(Y)$ under local trivializations. Consider the trivialization of $P'|_{C_i}$ induced from the unitary frame $e_i'/ \| e_i'\|_{H'}$ where $e_i'$ is the one in \eqref{eqn24} (which is unique by the rigidification). Together with $\phi_i \in {\bf Fr}_i$ it determines a trivialization of $P|_{C_i}$ and hence $Y|_{C_i}$, which, by abuse of notation, is still denoted by $\phi_i^\bullet$. Then for $u \in \Gamma(Y)$, we write $u_i^\bullet: C_i \to \tilde X$ the map induced from $u$ and $\phi_i^{\bullet}$.
\end{notation}

\subsubsection{The lift of the superpotential}

We describe how to lift the superpotential $W$ to $Y$. Given $A = (A', A'') \in {\mc A}$. By the isomorphism $\phi'$ in the $r$-spin structure (see \ref{eqn23}), around each point $q \in \Sigma^*$, there exists local coordinate $z_q$ and local frame $e_q'$ of $L'$, which is holomorphic with respect to $A'$, such that $\phi' ( (e_q')^{\otimes r}) = dz_q$. 
Let $e_q''$ be an arbitrary local frame of $P''$. Define
\beqn
{\mc W}_A ( [e_q', e_q'', x]) = W(x) dz_q.
\eeqn
This lift only depends on $A'$ but not $A''$. 

The lift of perturbed superpotential does depend on both components of the connection and the frames. Upon choosing $\uds P$, denote the perturbation of ${\mc W}_A$ by $\tilde{\mc W}_{A, \uds\phi}$. Instead of stating the full definition, for the purpose of the current paper, we only list a few properties of $\tilde{\mc W}_{A, \uds\phi}$. The details can be found in \cite{Tian_Xu}.
\begin{enumerate}

\item $\tilde{\mc W}_{A, \uds\phi} \in \Gamma(Y, \pi^* K_{\Sigma^*})$ depends on both $A\in {\mc A}$ and $\uds\phi$ smoothly. 

\item Given $g: \Sigma \to G$ with $g({\rm z}_i) = {\bf Id}$, viewed as a gauge transformation on $P$, one has
\beqn
\tilde{\mc W}_{g^* A, \uds\phi} ( g^{-1} x) = \tilde{\mc W}_{A, \uds\phi} (x).
\eeqn

\item Outside the cylindrical ends around broad markings, $\tilde{\mc W}_{A, \uds\phi} = {\mc W}_A$. 

\item There is a number ${\mz t}_A$ depending smoothly on $A \in {\mc A}$, and, on a cylindrical end around a broad marking ${\rm z}_i$, there exist holomorphic frames $e_i'$, $e_i''$ of $L'$, $L''$, such that $\phi' ((e_i')^{\otimes r}) = w^{m_i} dw/w$ and 
\beqn
\tilde{\mc W}_{A, \uds\phi} ([ e_i', e_i'', z, x]) = \tilde{W}_i^{{\mz t}_A} \circ w^{\frac{m_i}{r}}.
\eeqn
(Recall the notation $W_i^{\mz t}$ from \eqref{eqn25}.) Moreover, if $\uds g = (g_1'', \ldots, g_b'')$ where $g_i'' \in G''$, then one has 
\beqn
\tilde{\mc W}_{A, \uds g \cdot \uds \phi}|_{C_i} = \tilde{\mc W}_{A, \uds\phi}|_{C_i} \circ g_i''.
\eeqn
\end{enumerate}

\subsubsection{The gauged Witten equation}

For the fibre bundle $Y \to \Sigma^*$, the K\"ahler structure on $\tilde{X}$ induces a Hermitian structure on the vertical tangent bundle $T^\bot Y$. On the other hand, for any continuous connection $A$, the tangent bundle $TY$ splits as the direct sum of $T^\bot Y$ and the horizontal tangent bundle. The latter is isomorphic to $\pi^* T\Sigma^*$, therefore the connection induces an integrable almost complex structure on $Y$. The vertical differential of $\tilde{\mc W}_{A, \uds\phi}$ is a section
\beqn
d \tilde{\mc W}_{A, \uds\phi} \in \Gamma ( Y, \pi^* K_{\Sigma^*} \otimes ( T^\bot Y )^* ).
\eeqn
The Hermitian metric on $T^\bot Y$ induces a conjugate linear isomorphism $T^\bot Y \simeq ( T^\bot Y )^*$; the complex structure on $\Sigma^*$ also induces a conjugate linear isomorphism $K_{\Sigma^*} \simeq \Lambda^{0,1}_{\Sigma^*}$. Therefore we have a conjugate linear isomorphism
\beqn
\pi^* K_{\Sigma^*} \otimes ( T^\bot Y )^* \simeq \pi^* \Lambda^{0,1}_{\Sigma^*} \otimes T^\bot Y.
\eeqn
The image of $d \tilde{\mc W}_{A, \uds\phi}$ under this map is denoted by $\nabla \tilde{\mc W}_{A, \uds\phi} \in \Gamma ( Y, \pi^* \Lambda^{0,1}_{\Sigma^*} \otimes T^\bot Y )$. The $\uds P$-perturbed gauged Witten equation reads
\beq\label{eqn27}
\left\{ \begin{array}{ccc}
\ov\partial_A u + \nabla \tilde{\mc W}_{A, \uds\phi} (u) & = & 0;\\
* F_A + \mu^* (u) & = & 0.
\end{array}\right.
\eeq
Each term in the system is defined as follows: the connection $A$ induces a continuous splitting $TY \simeq T^\bot Y \oplus \pi^* T \Sigma^*$ and $d_A u \in W^{1, p}_{\rm loc}(T^* \Sigma^* \otimes u^* T^\bot Y )$ is the covariant derivative of $u$; the $G$-invariant complex structure $J$ induces a complex structure on $T^\bot Y$ and $\ov\partial_A u$ is the $(0, 1)$-part of $d_A u$ with respect to this complex structure. $\nabla \tilde{\mc W}_{A, \uds\phi} (u)$ is the pull-back of $\nabla \tilde{\mc W}_{A, \uds\phi}$ by $u$, which lies in the same vector space as $\ov\partial_A u$. $F_A\in \Omega^2(\Sigma^*) \otimes {\mf g}$ is the curvature form of $A$, $*: \Omega^2(\Sigma^*) \to \Omega^0(\Sigma^*)$ is the Hodge-star operator with respect to the smooth metric on $\Sigma$; the moment map $\mu$ lifts to a ${\mf g}$-valued function on $Y$ and $\mu^*(u)$ is the dual of $\mu(u)$ with respect to the metric defined by \eqref{eqn22}.

\subsection{Asymptotic behavior and compactness}\label{subsection25}

In \cite{Tian_Xu} we proved that if ${\bm u} = (A, u, {\uds \phi})$ is a finite energy solution to \eqref{eqn27} with bounded image (which we called a {\it bounded} solution), then $u$ has good asymptotic behavior at punctures. More precisely, let $(s, t)$ be the cylindrical coordinates on $C_i$. If ${\rm z}_i$ is narrow, i.e., the fixed point set of the monodromy acting on $X$ is isolated, then there is $\tilde{p}_i = (\star, p_i) \in \tilde{X}$ s.t.	 
\beqn
\lim_{z \to {\rm z}_i} e^{\frac{\i m_i t}{r}} u_i^{\bullet} (s, t) = \tilde{p}_i; 
\eeqn
if ${\rm z}_i$ is broad, then there is a critical point $\kappa_i = (x_i, p_i) \in {\bf Crit} [ \tilde{W}_i|_{{\tilde{X}_i}} ]$ such that 
\beqn
\lim_{z \to {\rm z}_i} e^{ \frac{\i m_i t}{r}} u_i^{\bullet} (s, t) = \kappa_i^{{\mz t}_A}:= ({\mz t}_A x_i, p_i).
\eeqn
Moreover, the rate of convergence is exponential, i.e., can be controlled by $e^{-\tau_0 s}$ for some $\tau_0>0$. A consequence of this result is that any bounded solution defines a homology class $B \in H^G_2(\tilde{X}; {\mb Q})$. We call a class in $H_2^G(\tilde{X}; {\mb Q})$ an {\it equivariant curve class}\footnote{When $X = {\mb C}^N$, an equivariant curve class is essentially the degrees of $L'$ and $L''$.}.

The behavior of the perturbed gauged Witten equation near broad punctures calls for the study of solitons, i.e., solutions over the infinite cylinder. Let $\gamma_i = \exp( \frac{\i 2\pi m_i}{r} )$ be the monodromy at ${\rm z}_i$ and $\tilde{W}_i$ be the perturbed superpotential. For ${\mz t} \in (0, 1)$, consider the following equation for maps $\sigma_i: {\mb R} \times S^1 \to \tilde{X}$.
\beq\label{eqn28}
\frac{\partial \sigma_i }{\partial s} + J \Big[ \frac{\partial \sigma_i}{\partial t} + {\mc X}_{\i m_i /r} (\sigma_i) \Big] + \nabla \tilde{W}_i^{{\mz t}}(e^{\frac{\i m_i t}{r}} \sigma_i) = 0.
\eeq
Since $\tilde{W}_i$ is $\gamma_i$-invariant, $\nabla \tilde{W}_i^{{\mz t}}( e^{m_i \frac{\i t}{r}}\sigma_i)$ is well-defined over ${\mb R} \times S^1$. A bounded solution to \eqref{eqn28} is called a {\it soliton}, for which we know there are $\kappa_{i, \pm} \in {\bf Crit} \big[ \tilde{W}_{\tilde{X}_i}^{{\mz t}}\big]$ such that
\beqn
\lim_{s \to \pm \infty} e^{\frac{\i m_i t}{r}} \sigma_i(s, t) = \kappa_{i, \pm}.
\eeqn
Then we have the energy identity
\beq\label{eqn29}
\Big\| \frac{\partial \sigma_i }{\partial s} + J \Big[ \frac{\partial \sigma_i}{\partial t} + {\mc X}_{\i m_i/r}  (\sigma_i) \Big] \Big\|_{L^2({\mb R} \times S^1)}^2 = 2\pi \left[ \tilde{W}_i^{{\mz t}} (\kappa_{i, -}) - \tilde{W}_i^{{\mz t}} (\kappa_{i, +}) \right].
\eeq
We have known from \cite{Tian_Xu} that nontrivial solitons exist only when $\tilde{W}_i$ is not strongly regular, i.e., $\tilde{W}_i|_{\tilde{X}_i}$ has two different critical values with identical imaginary part.

We call a soliton $\sigma_i$ a {\it BPS soliton} if $e^{ \i m_i t/r} \sigma_i(s, t)$ is independent of $t$ for all $s \in {\mb R}$; otherwise it is called a {\it non-BPS soliton}. If $\sigma_i$ is a BPS soliton, then it is easy to see that $e^{\i m_i t/r} \sigma_i(s, t) \in \tilde{X}_i$ and is a negative gradient flow line of the real part of $\tilde{W}_i^{\mz t}|_{\tilde{X}_i}$.

Fixing an equivariant curve class $B$, a regular perturbation $\uds P$, and a collection of critical points $\uds \kappa = (\kappa_1, \ldots, \kappa_b)$. Let $\hat {\mc M}_{\uds P}^* ({\mc C}, B, \uds \kappa)$ be the the set of solutions to the $\uds P$-perturbed gauged Witten equation over ${\mc C}$ that represent $B$ and have asymptotic limits at broad markings labelled by $\uds \kappa$. By the gauge symmetry of the perturbed gauged Witten equation, one can define the moduli space
\beqn
{\mc M}_{\uds P}^* ({\mc C}, B, \uds\kappa) = \hat {\mc M}_{\uds P}^* ({\mc C}, B, \uds \kappa) / {\mc G}.
\eeqn

By the compactness result of \cite{Tian_Xu}, ${\mc M}_{\uds P}^* ({\mc C}, B, \uds\kappa)$ can be compactified by adding soliton solutions. Let ${\mc M}_{\uds P}({\mc C}, B, \uds\kappa)$ be the compactified moduli space. Moreover, the notion of sequential convergence defined in \cite{Tian_Xu} induces a topology on ${\mc M}_{\uds P}({\mc C}, B, \uds\kappa)$, in the same way as Gromov--Witten theory (see \cite[Section 5.6]{McDuff_Salamon_2004}) which is Hausdorff and compact. In particular, when $\uds P$ is strongly regular, no soliton solutions need to be added and ${\mc M}_{\uds P}^* ({\mc C}, B, \uds\kappa) = {\mc M}_{\uds P}({\mc C}, B, \uds \kappa)$.

\section{Definition of the Correlation Function}\label{section3}

The definition of the GLSM correlation function given in \cite[Section 3]{Tian_Xu_2} relies on the following theorem. For the precise meanings of virtual orbifold atlas and the virtual cardinality, see the section7.

\begin{thm}\label{thm31}
Given a smooth $r$-spin curve ${\mc C}$, an equivariant curve class $B$, a strongly regular perturbation $\uds P = (P_1, \ldots, P_b)$, and a choice of asymptotic constrains at broad punctures ${\uds \kappa}$, the moduli space ${\mc M}_{\uds P} ({\mc C}, B, {\uds \kappa})$  admits an oriented virtual orbifold atlas. In particular, when the virtual dimension is zero, there is a well-defined virtual cardinality
\beqn
\# {\mc M}_{\uds P} ({\mc C}, B, {\uds \kappa}) \in {\mb Q}.
\eeqn
Moreover, when ${\mc C}$ has at least one broad marking, one has
\beqn
\# {\mc M}_{\uds P}({\mc C}, B, {\uds \kappa}) \in {\mb Z}.
\eeqn
\end{thm}

Theorem \ref{thm31} will be proved momentarily, in Subsections \ref{subsection32}--\ref{subsection34}. Before doing that we first recall the definition of the correlation function.

\subsection{Definition of the correlation function}

\subsubsection{The state space}

\begin{defn} Let $\gamma$ be an element of ${\mb Z}_r$.
\begin{enumerate}

\item If $\gamma$ is narrow, then the $\gamma$-sector of the state space ${\mc H}_\gamma$ is a 1-dimensional ${\mb Q}$-vector space, generated by an element $e_\gamma$.

\item If $\gamma$ is broad, then the $\gamma$-sector of the state space is
\beqn
{\mc H}_\gamma:= H^{\rm mid} \big( X_{Q_\gamma}^c/ {\mb C}^*; {\mb Q} \big).
\eeqn
Here $Q_\gamma$ is the restriction of $Q$ to $X_\gamma$ and $X_{Q_\gamma}^c: = X_\gamma \setminus Q_\gamma^{-1}(0)$. 

\item The total GLSM state space is
\beqn
{\mc H}_{\rm GLSM}:={\mc H}_Q:= \bigoplus_{\gamma\in {\mb Z}_r} {\mc H}_\gamma.
\eeqn
\end{enumerate}
\end{defn}

Suppose $\gamma$ is broad. Since $Q_\gamma:= Q|_{X_\gamma}: X \to {\mb C}$ is homogeneous of degree $r$, for any $a \in {\mb C}^*$, the inclusion $Q^a_\gamma:= Q_\gamma^{-1}(a) \hookrightarrow X_{Q_\gamma}^c$ induces a diffeomorphism
\beqn
Q_\gamma^a/ {\mb Z}_r \simeq X_{Q_\gamma}^c/ {\mb C}^*
\eeqn
as orbifolds. By the basic property of cohomology of orbifolds and equivariant cohomology, we know that
\beq\label{eqn31}
{\mc H}_\gamma \simeq H^{\rm mid} \big( Q_\gamma^a; {\mb Q} \big)^{{\mb Z}_r}.
\eeq
The monodromy action on the cohomology $H^*(Q_\gamma^a; {\mb Q} \big)$ induced from the locally trivial fibration $Q_\gamma: X_{Q_\gamma}^c \to {\mb C}^*$ is equivalent to operator on $H^*(Q_\gamma^a; {\mb Z})$ induced from the ${\mb Z}_r$-action on $Q_\gamma^a$. Therefore, ${\mc H}_\gamma$ is the monodromy invariant part of the middle dimensional rational cohomology of a regular fibre of $Q_\gamma$.

For each $\gamma\in {\mb Z}_r$, we assume that there exists a perfect pairing $\langle \cdot, \cdot \rangle_\gamma: {\mc H}_\gamma \otimes {\mc H}_{\gamma} \to {\mb Q}$ such that with respect to the obvious isomorphism $\varepsilon: {\mc H}_\gamma \to {\mc H}_{\gamma^{-1}}$, $\langle \cdot, \cdot \rangle_\gamma = \langle \varepsilon \cdot, \varepsilon \cdot \rangle_{\gamma^{-1}}$. In many situations one has a naturally defined perfect pairing. Indeed, we have the following ${\mb Z}_r$-equivariant exact sequence
\beqn
\xymatrix{ H^{*-1} \big( X_\gamma \big) \ar[r] & H^{*-1} \big( Q_\gamma^a \big) \ar[r] & H^* \big( X_\gamma, Q_\gamma^a \big) \ar[r] & H^* \big( X_\gamma \big)}.
\eeqn
If the arrow in the middle is an isomorphism in ${\mb Q}$-coefficients for $*= n_\gamma$ (it is the case when $X = {\mb C}^N$), then one can define a pairing between $H^{\rm mid} ( Q_\gamma^a; {\mb Q} )$ and $H^{\rm mid} ( Q_\gamma^{-a}; {\mb Q} )$ as in \cite[Page 36]{FJR2}. Nonetheless, the pairing defines an isomorphism
\beq\label{eqn32}
{\mc O} \mapsto {\mc O}^*: {\mc H}_\gamma \mapsto {\mc H}_\gamma^* \simeq H_{\rm mid} \big( Q_\gamma^a; {\mb Q} \big)^{{\mb Z}_r}.
\eeq

\begin{rem}
One can shift the degrees of states in ${\mc H}_\gamma$ by the convention of Chen--Ruan cohomology. However in this paper we do not do that explicitly; indeed, in this paper, we are skipping all discussions about the dimensions of the virtual cycles as well as the degrees of states. They can be easily recovered from the index formula proved in \cite{Tian_Xu}.
\end{rem}

\subsubsection{Vanishing cycles in $Q_\gamma^a$}

The use of Lagrange multiplier requires us to consider the complex Morse theory on $Q_\gamma^a$. If we have a holomorphic Morse function $F: Q_\gamma^a\to {\mb C}$, then the unstable or stable manifolds of its critical points (with respect to the negative gradient flow of the real part of $F|_{Q_\gamma^a}$) represent certain $\infty$-relative cycles. For any compact subset $Z \subset X$, consider the relative homology $H_* \big( Q_\gamma^a, Q_\gamma^a \setminus Z \big)$. The inverse limit with respect to the direct system of compact subsets under inclusion is denoted by $H_* ( Q_\gamma^a, \infty)$. This is the dual space of $H^*_c ( Q_\gamma^a )$. Then we have the intersection pairing
\beq\label{eqn33}
\cap: H_{\rm mid} \big( Q_\gamma^a \big) \otimes H_{\rm mid} \big( Q_\gamma^a, \infty \big) \to {\mb Z}.
\eeq

Now take a strongly regular perturbation $P_i = (a_i, F_i) \in \tilde{\bf V}_i$. Abbreviate $Q_{\gamma_i}$ by $Q_i$ and $Q_i^{-1}(a_i)$ by $Q_i^{a_i}$. Consider the negative gradient flow of ${\bf Re} F_i$ restricted to $Q_i^{a_i}$. For each critical point $\kappa_i$ of $F_i|_{Q_i^{a_i}}$, denote by
\beqn
\big[ W_{\kappa_i}^u \big] \in H_{\rm mid} \big( Q_i^{a_i}, F_i^{-\infty} ; {\mb Z} \big)\ \Big( {\rm resp.}\ \big[ W_{\kappa_i}^s \big] \in H_{\rm mid} \big( Q_i^{a_i}, F_i^{+\infty}; {\mb Z} \big) \Big)
\eeqn
the class of the unstable (resp. stable) manifold of this flow. Here
\beqn
F_i^{+\infty}= Q_i^{a_i} \cap ( {\bf Re} F_i )^{-1} ( [M, +\infty) ),\ F_i^{-\infty}= Q_i^{a_i} \cap ( {\bf Re} F_i )^{-1} ( (-\infty, -M])
\eeqn
for some $M>>0$. We still use $\big[ W_{\kappa_i}^{u/s} \big]$ to denote their images under the map
\beqn
H_{\rm mid} \big( Q_i^{a_i}, F_i^{\pm\infty}; {\mb Z} \big) \to H_{\rm mid} \big( Q_i^{a_i}, \infty; {\mb Z} \big).
\eeqn

\subsubsection{The correlation function}

The correlation function is a multilinear map
\beq\label{eqn34}
\big\langle \cdot \big\rangle_B: \bigotimes_{j=1}^k {\mc H}_{\gamma_i} \to {\mb Q}.
\eeq
To define \eqref{eqn34}, we choose the last $n$ inputs (narrow states) to be the generators of the corresponding sectors ${\mc O}_i = e_{\gamma_i} \in {\mc H}_{\gamma_i}$, $j = b+1, \ldots, b+n$. Suppose the first $b$ inputs (the broad states) are ${\mc O}_i \in {\mc H}_{\gamma_i}$, $j = 1, \ldots, b$. Then define
\beq\label{eqn35}
\big\langle {\mc O}_1, \ldots, {\mc O}_b, {\mc O}_{b+1}, \ldots, {\mc O}_{b + n} \big\rangle_{B, {\uds P}}: = \sum_{{\uds \kappa}} \# {\mc M}_{\uds P} ( {\mc C}, B, {\uds \kappa} ) \cdot \prod_{j=1}^b {\mc O}_i^* \cap \big[ W_{\kappa_i}^u \big].
\eeq
Here $\# {\mc M}_{\uds P}( {\mc C}, B, \uds\kappa) \in {\mb Q}$ is the virtual cardinality of Theorem \ref{thm31}, and ${\mc O}_i^* \in H_{\rm mid} ( Q_i^{a_i}; {\mb Q} )^{{\mb Z}_r}$ is the image of ${\mc O}_i$ under \eqref{eqn32} and the $\cap$ is the intersection \eqref{eqn33}. By linear extension, this provides a correlation function which {\it a priori} depends on the choice of the strongly regular perturbation $\uds P$. Eventually we will prove the independence of $\uds P$, thus \eqref{eqn34} is well-defined.

\subsection{Review the linear theory}\label{subsection32}

Now we start to rigorously define the virtual cardinality $\# {\mc M}_{\uds P}({\mc C}, B, \uds \kappa)$, or, more generally, a virtual cycle on ${\mc M}_{\uds P}({\mc C}, B, \uds\kappa)$. For the first step we review the linear Fredholm theory for the perturbed gauged Witten equation over ${\mc C}$. 

As before, ${\mc C} = (\Sigma, {\bf z}, {\bf m}, L', \phi')$ is a smooth $r$-spin curve with first $b$ markings broad and last $n$ markings narrow. Let $\tau>0$, $k \geq 0$. Let $E \to \Sigma^*$ be a vector bundle. Let $W_\tau^{k, p}(\Sigma^*, E)$ be the space of $W^{k, p}_{\rm loc}$ sections such that its restriction on each cylindrical end is of class $W^{k, p}_\tau$. Here the Sobolev norms is taken with respect to some fixed choice of connection on $E$. We drop $\Sigma^*$ from the notation in this section and abbreviate the space by $W_\tau^{k, p}(E)$. When $k=0$ we also denote it by $L_\tau^p(E)$. 

\begin{defn}\label{defn34} Given an equivariant curve class $B$, $\uds \kappa$, and a small $\tau>0$, we define
\beqn
{\mc B} := {\mc B}_\tau^{1, p} ( B, {\uds \kappa} ) \subset {\mc A}_\tau^{1, p} \times  W_{\rm loc}^{1, p} ( \Sigma^*, Y ) \times \uds{\bf Fr},
\eeqn
which consists of tuples ${\bm u} = (A, u, \uds\phi)$ satisfying the following conditions.
\begin{enumerate}
\item For each broad ${\rm z}_i$, $u$ has the asymptotic limit $\kappa_i$ at ${\rm z}_i$ (see Subsection \ref{subsection25}).

\item For each narrow ${\rm z}_i$, $u_i^{\bullet}$ is asymptotic to some $\tilde{p}_i  = (\star, p_i) \in \tilde{X}_i$.

\item The above two conditions implies that $u$ extends to an orbifold section over ${\mc C}$ and we require that the section represents the equivariant curve class $B$.
\end{enumerate}
The rate of convergence of $u_i^{\bullet}$ to its limit is in $W_\tau^{1, p}$.
\end{defn}

${\mc B}$ is a Banach manifold. A general point is denoted as ${\bm u} = (A, u, \uds \phi)$. Define a Banach vector bundle ${\mc E}\to {\mc B}$ whose fibre over ${\bm u}$ is 
\beqn
{\mc E}_{\bm u}^-:= L_\tau^p(\Sigma^*, \Lambda^{0,1} \otimes u^* T^\bot Y) \oplus L_\tau^p( \Sigma^*, {\mf g}).
\eeqn

\subsubsection{Deformation theory and gauge fixing}

Let ${\mc G}_\tau$ be the set of $W^{2, p}_{\rm loc}$ maps from $\Sigma^*$ to $G$ that are asymptotic to the identity at all punctures, and such that the converges at a marking is in $W^{2, p}_\tau$. Then ${\mc G}_\tau$ acts on ${\mc B}$ on the first two components $A$ and $u$ by gauge transformations. Given ${\bm u} \in {\mc B}$, the linearizations of the gauge transformation and the gauged Witten equation induce the deformation complex
\beqn
\xymatrix{{\bf Lie} {\mc G}  \ar[r] & T_{\bm u} {\mc B} \ar[r] & {\mc E}_{\bm u}^-.}
\eeqn
A usual way of considering problems with gauge symmetry is to take special slices of the action by the group of gauge transformations. In our case the gauge group is abelian, so one can transform all connections to a common slice. More precisely, choose from now on a smooth reference connection $A_0 \in {\mc A}_\tau^{1, p}$. Let $-k$ be the index of the operator $\Delta: W_\tau^{2, p}({\mf g}) \to L_\tau^p({\mf g})$. Choose $s_1, \ldots, s_k\in C_0^\infty( \Sigma^*, {\mf g}) \subset L_\tau^p(\Sigma, {\mf g})$ such that
\beqn
{\bf Image} \big( \Delta: W_\tau^{2, p}({\mf g}) \to L_\tau^p({\mf g}) \big)   + {\bf Span} \big\{ s_1, \ldots, s_k \big\} =  L_\tau^p({\mf g}).
\eeqn
Denote $\Lambda_{GF} = {\bf Span}\{s_1, \ldots, s_k\}$ and denote the following composition by $\ov{d^*}$.
\beqn
\xymatrix{ T{\mc A}_\tau^{1, p} \ar[r]^{d^*} & L_\tau^p({\mf g}) \ar[r] & \displaystyle \frac{L_\tau^p({\mf g})}{\Lambda_{GF}}}.
\eeqn
$A$ is said to be in Coulomb gauge (relative to $A_0$) if $\ov{d^*} (A - A_0) = 0$. From the definition it is easy to see that any connection in ${\mc A}_\tau^{1, p}$ can be transformed via a gauge transformation in ${\mc G}_\tau$ to a unique connection in Coulomb gauge relative to $A_0$.

Incorporating the gauge fixing condition, we define a section of a bundle over ${\mc B}$ which describes the deformation theory of the gauged Witten equation. Define a Banach space bundle ${\mc E} \to {\mc B}$ whose fibre over ${\bm u}  = (A, u, {\uds \phi}) \in {\mc B}$ is 
\beqn
{\mc E}_{\bm u} = L_\tau^p(\Lambda^{0,1} \otimes u^* T^\bot Y) \oplus L_\tau^p({\mf g}) \oplus \big( L_\tau^p( {\mf g}) /\Lambda_{GF} \big).
\eeqn
Denote the last summand by $\ov{L_\tau^p( {\mf g}) }$. Define ${\mc F}: {\mc B} \to {\mc E}$ by ${\mc F}({\bm u}) = ({\mc W}({\bm u}), {\mc V}({\bm u}))$ where
\beqn
{\mc W}({\bm u} ) = \ov\partial_A u + \nabla \tilde{\mc W}_{A, \uds\phi} (u)\in L_\tau^p(\Lambda^{0,1} \otimes u^* T^\bot Y),
\eeqn
\beqn
{\mc V}({\bm u} ) = \big( * F_A + \mu(u),\ \ov{d^*} (A - A_0) \big) \in L_\tau^p({\mf g}) \oplus \ov{ L_\tau^p({\mf g}) }.
\eeqn
The linearized operator at ${\bm u}$ is a bounded linear operator $D_{\bm u} {\mc F}: T_{\bm u} {\mc B} \to {\mc E}_{\bm u}$.

\begin{prop}\hfill
\begin{enumerate}

\item ${\mc B}$ is a Banach manifold, whose tangent space at ${\bm u} = (A, u, {\uds \phi})$ is isomorphic to
\beq\label{eqn36}
T_{\bm u} {\mc B}\simeq T {\mc A}_\tau^{1, p} \oplus W_\tau^{1, p} \big( u^* T^\bot Y \big) \oplus {\mb C}^n \oplus {\mb R}^b.
\eeq

\item ${\mc G}_\tau$ acts smoothly on ${\mc B}$ such that the isomorphism \eqref{eqn36} is equivariant in a natural way. When there is at least one puncture on ${\mc C}$, the ${\mc G}_\tau$-action is free.

\item For each curve class $B$, there exists $\tau = \tau_B > 0$ satisfying the following condition. For any bounded solution $(A, u, {\uds \phi})$ to the perturbed gauged Witten equation over ${\mc C}$ with curve class $B$, there exist a gauge transformation $g\in {\mc G}$ and ${\uds \kappa} = (\kappa_1, \ldots, \kappa_b)$ such that $g^* (A, u, {\uds \phi}) \in {\mc B}_\tau^{1, p} (B, {\uds \kappa})$.

\item ${\mc E} \to {\mc B}$ is a Banach space bundle and ${\mc F}: {\mc B} \to {\mc E}$ is a smooth section. For any ${\bm u} \in {\mc B}$, the linearized operator $D_{\bm u} {\mc F}:T_{\bm u} {\mc B} \to {\mc E}_{\bm u}$ is Fredholm.

\item The natural map ${\mc F}^{-1}(0) \to {\mc M}_{\uds P}^*( {\mc C}, B, \uds\kappa)$ is a homeomorphism. 
\end{enumerate}

\end{prop}

\begin{rem}
This proposition has been almost proved in \cite{Tian_Xu} so here we only explain what was missing there. For (i), the difference from the case of \cite{Tian_Xu} is that we shall vary the framing at broad punctures. This gives the additional ${\mb R}^b$-factor in \eqref{eqn36}. For (ii), we know that an automorphism of a connection is a constant gauge transformation. When there are punctures, ${\mc G}$ contains no constants other than the identity.
\end{rem}

\subsection{The orientation of $D_{\bm u} {\mc F}$}\label{subsection33}

(See Section \ref{section6} for related discussions.) By an orientation of a Fredholm operator $F: E \to E'$ we mean an orientation of the determinant line $\det F := \det {\bf Ker} F \otimes (\det {\bf Coker} F)^\vee$. An orientation of a continuous family of Fredholm operators $F_x: E_x \to E_x'$ for $x \in N$ is a continuous trivialization of the determinant line bundle $\det F \to N$. Hence if there is a homotopy of Fredholm operators $F_{x, t}: E_x \to E_x'$, $t \in [0, 1]$, then the orientability problem of the family $F_{x, 0}$ is equivalent to that of $F_{x, 1}$. 

In our case, consider the family of linearized operators $D_{\bm u} {\mc F}: T_{\bm u} {\mc B} \to {\mc E}_{\bm u}$ for ${\bm u} \in {\mc B}$. The first component of the gauged Witten equation also gives another family of operators $D_{\bm u} {\mc W}: W_\tau^{1, p}(u^* T^\bot Y ) \to L_\tau^p(\Lambda^{0, 1} \otimes u^* T^\bot Y )$.

\begin{lemma}
The family $\{D_{\bm u} {\mc F}\}_{{\bm u}\in {\mc B}}$ is oriented if and only the family $\{ D_{\bm u} {\mc W}\}_{{\bm u} \in {\mc B}}$ is oriented, and there is a canonical identification between their orientations.
\end{lemma}

\begin{proof}
The finite rank part of $D_{\bm u} {\mc F}$ coming from the last two summands of \eqref{eqn36} doesn't affect the orientability problem. So one only needs to consider following operator at $a = 1$.
\beqn
\left[ \begin{array}{c} \beta \\ v
\end{array}\right] \mapsto \left[ \begin{array}{c} D_{\bm u} {\mc W} (v) +  a D_{\bm u}'{\mc W}(\beta) \\ * d\beta + a \nu d\mu(u) \cdot v  \\
\ov{d^*} \beta  \end{array}\right].
\eeqn
Here $\beta$ and $v$ are infinitesimal deformations of $A$ and $u$, respectively, $a \in [0, 1]$ and $D_{\bm u}' {\mc W}$ is the partial derivative of ${\mc W}$ in the $\beta$ direction. By varying $a$ from $1$ to $0$, this family of operators remain Fredholm and hence the orientation is reduced to the case of $a = 0$. Moreover, the operator $\beta \mapsto (* d\beta, \ov{d^*} \beta)$ is independent of ${\bm u}$, hence is naturally orientable. Therefore, the orientability is reduced to the family $\{ D_{\bm u} {\mc W}\}_{{\bm u}\in {\mc B}}$. \end{proof}

Next, notice that ${\mc B}$ is not a product, but a fibre bundle: the condition on the section $u$ depends on the connection $A$ and the framing ${\uds \phi}$. We denote by ${\mc B}_{A, {\uds \phi}} ({\uds \kappa})\subset {\mc B}({\uds \kappa})$ be the subset consisting of $(A, u, {\uds \phi})$ with fixed $(A, {\uds \phi})$. Moreover, since $A$ lies in a contractible space, which doesn't affect the orientability. Hence we omit the dependence of $A$ and denote by ${\mc B}({\uds \phi}, {\uds \kappa})$ be the space of all such $(A, u, {\uds \phi})$'s. Then the orientability of $\{D_{\bm u}{\mc W}\}_{{\bm u} \in {\mc B}( {\uds \phi}, {\uds \kappa})}$ is similar to the case of ordinary Gromov--Witten and Floer theory. The orientation can be chosen in a coherent way in the sense of \cite{Floer_Hofer_orientation}. More precisely, it means the following. For each critical point $\kappa_i \in {\bf Crit} \big[ \tilde{W}_i|_{\tilde{X}_{\gamma_i}} \big]$, one can choose an orientation on the unstable manifold $W^u_{\kappa_i}$. Then for two critical points $\kappa_i, \kappa_i'$, choose a curve $l_i: {\mb R} \to \tilde{X}_i$ such that $l_i|_{(-\infty, -1]} = \kappa_i$ and $l_i|_{[1, +\infty)} = \kappa_i'$. Define
\beqn
\tilde{l}_i: {\mb R} \times S^1 \to \tilde{X},\ \tilde{l}_i (s, t) = e^{-m_i \frac{\i t}{r}} l_i (s).
\eeqn
Then by gluing $u \in {\mc B}( \uds\phi, {\uds \kappa})$ with $\tilde{l}_i$ in a consistent way, one obtains a continuous map
\beqn
\uds l: {\mc B}({\uds \phi}, {\uds \kappa}) \to {\mc B}({\uds \phi}, {\uds \kappa}').
\eeqn
Let the determinant line bundle be ${\mc L}\to {\mc B}({\uds \varphi}, {\uds \kappa})$ and ${\mc L}' \to {\mc B}({\uds \varphi}, {\uds \kappa}')$ respectively.

On the other hand, the choice of orientations on the unstable manifolds of $\kappa_i$ and $\kappa_i'$ induces an orientation of the operator
\beqn
D_{l_i}: W^{1, p}({\mb R}, l_i^* T\tilde{X}_i) \to L^p({\mb R}, l_i^* T\tilde{X}_i),\ D_{l_i}(\xi) = \nabla_s \xi + \nabla^2 \tilde{W}_i (l_i) \cdot \xi
\eeqn
and hence an orientation of the linearization of the soliton equation along $\tilde{l}_i$ (for details see Section \ref{section6}). The gluing process induces an isomorphism ${\uds l}{}_*: {\mc L} \to {\uds l}^* {\mc L}'$ of the determinant line bundles (well-defined up to homotopy).

\begin{defn}
An orientation of the family $\{ D_{\bm u} {\mc W}\}_{{\bm u} \in {\mc B}({\uds \phi}, {\uds \kappa})}$ and an orientation of the family $\{D_{{\bm u}'}{\mc W}\}_{{\bm u}' \in {\mc B}({\uds \phi}, {\uds \kappa}')}$ are coherent with respect to the orientations on the unstable manifolds of all $\kappa_i$ and $\kappa_i'$, if they are consistent with the isomorphism ${\uds l}{}_*$.
\end{defn}

The following proposition follows from the well-known results of Floer--Hofer \cite{Floer_Hofer_orientation}.
\begin{prop}
Fix $\uds\phi$ and $\uds \kappa$.
\begin{enumerate}
\item The family $\{ D_{{\bm u}}{\mc W}\}_{ {\bm u} \in {\mc B}({\uds \phi}, {\uds \kappa})}$ is orientable.

\item If we choose orientations on the unstable manifolds $W^u_{\kappa_i}$ for all $i$, then there is a canonically induced orientation on $\{D_{\bm u} {\mc W}\}_{ {\bm u} \in {\mc B}({\uds \phi}, {\uds \kappa})}$. If we change the orientation on one $W^u_{\kappa_i}$, then the orientation of $\{ D_{\bm u}{\mc W}\}_{{\bm u} \in {\mc B}({\uds \phi}, {\uds \kappa})}$ changes.

\item Let $\uds \kappa' = (\kappa_1', \ldots, \kappa_b')$ be another set of critical points and choose orientations on $W^u_{\kappa_i}$ and $W^u_{\kappa_i'}$ for all $i$. Then the induced orientation on $\{ D_{{\bm u}} {\mc W}\}_{{\bm u} \in {\mc B}({\uds \phi}, {\uds \kappa})}$ and the induced orientation on $\{ D_{{\bm u}'} {\mc W}\}_{{\bm u}'\in {\mc B}({\uds \phi}, {\uds \kappa}')}$ are coherent.
\end{enumerate}
\end{prop}

Lastly, we need to show that the orientation bundle over the space of all ${\uds \phi}$ is trivial. It suffices to consider the variation of one $\phi_i$. It is easy to see that changing $\phi_i$ to $e^{\i \theta} \phi_i$ is the same as fixing $\phi_i$ while changing $(a_i, F_i)$ to $(e^{\i r \theta} a_i, F_i \circ e^{-\i \theta})$. On the other hand, the orientation of $\{ D_{\bm u} {\mc W}\}_{{\bm u} \in {\mc B}({\uds \phi}, {\uds \kappa})}$ only depends on the orientation on $W^u_{\kappa_i}$. As we vary $\theta$, the critical point $\kappa_i = (p, x_i)$ is changed to $e^{\i \theta} \kappa_i:= (e^{- \i r \theta} p, e^{\i \theta} x_i)$. Moreover, $e^{\i \theta}$ will map the unstable manifold of $\kappa_i$ to that of $e^{\i \theta} \kappa_i$. It is also easy to see that as $\theta$ moves from $0$ to $2\pi$, the orientation on the unstable manifold doesn't alter. Therefore, the orientation bundle is trivial over the space of ${\uds \phi}$.

\subsection{Virtual cycle on ${\mc M}_{\uds P}({\uds \kappa})$}\label{subsection34}

Since ${\uds P}$ is strongly regular, ${\mc M}_{\uds P}({\uds \kappa}) = \ov{\mc M}_{\uds P}({\uds \kappa})$ because there is no nontrivial solitons. Hence ${\mc M}_{\uds P}({\uds \kappa})$ is compact with respect to the weak topology. We remark that the weak topology on ${\mc M}_{\uds P}({\uds \kappa})$ coincides with the topology induced from ${\mc B}$. This is due to the common feature of $\epsilon$-regularity in elliptic theories; in our case, it was proved in \cite[Section 4]{Tian_Xu} that when the energy on a cylindrical end is sufficiently small, the energy density should decay exponentially. Hence if a sequence in ${\mc M}_{\uds P}({\uds \kappa})$ converges in the weak topology, since there cannot be any energy escape at infinity to form a nontrivial soliton in the limit, the sequence also converges strongly. Abbreviate ${\mc M}_{\uds P}({\uds \kappa})$ by ${\mc M}$. The remaining of this subsection is devoted to the construction of a virtual orbifold atlas on ${\mc M}$.

Since the moduli spaces we considered here are compact, it is possible to construct the virtual cycle using simpler argument, For example, one can cover the whole moduli space by a single chart. However, we provide a more general argument where the atlas has charts indexed by a partially ordered set, to illustrate how to construct virtual cycle in more general cases.

\subsubsection{The local model}

The first step is to construct a local chart around each point of ${\mc M}$, represented by certain ${\bm u} \in {\mc F}^{-1}(0)$. If $\Sigma$ has at least one puncture, then the ${\mc G}$-action on ${\mc A}$ is free and there is no nontrivial automorphisms of solutions to ${\mc F}({\bm u}) = 0$. Then ${\mc M} = {\mc F}^{-1}(0)$. If $\Sigma$ has no puncture (see Remark \ref{rem311} below), then ${\bm u}$ possibly has a nontrivial finite automorphism group $\Gammait \subset G$ and ${\mc M} \simeq {\mc F}^{-1}(0)/ G$. In any case, the linearization $D_{\bm u} {\mc F}$ is $\Gammait$-equivariant. By compactness of ${\mc F}^{-1}(0)$, there are only finitely many different automorphism groups. Let $\tilde\Gammait \subset G$ be the finite group generated by these automorphisms groups. 

We would like to find an obstruction space $\tilde E_{\bm u} \subset {\mc E}_{\bm u}$, which is a $\Gammait_{\bm u}$-invariant finite dimensional subspace such that 
\beq\label{eqn37}
{\bf Image} ( D_{\bm u} {\mc F}) + \tilde E_{\bm u} = {\mc E}_{\bm u}.
\eeq
In the approach of \cite{Fukaya_Ono}, if such an obstruction space is chosen, then one can extend it to nearby objects. In the approach of \cite{Li_Tian}, we can choose the obstruction space in a more global fashion so that it can be extended to arbitrary objects (at least in the same stratum). Although there is no essential difference, but the choice will affect the remaining steps of the virtual cycle construction. In the following we will follow the approach of \cite{Li_Tian}.

Let $V_c'$ be the space of compactly supported smooth sections of $\pi^* \Omega_{\Sigma^*}^{0,1} \otimes T^\bot Y$, where $\pi: Y \to \Sigma^*$ is the projection. Elements of $G$, viewed as constant gauge transformations, acts on $V_c'$ through the gauge action on $Y$. Let $V_c''$ be the space of compactly supported smooth sections in $L_\tau^p({\mf g}) \oplus \ov{L_\tau^p({\mf g})}$, i.e., the target space of the vortex equation plus the gauge fixing condition.

\begin{lemma}
There exists a finite dimensional subspace $V_{\bm u}' \subset V_c'$ and a finite dimensional subspace $\tilde E_{\bm u}'' \subset V_c''$ such that if we denote $\tilde E_{\bm u}':= u^* (V_{\bm u}') \subset {\mc E}_{\bm u}'$ and $\tilde E_{\bm u} = \tilde E_{\bm u}' \oplus \tilde E_{\bm u}''$, then \eqref{eqn37} holds. Moreover, when $\Sigma$ has no punctures, $V_{\bm u}'$ can be chosen to be $\tilde \Gammait$-invariant.
\end{lemma}

\begin{proof}
It suffices to prove that there exists $V_{\bm u}' \subset V_c'$ such that the partial derivative of ${\mc W}$ at ${\bm u}$ in the $u$-direction is transversal to $u^* V_{\bm u}'$. This partial derivative is a Cauchy--Riemann type operator, which gives a Fredholm operator $D_{\bm u} {\mc W}$. Then there exists a finite dimensional subspace $\tilde E_{\bm u}' \subset L_\tau^p(\Sigma^*, \Lambda^{0,1} \otimes u^* T^\bot Y)$ which is generated by sections with compact support and which is transversal to the image of $D_{\bm u} {\mc W}$. Since $u$ is embedded into the total space $Y$, one can extend $\tilde E_{\bm u}'$ to a globally defined subspace $V_{{\bm u}}'$ of compactly supported sections on $Y$. Lastly, when $\Sigma$ has no punctures and ${\bm u}$ may have a nontrivial automorphism group $\Gammait_{{\bm u}} \subset G$, we enlarge $V_{{\bm u}}'$ to $\Gammait_{{\bm u}} V_{{\bm u}'}$, which is still transverse to the image of $D_{{\bm u}} {\mc W}$. 
\end{proof}

For $e = (e', e'') \in \tilde E_{\bm u}' \oplus  \tilde E_{\bm u}''$ and ${\bm u}' = (A', u', \uds\varphi') \in {\mc B}$, denote $e({\bm u}') = (e'({\bm u}'), e'')$, where $e'({\bm u}') \in \Gamma(\Sigma^*, \Lambda^{0,1} \otimes (u')^* T^\bot Y)$ is the pullback of $e'$ via $u'$. Then one can consider the following section
\begin{align*}
&\ {\mc F}_{\bm u}: \tilde E_{\bm u} \times {\mc B} \to {\mc E},\ &\ {\mc F}_{\bm u}(e, {\bm u}') = e({\bm u}') + {\mc F}({\bm u}').
\end{align*}
By the transversality condition of $E_{\bm u}$, there exists a $\Gammait_{\bm u}$-invariant neighborhood ${\mc U}_{\bm u} \subset {\mc B}$ of ${\bm u}$ such that $\tilde{U}_{\bm u}: = {\mc F}_{\bm u}^{-1}(0) \cap (\tilde E_{{\bm u}} \times {\mc U}_{\bm u})$ is a smooth manifold. We shrink ${\mc U}_{\bm u}$ such that
\beqn
{\mc U}_{\bm u} \cap \gamma {\mc U}_{\bm u} = \emptyset,\ \forall \gamma \in \tilde \Gammait \setminus \Gammait_{\bm u}.
\eeqn
Denote $U_{\bm u} = \tilde{U}_{\bm u}/ \Gammait_{\bm u}$ which is an orbifold. By abuse of notation, denote $E_{\bm u} = (\tilde{U}_{\bm u} \times \tilde E_{\bm u})/ \Gammait_{\bm u}$ as an orbibundle over $U_{\bm u}$. There is the natural section $S_{\bm u}: U_{\bm u} \to E_{\bm u}$ induced by the projection $\tilde {\mc U} \times \tilde E_{\bm u} \to \tilde E_{\bm u}$. $S_{\bm u}^{-1}(0)$ naturally embeds into ${\mc M}_{\uds P}(B)$, via a map $\psi_{\bm u}$ whose image $F_{\bm u}$ is an open neighborhood of ${\bm u}$ in ${\mc M}_{{\uds P}}(B)$. This gives a local chart $C_{{\bm u}} = (U_{\bm u}, E_{\bm u}, S_{\bm u}, \psi_{\bm u}, F_{\bm u})$.

\subsubsection{The local charts}

Choose finitely many ${\bm u}_i \in {\mc M}$, $i=1, \ldots, N$ such that the collection of footprints $\{ F_{{\bm u}_i} \}$ in the charts constructed above forms an open cover of ${\mc M}$. Now we construct a collection of charts $C_i = ( U_i, E_i, S_i, \psi_i, F_i)$ still indexed by elements in $\{ 1, \ldots, N\}$ such that there are coordinate changes among them. For each $i$, define 
\begin{align*}
&\  \Gammait_i: = \langle \Gammait_{{\bm u}_1}, \ldots, \Gammait_{{\bm u}_i}\rangle \subset \tilde \Gammait \subset G,\ &\ \tilde E_i:= \bigoplus_{1 \leq l \leq i} \tilde E_{{\bm u}_i}.
\end{align*}
An element of $\tilde E_i$ is denoted by $\uds e_i = (e_1, \ldots, e_i)$. Since each $\tilde E_{{\bm u}_i}$ is $\tilde \Gammait$-invariant, $\tilde E_i$ is acted by $\Gammait_i$. Moreover, for $i \leq j$, there is a natural inclusion $\tilde E_i \hookrightarrow \tilde E_j$.

On the other hand, define the section 
\begin{align*}
&\ {\mc F}_i: \tilde E_i \times {\mc B} \to {\mc E},\ &\ {\mc F}_i( e_1, \ldots, e_i, {\bm u}) = \sum_{1\leq l \leq i} e_l({\bm u}) + {\mc F}({\bm u}).
\end{align*}
By the $\Gammait_i$-invariance of $E_i$ and the gauge invariance of ${\mc F}$, one has that ${\mc F}_i$ is a $\Gammait_i$-equivariant section. Denote ${\mc U}_i = \Gammait_i {\mc U}_{{\bm u}_i} \subset {\mc B}$, which is a $\Gammait_i$-invariant open subset of the Banach manifold. Since ${\mc U}_{{\bm u}_i} \subset {\mc B}$ is chosen such that ${\mc F}_{{\bm u}_i}^{-1}(0) \cap ( \tilde E_{{\bm u}_i} \times {\mc U}_{{\bm u}_i})$ is transverse, the ``thickening'' $\tilde{U}_i:= {\mc F}_i^{-1}(0) \cap ( \tilde E_i \times {\mc U}_i )$ is also a $\Gammait_i$-invariant smooth manifold. Define $U_i = \tilde{U}_i/ \Gammait_i$ which is an orbifold (where the $\Gammait_i$-action is effective), and denote $E_i:= (\tilde E_i \times \tilde{U}_i) / \Gammait_i$ the orbifold bundle over $U_i$. Similar as above, one obtains a chart $C_i = (U_i, E_i, S_i, \psi_i, F_i)$. Notice that $F_i = F_{{\bm u}_i}$, so the collection of $F_i$ still cover ${\mc M}$. 

\subsubsection{The coordinate changes}

Now we construct coordinate changes. For each pair $i \leq j$, if $F_i \cap F_j \neq \emptyset$, define 
\beqn
{\mc U}_{ji} = {\mc U}_i \cap {\mc U}_j.
\eeqn
It is $\Gammait_i$-invariant since $\Gammait_i \subset \Gammait_j$. Define 
\beqn
\tilde U_{ji} = {\mc F}_i^{-1}(0) \cap ( \tilde E_i \times {\mc U}_{ji}) \subset \tilde U_i,\ \ U_{ji} = \tilde U_{ji}/ \Gammait_i.
\eeqn
Then $U_{ji} \subset U_i$ is open and the footprint is exactly $F_i \cap F_j$. Moreover, because $\tilde E_i \subset \tilde E_j$, there is an equivariant embedding $\tilde U_{ji} \to \tilde U_j$ which descends to an orbifold embedding $\phi_{ji}: U_{ji} \to U_j$. Combining with the inclusion $\tilde E_i \hookrightarrow \tilde E_j$, we obtain a bundle embedding $\wh\phi_{ji}: E_i|_{U_{ji}} \to E_j$.

Now we check that the triple $T_{ji} = (U_{ji}, \phi_{ji}, \wh\phi_{ji})$ is indeed a coordinate change from $C_i$ to $C_j$ in the sense of Definition \ref{defn722}. First, by construction, the footprint of $U_{ji}$ is $F_i \cap F_j$. Second, if $x_k \in U_i$ converges to $x_\infty \in U_i$ and the embedding images $y_k = \phi_{ji}(x_k)$ converges to $y_\infty \in U_j$, then we can lift $x_k$, $y_k$ to sequences $\tilde x_k  = \tilde y_k \in {\mc U}_{ji}$. The convergences imply that $\tilde x_k$ and $\tilde y_k$ converge to $\tilde x_\infty = \tilde y_\infty \in {\mc U}_{ji}$. Hence $x_\infty \in U_{ji}$ and $y_\infty = \phi_{ji}(x_\infty)$. Therefore, $T_{ji}$ is a coordinate change from $C_i$ to $C_j$. 

We check the cocycle condition and the overlapping condition of Definition \ref{defn724}. The cocycle condition follows directly from our construction. The overlapping condition holds because the index set $\{1, \ldots, N\}$ is totally ordered. Therefore, the atlas ${\mf A} = ( \{ C_i\}, \{ T_{ji} \})$ is a virtual orbifold atlas. By Theorem \ref{thm728}, an appropriate shrinking of ${\mf A}$ produces a good atlas. 

Lastly, combining the previous discussion on orientations, the good atlas is oriented. Then by the abstract topological construction of the virtual cycle given in the section7, there is a well-defined virtual cardinality $\# {\mc M}\in {\mb Q}$ whenever the virtual dimension is zero. This finishes the proof of Theorem \ref{thm31}.

\begin{rem}\label{rem311}
We comment on a special case. When there is no puncture, i.e., there is a smooth $r$-th root of $K_\Sigma$, which is only possible when $r$ divides $2g-2$, the ${\mc G}_\tau$-action on ${\mc B}$ may not be free because ${\mc G}_\tau$ contains constant gauge transformations now.  On the other hand, by the energy identity, if $(A, u)$ is a solution in this case, then $u(\Sigma) \subset {\bf Crit} W$ and $(A, u)$ is also a solution to the symplectic vortex equation. In this case, in order to get a good theory, one has to choose the moment map (by adding to $\mu$ a constant in ${\mf g}$) so that for solutions $(A, u)$, $u(\Sigma)$ is not contained in $\mu^{-1}(c)$ for a singular value $c$ of $\mu$. Then the automorphism group of a solution $(A, u)$ is at most finite. 
\end{rem}

\begin{rem}
There have been various choices made in the construction of the good virtual orbifold atlas. We have to prove that the virtual count is independent of these choices. The two major choices are: the finite many charts $C_{{\bm u}_i}$ centered at $[{\bm u}_i]\in {\mc M}$ whose footprints cover ${\mc M}$; the various shrinkings we made in order to achieve a good virtual orbifold atlas. The independence from the second set of choices can be proved in an abstract setting. On the other hand, if we have two collections of charts $\{ C_{{\bm u}_i} \}$ and $\{ C_{{\bm u}_i'} \}$, then one can construct a virtual orbifold atlas with boundary on the product ${\mc M} \times [0, 1]$, such that the two boundary components have the corresponding virtual orbifold atlases constructed out of these two collection of basic charts. Using this cobordism one can prove that the virtual count is well-defined.
\end{rem}

\section{Invariance of the Correlation Function}\label{section4}

In this section we prove the invariance of the correlation function under changes of strongly regular perturbations. The main idea of the proof has been provided in \cite{Tian_Xu_2} and we give more details here, although the most technical part (the wall-crossing formula) will be completed in the companion paper \cite{Tian_Xu_4}.

\begin{thm}\label{thm41}
The correlation function defined by \eqref{eqn35} is independent of the choice of the strongly regular perturbation ${\uds P}$.
\end{thm}

We need to compare the correlation functions defined for two different strongly regular perturbations. This comparison is obviously reduced to the case that they only differ for a single broad puncture. Therefore, to save notations for the remaining of this paper, we assume that ${\mc C}$ has only one broad marking ${\bf z}$ with monodromy labelled by $m\in \{0, 1, \ldots, r-1\}$. Let $\phi \in {\bf Fr}$ be the frame variable. 

The proof consists of several steps. First we reduce the comparison to the case that $a_- = a_+ = a$. This reduction is obtained in the next subsection. Then to compare $(a, F_-)$ and $(a, F_+)$, we connect $F_-$ and $F_+$ by a smooth path $F_\iota \in {\bf V} \setminus {\bf V}_a^{\rm sing}$ and use a cobordism argument. The path may cross the wall where the perturbation is not strongly regular, and we have to prove a wall-crossing formula (Theorem \ref{thm46}). Using the wall-crossing formula the invariance of the correlation function is proved.

\subsection{Independence on $a$}

Take $a \in {\mb C}^*$ and $F \in {\bf V} \setminus \big( {\bf V}_a^{\rm sing} \cup {\bf V}_a^{\rm wall} \big)$. For $a_\lambda = \lambda^r a$ with $\lambda > 0$, we define $F_\lambda (x) = \lambda^r F ( \lambda^{-1} x)$ where we view $\lambda$ as an element of $(G')^{\mb C}$. Denote $P_\lambda = (a_\lambda, F_\lambda)$ and $\tilde{W}_\lambda = W - a_\lambda p + F_\lambda$. Then 
\beqn
\left\{ \begin{array}{c} Q (x) = a; \\
                         p dQ (x) + d F(x) = 0 \end{array} \right. \Longleftrightarrow \left\{ \begin{array}{c} Q( \lambda x) = a_\lambda; \\
													p dQ ( \lambda x ) + d F_\lambda ( \lambda x) = 0 \end{array} \right.
\eeqn
Then $\kappa = (p, x) \in {\bf Crit} \big[ \tilde{W}|_{\tilde{X}_\gamma} \big]$ implies that $\kappa_\lambda:= (p, \lambda x) \in {\bf Crit} \big[\tilde{W}_\lambda|_{\tilde{X}_\gamma} \big]$. Moreover, $\tilde{W}_\lambda (\kappa_\lambda) =  \lambda^r \tilde{W} (\kappa)$.
Therefore, $P$ is strongly regular if and only if $P_\lambda$ is strongly regular. Now for any $\lambda>1$, consider the family of perturbations $P_\alpha$ for $\alpha \in [1, \lambda]$. Consider the moduli space
\beqn
\tilde{\mc M}_{\tilde P} (\kappa) = \Big\{ (\alpha, [{\bm u}])\ |\ \alpha \in [1, \lambda],\ [{\bm u}] \in {\mc M}_{P_\alpha} ( \kappa_\alpha) \Big\}.
\eeqn

\begin{thm}\label{thm42}
There is an oriented virtual orbifold atlas on $\tilde{\mc M}_{\tilde{P}}(\tilde{\kappa})$ with oriented virtual boundary $\left[ {\mc M}_{P_\lambda} (\kappa_\lambda) \right] \sqcup \left[ - {\mc M}_P (\kappa) \right]$. Therefore $\# {\mc M}_P(\kappa) = \# {\mc M}_{P_\lambda} (\kappa_\lambda)$.
\end{thm}

\begin{proof}
The proof is a standard homotopy argument and we omit the details.
\end{proof}

Therefore, the comparison is reduced to the case that $a_- = e^{\i \theta} a_+ = e^{\i \theta} a$ for some $\theta\in [0, 2\pi)$. Take $F_\theta (x) = F( e^{- \i \theta/r} x)$ and $P_\theta = ( e^{\i \theta} a, F_\theta)$ and $\tilde{W}_\theta = W - a_\theta p + F_\theta$. Then
\beqn
\left\{ \begin{array}{c} Q(x) = a; \\
                     p dQ ( x) + dF (x) = 0 \end{array} \right. \Longleftrightarrow \left\{ \begin{array}{c} Q ( e^{\i \theta /r} x) = a_\theta;  \\
										( e^{- \i \theta} p) dQ ( e^{\i \theta/r} x) + dF_\theta ( e^{\i \theta/r} x) = 0 \end{array} \right. \eeqn
Then if $\kappa = (p, x)$ is a critical point of $\tilde{W}$ restricted to $\tilde{X}_\gamma$, then $\kappa_\theta:= (e^{-\i \theta}p, e^{\i \theta/r} x)$ is a critical point of $\tilde{W}_\theta$ restricted to $\tilde{X}_\gamma$. Moreover, $\tilde{W}_\theta (\kappa_\theta) = \tilde{W}(\kappa)$. So all critical values of $\tilde{W}_\theta$ still have distinct imaginary parts, namely, $P_\theta$ is strongly regular. 

\begin{thm}\label{thm43}
There exist a homeomorphism ${\mc M}_P(\kappa) \simeq {\mc M}_{P_\theta} (\kappa_\theta)$ and an identification between their oriented virtual orbifold atlases, compatible with this homeomorphism. Therefore $\# {\mc M}_P (\kappa) = \# {\mc M}_{P_\theta}(\kappa_\theta)$. 
\end{thm}

\begin{proof}
We define a gauge transformation $g_\theta'': \Sigma^* \to G''$ which is equal to $e^{\i \theta/r}$ over $C_i$ and, by using a cut-off function, extended to the identity away from $C_i$. Notice that this is not an element of ${\mc G}$. For any solution ${\bm u} = (A, u, \phi) \in \tilde{\mc M}_P (\kappa)$, define 
\beqn
{\bm u}_\theta = (A_\theta, u_\theta, \phi_\theta):= ( g_\theta^* A, g_\theta^* u, e^{-\i\theta/r} \phi).
\eeqn
We would like to show that ${\bm u}_\theta$ is a solution to the gauged Witten equation with perturbation $P_\theta$. First, away from the cylindrical end where $\tilde{\mc W}_{A, \phi}$ coincides with the unperturbed ${\mc W}_A$, ${\bm u}_\theta$ still solves the equation by the gauge invariance property of $\tilde{\mc W}_{A, \phi}$. On the cylindrical end, denote $\check{g}_\theta'' = e^{-\i \theta/r} g_\theta''$. Then we see that $\check{A}_\theta:= (\check{g}_\theta'')^*A = (g_\theta'')^* A$. Then by the property of $\tilde{\mc W}_{A, \phi}$ (see Subsection \ref{subsection24}), over the cylindrical end one has
\beqn
\tilde{\mc W}_{A_\theta, \phi_\theta}  = \tilde{\mc W}_{ \check{A}_\theta, \phi_\theta} = \tilde{\mc W}_{A, \phi_\theta} \circ \check{g}_\theta'' = \tilde{\mc W}_{A, \phi} \circ g_\theta''.
\eeqn
Then since $g_\theta''$ is a unitary gauge transformation which preserves the metric, one has
\beqn
( g_\theta'')_*  \Big[  \ov\partial_{A_\theta} u_\theta + \nabla \tilde{\mc W}_{A_\theta, \phi_\theta} (u_\theta) \Big] =  \ov\partial_A u + (g_\theta'')_* \nabla \tilde{\mc W}_{A_\theta, \phi_\theta} (u_\theta) = \ov\partial_A u + \nabla \tilde{\mc W}_{A, \phi}(u) = 0.
\eeqn
Therefore, ${\bm u}_\theta$ solves the gauged Witten equation with perturbation $P_\theta$. Hence $g_\theta''$ induces a homeomorphism between ${\mc M}_P(\kappa)$ and ${\mc M}_{P_\theta}(\kappa_\theta)$. It also induces an identification between the the virtual orbifold atlases because $g_\theta''$ is globally defined. 
\end{proof}

\begin{cor}
The correlation functions defined for $P$ and $P_\theta$ coincide.
\end{cor}

\begin{proof}
The map $f_\theta: x\mapsto e^{\i \theta/r} x$ induces an isomorphism 
\beqn
f_\theta^*: H^{\rm mid} (Q_\gamma^{a_\theta})^{{\mb Z}_r} \simeq H^{\rm mid} (Q_\gamma^a)^{{\mb Z}_r}
\eeqn
which intertwines with the isomorphisms $H^{\rm mid}(Q_\gamma^{a_\theta})^{{\mb Z}_r} \simeq {\mc H}_\gamma \simeq H^{\rm mid}( Q_\gamma^a)^{{\mb Z}_r}$. 
Moreover $f_\theta$ maps gradient flows of $F|_{Q_\gamma^a}$ to the gradient flow of $F_\theta|_{Q_\gamma^{a_\theta}}$ since $f_\theta$ is an isometry. Therefore, given ${\mc O} \in {\mc H}_\gamma \simeq H^{\rm mid}(Q_\gamma^a)^{{\mb Z}_r}$ which corresponds to a cocycle ${\mc O}_\theta \in H^{{\rm mid}}(Q_\gamma^{a_\theta})^{{\mb Z}_r}$, for any $\kappa\in {\bf Crit} \tilde{W}|_{\tilde{X}_\gamma}$ and $\kappa_\theta \in {\bf Crit} \tilde{W}_\theta|_{\tilde{X}_\gamma}$, one has
\beqn
 {\mc O}^* \cap [W_\kappa^u] = {\mc O}_\theta^* \cap [ W_{\kappa_\theta}^u].
\eeqn
By the definition of the correlation function \eqref{eqn35} and Theorem \ref{thm43}, the correlation functions defined for $P$ and $P_\theta$ coincide.
\end{proof}

\subsection{The wall-crossing}

Now we reduce the comparison between correlation functions for the two strongly regular perturbations to the case that $a_+ = a_- = a$. The detail of this part will be greatly expanded and will be contained in the companion paper \cite{Tian_Xu_4}. Here we only discuss up to the point where the wall-crossing formula appears, and how it implies the invariance of the correlation function. Consider $F_\pm \in {\bf V} \setminus ( {\bf V}_a^{\rm sing} \cup {\bf V}_a^{\rm wall} )$. Since ${\bf V}_a^{\rm sing}$ is a proper complex analytic set, ${\bf V} \setminus {\bf V}_a^{\rm sing}$ is path-connected. So there is a smooth path $\tilde F$ in ${\bf V} \setminus {\bf V}_a^{\rm sing}$ connecting $F_\pm$. We also write the path $\tilde F$ by $F_\iota$ for $\iota \in [\iota_-, \iota_+]$ and denote $P_\iota = (a, F_\iota)$. We can perturb the path such that it only intersects ${\bf V}_a^{\rm wall}$ through its smooth locus transversely. Therefore we may assume that the intersection only appears at $\iota = \iota_0\in (\iota_-, \iota_+)$. By the definition of ${\bf V}_a^{\rm wall}$, there are smooth paths $\upsilon_\iota, \kappa_\iota \in {\bf Crit} \tilde{W}_\iota |_{\tilde{X}_\gamma} \simeq {\bf Crit} F_\iota |_{Q_\gamma^a}$ such that
\beqn
{\bf Re} F_{\iota_0} ( \upsilon_{\iota_0} ) > {\bf Re} F_{\iota_0} (\kappa_{\iota_0} ),\ {\bf Im} F_{\iota_0} (\upsilon_{\iota_0} ) = {\bf Im} F_{\iota_0} (\kappa_{\iota_0} ),
\eeqn
\beq\label{eqn41}
 \left. \frac{d}{d\iota } \right|_{\iota = \iota_0} \big[ {\bf Im} F_\iota (\upsilon_\iota ) - {\bf Im} F_\iota  (\kappa_\iota ) \big] \neq 0.
\eeq
The other critical values all have distinct imaginary parts for all $\iota \in [\iota_-, \iota_+]$. Denote
\beqn
\upsilon_\pm = \upsilon_{\iota_\pm}, \kappa_\pm = \kappa_{\iota_\pm}.
\eeqn

\begin{defn}\label{defn45}
We define $(-1)^{\tilde{F}}:= {\rm sign} \tilde F$ to be $+1$/$-1$ if \eqref{eqn41} is positive/negative.
\end{defn}

\begin{thm}\label{thm46}
In the situation described above, the following holds. For $\kappa' \neq \kappa$, one has $\# {\mc M}_{P_+}(\kappa') = \# {\mc M}_{P_-}(\kappa')$. On the other hand, 
\beqn
\# {\mc M}_{P_+} ( \kappa ) - \# {\mc M}_{P_-} ( \kappa) = - (-1)^{\tilde F} \Big[ \# {\mc N}(\upsilon_{\iota_0} ,\kappa_{\iota_0} ) \Big]  \Big[ \# {\mc M}_{P_-}  ( \upsilon) \Big].
\eeqn
Here $\# {\mc N}(\upsilon_{\iota_0}, \kappa_{\iota_0} )$ is the algebraic counts of the number of BPS solitons in $\tilde X_\gamma$ for the perturbed function $\tilde W_{\iota_0} = W - ap + F_{\iota_0}$ between the two critical points $\upsilon_{\iota_0}$ and $\kappa_{\iota_0}$.
\end{thm}

The above theorem is referred to as the bifurcation or wall-crossing formula, which will be proved in the companion paper \cite{Tian_Xu_4}. Here we briefly explain the proof. One needs to construct a virtual fundamental chain over a universal moduli space fibred over $[\iota_-, \iota_+]$ parametrizing the path of perturbations. Because of the appearance of soliton solutions for the non-strongly regular perturbation $P_{\iota_0}$, the universal moduli space is not compact. One can compactify it by adding soliton solutions. Moreover, when constructing the virtual chain, an important fact is that non-BPS soliton solutions are interior points and BPS soliton solutions are boundary points of the universal moduli space. This is why the counting of BPS solitons appears in the wall-crossing formula. This is the same situation happening in Hamiltonian Floer theory when one has time-independent Hamiltonians and almost complex structures, and  this phenomenon also appeared in \cite{FJR3} in the case of Landau--Ginzburg theory.

On the other hand, the coefficients ${\mc O}_i^* \cap [ W_{\kappa_i}^u]$ in \eqref{eqn35} also undergo a bifurcation when crossing the wall. This is just the famous Picard--Lefschetz formula, stated as follows. A key issue is the sign. 
\begin{prop}[Picard--Lefschetz formula] \label{prop47}  For $M$ sufficiently large, we have
\beqn
[ W_{\upsilon_+}^u ] - [ W_{\upsilon_-}^u ] = (-1)^{\tilde{F}} \# {\mc S}(\upsilon_{\iota_0}, \kappa_{\iota_0} ) [ W_{\kappa_{\iota_0}}^u ]\in H_{\rm mid} ( Q_\gamma^a,\  {\bf Re} F_{\iota_0} \leq -M ).
\eeqn
Here $\# {\mc S}(\upsilon_{\iota_0}, \kappa_{\iota_0})$ is the algebraic count of BPS solitons in $Q_\gamma^a$ for the function $F_{\iota_0}$ between the two critical points $\upsilon_{\iota_0}$ and $\kappa_{\iota_0}$ (see Definition \ref{defn64}).
\end{prop}
This proposition is proved in Section \ref{section6} where a key issue is to make the sign correct. 

A further step is to identify the two countings $\# {\mc N}(\upsilon_{\iota_0}, \kappa_{\iota_0})$ and $\# {\mc S}(\upsilon_{\iota_0}, \kappa_{\iota_0})$. 

\begin{prop}\label{prop48}
One has $\# {\mc N}(\upsilon_{\iota_0}, \kappa_{\iota_0}) = \# {\mc S}(\upsilon_{\iota_0}, \kappa_{\iota_0})$.
\end{prop}
The proof, given in Section \ref{section5}, is based on an adiabatic limit argument which was previously used in \cite{Lagrange_multiplier}.

Now one can derive the equality between the correlation functions for $F_{\iota_-}$ and $F_{\iota_+}$. The proof has already been given in \cite{Tian_Xu_2}. However we still include it here. 

For each path of critical points $\kappa_\iota \in {\bf Crit} \big[ \tilde{W}_\iota |_{\tilde{X}_\gamma} \big]$, we denote by $n_\kappa^\pm$ the virtual cardinality of the moduli space ${\mc M}_{P_\pm}(\kappa_\pm)$. Then the correlation functions for $P_\pm$ are linear functionals on $H^{\rm mid} (Q_\gamma^a; {\mb Q})^{{\mb Z}_r}$ defined by
\beqn
\langle {\mc O} \rangle_\pm = \sum_{\kappa^\pm \in {\bf Crit} [ \tilde{W}_{\tilde{X}_\gamma}^\pm ] } n_\kappa^\pm \cdot ( {\mc O}^* \cap [(\kappa_\pm)^-]).
\eeqn
Suppose the wall-crossing happens between the two paths $\kappa_\iota', \kappa_\iota''$ such that
\beqn
{\bf Re} \tilde{W}_\iota (\kappa_\iota')  < {\bf Re} \tilde{W}_\iota ( \kappa_\iota''),\ \left. \frac{d}{d\iota }\right|_{\iota = 0} \left[ {\bf Im} \tilde{W}_\iota ( \kappa_\iota') - {\bf Im} \tilde{W}_\iota ( \kappa_\iota'') \right] \neq 0.
\eeqn
Then we have $n_\kappa^+ = n_\kappa^-$ for $\kappa \neq \kappa'$, $[ W_{\kappa_+}^u ] = [W_{\kappa_-}^u]$ for $\kappa \neq \kappa''$. Moreover, by Theorem \ref{thm46}, Proposition \ref{prop47} and Proposition \ref{prop48}, one has
\begin{multline*}
\langle {\mc O} \rangle_+ - \langle {\mc O} \rangle_-  = n_{\kappa'}^+ \cdot \big( {\mc O}^* \cap [W_{\kappa_+'}^u] \big) + n_{\kappa''}^+  \cdot \big( {\mc O}^* \cap [ W_{\kappa_+''}^u] \big) - n_{\kappa'}^- \cdot \big( {\mc O}^* \cap [W_{\kappa_-'}^u ] \big) + n_{\kappa''}^- \cdot \big( {\mc O}^* \cap [W_{\kappa_-''}^u ] \big)\\
= \big( n_{\kappa'}^+ - n_{\kappa'}^- \big) \cdot \big( {\mc O}^* \cap [ W_{\kappa_+'}^u ] \big) +  n_{\kappa''}^+ \cdot \big( {\mc O}^* \cap ( [ W_{\kappa_+''}^u] - [W_{\kappa_-''}^u ] ) \big)\\
= - (-1)^{\tilde{F}} \big( \# {\mc N}(\kappa_0'', \kappa_0') \big) \cdot n_{\kappa''}^+ \cdot \big( {\mc O}^* \cap [W_{\kappa_+'}^u] \big)  +   n_{\kappa''}^+ \cdot \big( {\mc O}^* \cap (-1)^{\tilde{F}} \# {\mc N}(\kappa_0'', \kappa_0') [W_{\kappa_-'}^u] \big) = 0.
\end{multline*}
Therefore it completes the proof of Theorem \ref{thm41}.

\section{Proof of Proposition \ref{prop48}}\label{section5}

In this section we proof Proposition \ref{prop48}. Since we will work purely inside the fixed point $\tilde{X}_\gamma \subset \tilde X$, to save notations we assume that $\gamma = {\rm Id}$. Before going into details let us look at the problem from an intuitive perspective. Let $F = F_{\iota_0} \in {\bf V}_a \setminus {\bf V}_a^{\rm sing}$ be the regular but not strongly regular perturbation. Denote $F^a:= F|_{Q^a}$. Denote $\tilde W = W - ap + F$. Choose $\upsilon, \kappa \in {\bf Crit} \tilde W$, which, by the principle of Lagrange multiplier, corresponds to two critical points of the restriction $F^a$. By abuse of notation we still denote them by $\upsilon$ and $\kappa$. Let ${\mc S}_{F^a} (\upsilon, \kappa)$ be the space of solutions to
\beqn
\frac{ dy}{ds} + \nabla F^a( y(s)) = 0,\ \lim_{s \to -\infty} y(s) = \upsilon,\ \lim_{s \to +\infty} y(s) = \kappa
\eeqn
modulo time translation. On the other hand, the gradient flow equation of $\tilde{W}$ can be written in components as
\beqn
\left\{ \begin{array}{rcc}
\displaystyle \frac{dx}{ds} + \ov{p(s)} \nabla Q(x(s)) + \nabla F(x(s)) & = & 0,\\[0.2cm]
\displaystyle \frac{dp}{ds} + \ov{Q(x(s)) - a} & = & 0.
\end{array}\right.
\eeqn
We introduce a real parameter $\lambda>0$ and consider
\beq\label{eqn51}
\left\{ \begin{array}{rcc}
\displaystyle \frac{dx}{ds} + \ov{p(s)} \nabla Q(x(s)) + \nabla F(x(s)) & = & 0,\\[0.2cm]
\displaystyle \frac{dp}{ds} + \lambda^2 \left[  \ov{Q(x(s)) - a} \right] & = & 0,\\[0.2cm]
\displaystyle \lim_{s \to -\infty} (x(s), p(s)) & = & \upsilon, \\[0.2cm]
\displaystyle \lim_{s \to +\infty} (x(s), p(s)) & = & \kappa.
\end{array}\right.
\eeq
In finite dimensional Morse theory, it is convenient to consider the energy functional
\beq\label{eqn52}
E(x, p)= \frac{1}{2} \left[ \big\| \ov{p(s)} \nabla Q(x) + \nabla F(x) \big\|_{L^2}^2 + \lambda^2 \big\| \ov{Q(x) - a} \big\|_{L^2}^2  +  \big\| \dot x \big\|_{L^2}^2 + \frac{1}{\lambda^2} \big\| \dot p \big\|_{L^2}^2 \right].
\eeq
The energy of a solution to \eqref{eqn51} is equal to $\tilde{W} (\upsilon)- \tilde{W}(\kappa)$, which is independent of $\lambda$. As indicated by the result of \cite{Lagrange_multiplier}, as $\lambda \to +\infty$, solutions are approximately negative gradient lines of the restriction ${\bf Re} F^a$. If we denote by ${\mc S}_{\tilde{W}}^\lambda(\upsilon, \kappa)$ the moduli space of solutions to \eqref{eqn51} (modulo translation), then we will construct an orientation-preserving homeomorphism
\beqn
{\mc S}_{\tilde{W}}^\lambda(\upsilon, \kappa) \simeq {\mc S}_{F^a} (\upsilon, \kappa),\ \forall \lambda >>0.
\eeqn
Notice that ${\mc S}_{\tilde{W}}^1 (\upsilon, \kappa) = {\mc N}(\upsilon, \kappa)$. Moreover, by a homotopy argument the (virtual) counting of ${\mc S}_{\tilde{W}}^\lambda(\upsilon, \kappa)$ is independent of $\lambda$. This will complete the proof of Proposition \ref{prop48}. In the remaining of this section, we carry out the details of the above arguments. 

\subsection{On transversality}\label{subsection51}

Generically, there is no flow line connecting two distinct critical points with equal index. On the other hand, for a holomorphic Morse function $f$ on an $n$-dimensional complex manifold, critical points of $f$, viewed as critical points of ${\bf Re} f$, all have index $n$. Therefore, a BPS soliton, viewed as a negative gradient flow line of ${\bf Re} f$, is not transverse. For any solution $\rho: {\mb R} \to X$ to the equation	
\beq\label{eqn53}
\rho'(s) + \nabla f (\rho(s)) = 0,\ \lim_{s \to -\infty} \rho (s) = p,\ \lim_{s \to +\infty} \rho(s) = q,
\eeq
its linearization $D_\rho: W^{1, 2}(x^* TX) \to L^2(x^* TX)$ is not transverse; a necessary condition for having one solution is $
{\bf Im} f(p) = {\bf Im} f(q)$. 

We introduce the equation on pairs $\hat\rho = (h, \rho)$ where $h \in {\mb R}$ and $\rho: {\mb R} \to X$,
\beq\label{eqn54}
\rho'(s) + \nabla f(\rho(s)) - h J \nabla f( \rho (s)) = 0,\ \lim_{s \to -\infty} \rho (s) = p,\ \lim_{s \to +\infty} \rho (s) = q.
\eeq
If $(h, \rho)$ solves the solution, then $\rho$ is a negative gradient flow line of ${\bf Re} (f + h \i f)$, hence
\beqn
{\bf Im} ( f(p) + h \i f(p)) = {\bf Im} ( f(q) + h\i f(q)),
\eeqn
which implies that $h = 0$. Therefore, solutions to \eqref{eqn53} and solutions to \eqref{eqn54} are in one-to-one correspondence, by identifying $\rho$ with $\hat\rho = (0, \rho)$.

\begin{defn}\label{defn51}
A solution $\rho$ to \eqref{eqn53} is called {\it maximally transverse} (with respect to the Hermitian metric) if the corresponding linearization $\hat{D}_{\hat{\rho}}$ of \eqref{eqn54} is surjective, or equivalent, ${\bf Coker} D_\rho$ is one-dimensional (spanned by $J \nabla f(\rho)$).
\end{defn}

Now we specialize to the equation \eqref{eqn51}. Consider the Banach manifold ${\mc B}$ of $W^{1, 2}_{\rm loc}$-maps from ${\mb R}$ to $\tilde X$ which are asymptotic to $\upsilon$ and $\kappa$ at $+\infty$ and $-\infty$ respectively, both in a $W^{1, 2}$-sense. Consider the bundle ${\mc E}\to {\mc B}$ whose fibre over $\rho = (x_\rho, p_\rho) \in {\mc B}$ is
\beqn
{\mc E}_\rho = L^2(\rho^* T\tilde{X}) = L^2( x_\rho^* TX \oplus {\mb C}).
\eeqn
There is a family of sections
\beqn
{\mc F}_{\tilde{W}}^\lambda(\rho) = \left[ \begin{array}{c} \dot{x}_\rho + \ov{p}_\rho \nabla Q(x_\rho) + \nabla F(x_\rho) \\ \dot{p}_\rho + \lambda^2 \left[ \ov{(Q(x_\rho) - a)}\right]
\end{array}  \right].
\eeqn
Its linearization at $\rho$ reads
\beqn
D_\rho^\lambda (V, h) = \left[ \begin{array}{c} \nabla_t V + \nabla_V ( \ov{p} \nabla Q + \nabla F) + \ov{h} \nabla Q \\
                       h' + \lambda^2 dQ \cdot V
\end{array} \right].
\eeqn

Denote
\begin{align*}
&\ \tilde{\mc S}_{\tilde{W}}^\lambda(\upsilon, \kappa):= ( {\mc F}_{\tilde{W}}^\lambda)^{-1}(0),\ &\ {\mc S}_{\tilde{W}}^\lambda(\upsilon, \kappa) = \tilde{\mc S}_{\tilde{W}}^\lambda(\upsilon,\kappa)/{\mb R};
\end{align*}
\begin{align*}
&\ \tilde{\bm {\mc S}}_{\tilde{W}}:= \bigsqcup_{\lambda \geq 1} \{\lambda \} \times \tilde{\mc S}_{\tilde{W}}^\lambda (\upsilon, \kappa),\ &\ {\bm {\mc S}}_{\tilde{W}}:= \bigsqcup_{\lambda \geq 1} \{\lambda \} \times {\mc S}_{\tilde{W}}^\lambda (\upsilon, \kappa).
\end{align*}
$\tilde{\bm {\mc S}}_{\tilde{W}}$ can be viewed as the zero locus of ${\bm {\mc F}}_{\tilde{W}}: [1, +\infty) \times {\mc B} \to {\mc E}$ with ${\bm {\mc F}}_{\tilde{W}}(\lambda, \rho) = {\mc F}_{\tilde{W}}^\lambda(\rho)$. It is not hard to show that by perturbing the Hermitian metric on $X$, one can assume the following:
\begin{enumerate}

\item For any $\tilde\rho: = (\lambda, \rho) \in \tilde{\bm {\mc S}}_{\tilde{W}}$, the linearization of ${\bm {\mc F}}_{\tilde{W}}$ at $\tilde \rho$ has one-dimensional cokernel.

\item Each $\rho \in \tilde{\mc S}_{\tilde{W}}^1$ is maximally transverse.

\item Each $y \in \tilde{\mc S}_{F^a} (\upsilon, \kappa)$ is maximally transverse w.r.t. the induced metric on $Q^a$.
\end{enumerate}

\subsection{Adiabatic limit}\label{subsection52}

We first need to establish the compactness of $\tilde{\bm{\mc S}}_{\tilde{W}}$ as $\lambda \to +\infty$. Recall that $\star\in X$ is the unique critical point of $Q$. Therefore there is a well-defined function $q: X \setminus \{ \star\} \to {\mb C}$ defined by
\beqn
\big[ q(x) dQ(x)	+ dF (x)\big] \cdot \nabla Q(x) = 0.
\eeqn
Notice that for $x \in Q^a$,
\beqn
\ov{q(x)} \nabla Q + \nabla F(x) = \nabla F^a(x) \in T_x Q^a.
\eeqn

\begin{prop}\label{prop52}
Let $\lambda_i$ be a sequence of positive numbers such that $\lim \lambda_i = +\infty$. Suppose $\tilde{x}_i(s) = (x_i(s), p_i(s))$ is a sequence of solutions to \eqref{eqn51} with $\lambda = \lambda_i$. Then there is a subsequence (still indexed by $i$), a sequence of numbers $s_i\in {\mb R}$, and a solution to 
\beqn
\dot{x}(s) + \nabla F^a ( x(s)) = 0,\ x: {\mb R} \to Q^a,\ \lim_{s \to -\infty} x (s) = \upsilon,\ \lim_{s \to +\infty} x (s) = \kappa
\eeqn
such that the following conditions hold.
\begin{enumerate}
\item $x_i(s + s_i)$ converges to $x (s)$ uniformly on compact subsets of ${\mb R}$.

\item $p_i(s + s_i)$ converges to $q( x(s))$ uniformly .
\end{enumerate}
\end{prop}

\begin{proof}
The key point is the $C^0$-bound, i.e., there exists a compact subset $\tilde{K} \subset \tilde{X}$ such that $\tilde{x}_i(s) \in \tilde{K}$. Once this is proved, this proposition can be proved in the same way as in \cite{Lagrange_multiplier}. Let $3 d_a = {\rm dist}(\star, Q^a)>0$. First, we claim that for $i$ large enough, $x_i(s) \notin B_{d_a}(\star)$ for all $s\in {\mb R}$. We prove this claim by contradiction. Suppose $d( \star, x_i(s_i)) \leq d_a$.  Let $[s_i- a_i, s_i + b_i]$ be the maximal interval containing $s_i$ such that $x_i([s_i - a_i, s_i + b_i]) \subset B_{2 d_a}(\star)$. Then we have
\beqn
d_a \leq {\rm dist}(x_i(s_i - a_i), x_i(s_i)) \leq \int_{s_i - a_i}^{s_i} | \dot x_i (s)| ds  \leq \sqrt{a_i} \| \dot x_i \|_{L^2} \leq \sqrt{a_i E}.
\eeqn
Here $E$ is the energy of the solution $\tilde{x}_i$ defined by \eqref{eqn52}. Therefore, $a_i \geq d_a^2/ E$. Similarly $b_i \geq d_a^2/E$. On the other hand, there is $\epsilon_a>0$ such that $|Q(x_i(s)) - a| \geq \epsilon_a$ for all $s \in [s_i - a_i, s_i+ b_i]$. Then
\beqn
E \geq \lambda_i^2 \| Q(x_i) - a_i\|_{L^2({\mb R})}^2 \geq \lambda_i^2 \| Q(x_i) - a_i\|_{L^2([s_i - a_i, s_i + b_i])}^2 \geq \lambda_i^2 \epsilon_a^2 d_a^2 / E.
\eeqn
This is impossible for $i$ sufficiently large. Therefore the claim is proved.

Second, by the definition of $q(x)$, the properness of $\nabla Q$ (see Hypothesis \ref{hyp21}) and the boundedness of $\nabla F$ (see Hypothesis \ref{hyp25}), it is easy to see that $q$ and $\nabla q$ are uniformly bounded away from $B_{d_a}(\star)$. By \eqref{eqn51}, we have
\beq\label{eqn55}
\begin{split}
0 &\ = \frac{d}{ds} Q(x_i) - dQ(x_i) \cdot \dot x_i \\
&\ = \frac{d}{ds} Q(x_i) + dQ \cdot \Big[ \ov{p_i} \nabla Q(x_i) + \nabla F(x_i) \Big] \\
&\ = \frac{d}{ds} Q(x_i) + dQ \cdot \Big[ \ov{( p_i - q( x_i)  )} \nabla Q(x_i) \Big] \\
&\ = \frac{d}{ds} Q(x_i) + \lambda_i |\nabla Q(x_i)|^2 \Big[ \lambda_i^{-1} \ov{( p_i - q( x_i)  )} \Big] .
\end{split}
\eeq
We also have
\beq\label{eqn56}
\begin{split}
&\ \frac{d}{ds} \Big[ \lambda_i^{-1} \big( p_i - q(x_i)  \big) \Big] + \lambda_i \Big[ Q(x_i) - a \Big] \\
&\ = \lambda_i^{-1} \Big[ \dot p_i - dq \cdot \dot x_i \Big] + \lambda_i \Big[ Q(x_i) - a \Big]\\
&\ = - \lambda_i^{-1} dq \cdot \dot x_i.
\end{split}
\eeq

Consider the operator
\beqn
\begin{split}
D_{\lambda_i}: &\ W^{1, 2}({\mb R}, {\mb C}^2) \to L^2({\mb R}, {\mb C}^2)\\
           &\ (f_1, f_2) \mapsto \left( \frac{ df_1}{ds} + \lambda_i |\nabla Q(x_i)|^2 \ov{f_2} ,\ \frac{df_2}{ds} + \lambda_i f_1 \right).
\end{split}
\eeqn
Notice that $|\nabla Q(x_i)|$ is uniformly bounded from below. Then as an unbounded operator from $L^2$ to $L^2$, $D_{\lambda_i}$ is bounded from below by $c \lambda_i$ for some constant $c>0$. On the other hand, it has a right inverse from $L^2$ to $W^{1, 2}$ which is uniformly bounded. Hence by \eqref{eqn55}, \eqref{eqn56}, we have
\beqn
\| Q(x_i) - a \|_{L^2({\mb R})} + \lambda_i^{-1} \| p_i - q(x_i)\|_{L^2({\mb R})} \leq c \lambda_i^{-2} \| dq \cdot \dot x_i\|_{L^2({\mb R})} \leq c \lambda_i^{-2} \| \dot x_i \|_{L^2({\mb R})}.
\eeqn
\beq\label{eqn57}
\| Q(x_i) - a\|_{W^{1, 2}({\mb R})} + \lambda_i^{-1} \|p_i - q(x_i)\|_{W^{1, 2}({\mb R})} \leq c \lambda_i^{-1} \| \dot x_i \|_{L^2({\mb R})}.
\eeq
In particular, $\|p_i\|_{L^\infty}$ and $\| \dot p_i \|_{L^2}$ are uniformly bounded.

Now we prove that $x_i(s)$ is uniformly bounded. Consider the moment map $\mu': X \to i {\mb R}$ of the $G'$-action. For some constant $c>0$, we have
\beqn
\begin{split}
\frac{d^2}{ds^2} \i \mu' ( x_i) &\ = \partial_s \big\langle \nabla \i \mu' (x_i), \dot x_i \big\rangle\\
                                    &\ = \partial_s \big\langle \nabla \i \mu' (x_i), - \ov{p_i} \nabla Q(x_i) - \nabla F(x_i) \big\rangle\\
																		&\ = -\partial_s \Big[ {\bf Re} \big( r W(\tilde{x}_i) + F_b (x_i) \big) \Big]\\
																		&\ = - \partial_s \Big[ {\bf Re} \big( r \tilde{W}(\tilde{x}_i) + F_b(x_i) - r F (x_i) - r ap_i\big) \Big]\\
																		&\ = -r \nabla^{\lambda_i} \tilde{W} (\tilde{x}_i) \cdot \dot{\tilde x}_i - \nabla (F_b - r F ) \cdot \dot x_i + r a {\bf Re} (\dot p_i) \\
																		&\ \geq r \big| \dot{\tilde x}_i \big|_{\lambda_i}^2 - c \big| \dot x_i \big| - c \big| \dot p_i \big|\\
																		&\ \geq - c \big| \dot x_i \big| - c \big| \dot p_i \big|.
\end{split}
\eeqn
Here $F_b= \langle \nabla (\i \mu'), \nabla F \rangle$ and the second last inequality follows from (1) of Hypothesis \ref{hyp25}, i.e., $F$ behaves like a linear function. Notice that $\dot x_i$ and $\dot p_i$ are both bounded in $L^2$. If there exists $s_i\in {\mb R}$ such that
\beqn
\lim_{i \to \infty} \i \mu' ( x_i(s_i)) = \lim_{i \to +\infty} \sup_{\mb R} \i \mu' (x_i) = +\infty,
\eeqn
then it is easy to derive that $\i \mu' (x_i( \cdot + s_i))$ diverges to infinity uniformly on compact sets. However, if this is true, then by (3) of Hypothesis \ref{hyp25} (which says $\nabla \tilde{W}$ is uniformly bounded from below outside a compact subset), the energy of $\tilde{x}_i$ can be infinitely large, which is impossible. Hence $\i \mu' (x_i)$ is uniformly bounded. Since $\mu'$ is proper (see Hypothesis \ref{hyp21}), $x_i$ is uniformly bounded.

It follows from \eqref{eqn57} that $x_i$ is in a small neighborhood of $Q^a$. Therefore, one can write $x_i(s) = (\ov{x}_i(s), Q(x_i(s)))$ with $\ov{x}_i(s) \in Q^a$. Moreover, for some $c>0$,
\beqn
\big| \dot{\ov{x}}_i (s) + \nabla F_a(\ov{x}_i(s)) \big| = \big| \dot{\ov{x}}_i (s) +  q(\ov{x}_i(s)) \nabla Q(\ov{x}_i(s)) + \nabla F(\ov{x}_i(s)) \big| \leq c \big| Q(x_i(s)) - a \big|
\eeqn
which converges uniformly to zero. Hence a subsequence of $\ov{x}_i$ (still indexed by $i$) converges (modulo reparametrization) to a flow line of $F_a$. Moreover, the flow line cannot break because it must connect $\upsilon$ and $\kappa$. This finishes the proof of Proposition \ref{prop52}.
\end{proof}

Now we define a compactification of ${\bm {\mc S}}_{\tilde{W}}$, i.e.,
\beqn
\ov{ {\bm {\mc S}}_{\tilde{W}} }:= {\bm {\mc S}}_{\tilde{W}} \sqcup \Big[  \{ +\infty\} \times {\mc S}_{F^a} (\upsilon, \kappa) \Big].
\eeqn
We say that a sequence $[\tilde\rho_i] = [\lambda_i, \rho_i] \in {\bm {\mc S}}_{\tilde{W}}$ converges to $[+\infty, y(s)]$ if up to translation, the two conditions of Proposition \ref{prop52} hold.

Lastly, we need construct a boundary chart near the $+\infty$ side of $\ov{ {\bm {\mc S}}_{\tilde{W}} }$. We have
\begin{prop}\label{prop53}
Suppose the negative gradient line $y: {\mb R} \to Q^a$ is maximally transverse. Then there exist $\Lambda = \Lambda_y >0$, and a continuous map
\beqn
\Phi_y: [\Lambda, +\infty] \to \ov{ {\bm {\mc S}}_{\tilde{W}}}
\eeqn
which is a homeomorphism onto a neighborhood of $[ +\infty, y(s)]$.
\end{prop}

This proposition will be proved shortly. It puts a boundary chart on $\ov{ {\bm {\mc S}}_{\tilde{W}} }$ and hence implies an oriented cobordism from ${\mc N}(\upsilon, \kappa)$ to ${\mc S}_{F^a}(\upsilon, \kappa)$. Therefore,
\beq\label{eqn58}
\# {\mc N}(\upsilon, \kappa) (\upsilon, \kappa) = \# {\mc S}_{F^a}( \upsilon, \kappa).
\eeq
This finishes the proof of Proposition \ref{prop48}.

\subsection{Proof of Proposition \ref{prop53}}

Choose a small $\epsilon>0$. Let $Q^{a, \epsilon}$ be the $\epsilon$-neighborhood of $Q^a$ inside $X$. Let $L \subset TX|_{Q^{a, \epsilon}}$ be the rank two bundle spanned by $\nabla Q$ and $J\nabla Q$; let $L^\bot$ be its orthogonal complement. Then for those $\rho \in {\mc B}$ which are contained in $Q^{a, \epsilon}$, we can decompose the domain and the target space of the linearization $D_\rho^\lambda$ as
\beq\label{eqn59}
\begin{array}{rcccl}
{\mc T}_\rho {\mc B} & \simeq & W_L(\rho) \oplus W_T(\rho) & = & W^{1, 2}( x^* L \oplus {\mb C}) \oplus W^{1, 2}( x^* L^\bot),\\
{\mc E}_\rho & \simeq & {\mc E}_L(\rho) \oplus {\mc E}_T (\rho) & = & L^2( x^* L \oplus {\mb C}) \oplus L^2 (x^* L^\bot).
\end{array}
\eeq
We rescale the norms on $W_L(\rho)$ and ${\mc E}_L(\rho)$ as follows. We identify $(h_1 \nabla Q, h_2)\in W_L(\rho)$ with $(h_1, h_2) \in W^{1, 2} \oplus W^{1, 2}$ and define
\beq\label{eqn510}
\| (h_1, h_2)\|_{W_\lambda} = \lambda	 \| h_1 \|_{L^2} + \| h_1'\|_{L^2} + \| h_2 \|_{L^2} + \lambda^{-1} \| h_2'\|_{L^2}.
\eeq
For $(h_1 \nabla Q, h_2) \in {\mc E}_L(\rho)$, we define
\beq\label{eqn511}
\| (h_1, h_2)\|_{L_\lambda} = \| h_1 \|_{L^2} + \lambda^{-1} \|h_2\|_{L^2}.
\eeq
The norms on the tangential components are unchanged, and we use $W_\lambda$ and $L_\lambda$ to denote the modified norms on $T_\rho{\mc B}$ and ${\mc E}_\rho$ respectively.

Let $y: {\mb R} \to Q^a$ be a negative gradient line of $F^a$. Then the linearization along $y$ reads
\beqn
D_y: W^{1, 2}( y^* TQ^a) \to L^2( y^* TQ^a),\ D_y (\xi_y) = \nabla_s \xi_y + \nabla_{\xi_y} \left( \nabla F + \ov{q( y)} \nabla Q \right).
\eeqn
We define
\beqn
D_y^+: {\mb R} \oplus W^{1, 2}( y^* TQ^a) \to L^2( y^* TQ^a),\ D_y^+ (a, \xi_y) = D_y(\xi_y) - aJ \left( \nabla F + \ov{q( y) }\nabla Q \right).
\eeqn
Since $y$ is maximally transverse, $D_y^+$ is surjective.

Now we define $\rho_y \in {\mc B}$ by $\rho_y (s) = ( y(s), q(y(s)))$. This is going to be our approximate solution to \eqref{eqn51} for $\lambda$ large. To apply the implicit function theorem, we have a few estimates to make. First, by straightforward calculation,
\beq\label{eqn512}
\| {\mc S}^\lambda( \rho_y) \|_{L_\lambda} = \lambda^{-1} \| dq \cdot y' \|_{L^2}.
\eeq
Let $D_{\rho_y}^\lambda: {\mc T}_{\rho_y} {\mc B} \to {\mc E}_{\rho_y}$ be the linearization along $\rho_y$. For $\xi_y \in W_T (\rho_y) \subset {\mc T}_{\rho_y} {\mc B}$, we have
\beqn
D_{\rho_y}^\lambda( \xi_y )  = \left( \nabla_t \xi_y + \nabla_{\xi_y} \left( \nabla F + \ov{q(y)} \nabla Q \right), \ \lambda^2 dQ \cdot \xi_y \right) = ( D_y(\xi_y), 0).
\eeqn
Hence with respect to the decompositions in \eqref{eqn59}, $D_{\rho_y}$ can be written as
\beqn
D_{\rho_y}^\lambda = \left[ \begin{array}{cc} D_y & A_y \\ 0 & D_y'
\end{array} \right]
\eeqn
where for $H = (h_1 \nabla Q, h_2) \in W_L(\rho_y)$, $A_y ( H )$ is real linear in $h_1$. Therefore, by \eqref{eqn510} \eqref{eqn511},
\beqn
\| A_y (h_1 \nabla Q, h_2) \|_{L_\lambda} \leq c_1 \| h_1 \|_{L^2} \leq c_1 \lambda^{-1} \| H \|_{W_\lambda}.
\eeqn

On the other hand, it is easy to see that $D_y': W_L(\rho_y) \to {\mc E}_L(\rho_y)$ has an inverse $Q_y'$ whose operator norm with respect to the modified Banach norms \eqref{eqn510} and \eqref{eqn511} is bounded by a number independent of $\lambda$. Therefore, we construct the operator
\beqn
Q_{\rho_y}^+ = \left[ \begin{array}{cc} Q_y^+ & 0 \\ 0 & Q_y'
\end{array} \right]: \big( {\mb R} \oplus {\mc E}_T( \rho_y )\big) \oplus {\mc E}_T(\rho_y) \to W_T(\rho_y) \oplus W_L( \rho_y).
\eeqn
Moreover, for certain constant $c>0$,
\beqn
\big\| {\rm Id} - D_{\rho_y}^+ Q_{\rho_y}^+  \big\|_\lambda  = \big\| Q_y^+ A_y \big\|_\lambda \leq c \lambda^{-1}.
\eeqn
Therefore, for $\lambda$ sufficiently large, $Q_{\rho_y}^+$ is approximately a right inverse to $D_{\rho_y}^+$.

Thirdly, we need to estimate the variation of the linearized operator near $\rho_y$. One can give a local chart of the Banach manifold ${\mc B}$ near $\rho_y$ and trivialize the bundle ${\mc E} \to {\mc B}$ over this chart as follows. Using the identification $Q^{a, \epsilon} \simeq Q^a \times B_\epsilon$, for any $\xi = (\xi_y, h_1 \nabla Q, h_2) \in {\mc T}_{\rho_y} {\mc B}$ with small norm, we identify it with the map $\Phi^B (\xi):= ( \ov\exp_y \xi_y, h_1, q(y) + h_2)$ into $Q^a \times B_\epsilon \times {\mb C}$, where $\ov\exp$ is the exponential map inside $Q^a$. We trivialize ${\mc E}$ over the image of $\Phi^B$. Let $D_{\xi}^{\lambda,+}: {\mb R} \oplus {\mc T}_{\rho_y} {\mc B} \to {\mc E}_{\rho_y}$ be the linearization of $\Phi^E \circ {\mc F}_{\tilde{W}}^{\lambda, +} \circ \Phi^B$ at $\xi \in {\mc T}_{\rho_y} {\mc B}$.
\begin{lemma}
There exist, $\epsilon, c, \Lambda>0$ such that for $\lambda \geq \Lambda$ and $\| \xi \|_{W_\lambda} \leq \epsilon$, we have
\beqn
\big\| D_\xi^{\lambda, +} - D_{\rho_y}^{\lambda, +} \big\| \leq c \| \xi \|_{W_\lambda}.
\eeqn
\end{lemma}

\begin{proof}
Similar to \cite[Lemma 9]{Lagrange_multiplier} and the detail is left to the reader.
\end{proof}

Therefore, one can apply the implicit function theorem and one derives that for each sufficiently large $\lambda$, there is a unique $\xi_\lambda^+ = (a_\lambda, \xi_\lambda) \in {\bf Image}( Q_{\rho_y}^+) \subset {\mb R} \oplus {\mc  T}_{\rho_y} {\mc B}$ such that
\beqn
\big( a_\lambda, \exp_{\rho_y} \xi_\lambda \big) \in \big( {\mc F}_{\tilde{W}}^{\lambda,+} \big)^{-1}(0).
\eeqn
We also know that $a_\lambda = 0$. Hence it defines a map $\Psi_y^+: [\Lambda, +\infty] \to \ov{ {\bm {\mc S}}_{\tilde{W}} }$ by $\lambda \mapsto [ \lambda, \Phi^B( \xi_\lambda) ]$. Moreover, one can show that it is a homeomorphism onto a neighborhood of $[+\infty, y(s)]$ inside $\ov{ {{\bm {\mc S}}}_{\tilde{W}}} $. The details are left to the reader. Hence $\Psi_y^+$ provides a boundary chart on $\ov{ {{\bm {\mc S}}_{\tilde{W}}} } $ and we have finished the proof of Proposition \ref{prop53}.

\section{Proof of Proposition \ref{prop47}}\label{section6}

In this section we prove the Picard--Lefschetz formula (Proposition \ref{prop47}). We think it is necessary to have a detailed discussion of orientations, which determines the sign of the bifurcation term appearing in the Picard--Lefschetz formula and the wall-crossing formula (Theorem \ref{thm46}). Our discussion will still be used in the companion paper \cite{Tian_Xu_4}.

\subsection{Linear algebra}

First we need to declare our rule of orienting operators. Let $F: X \to Y$ be a Fredholm operator between Banach spaces $X$, $Y$. The determinant line of $F$ is defined as the real one-dimensional vector space
\beqn
\det F := \det {\bf Ker} F \otimes (\det {\bf Coker} F)^\vee.
\eeqn
An orientation of $F$ is a homotopy class of trivializations of this space.

More generally, for any continuous family of Fredholm operators ${\mc F} = (F_a)_{a \in A}$ where $A$ is a topological space, there is a determinant line bundle over $A$. If $A$ is contractible, then $\det {\mc F}$ is trivial and a trivialization $\varphi_A: \det {\mc F} \to A \times {\mb R}$ can be easily write down if ${\bf dim} ( {\bf Ker} F_a )$ is a constant for $a \in A$. In this case, we say $\theta_1 \in \det F_{a_1}$ and $\theta_2 \in \det F_{a_2}$ are {\it in the same orientation} in $\det {\mc F}$ if the second factors of $\varphi_A (\theta_1), \varphi_A(\theta_2)$ are of the same sign. 

It becomes complicated when the dimensions of the kernels jump. For a Fredholm operator $F: X \to Y$, let $F^{(k)}_0 : {\mb R}^k \oplus X \to {\mb R}^k \oplus Y$ be the operator $F^{(k)}_0 ( v, x)  = (0, F(x))$. We define $\psi_0^{(k)}: \det F \to \det F^{(k)}_0$ by
\beqn
\psi_0^{(k)} \Big( \bigwedge_{i=1}^{m} x_{i} \wedge \bigwedge_{j=1}^{n} y_{j}^* \Big) = \Big(	 \bigwedge_{l=1}^{k} (e_l, 0) \wedge \bigwedge_{i=1}^{m} (0, x_{i}) \Big) \otimes \Big( \bigwedge_{l=1}^{k} (e^l, 0) \wedge \bigwedge_{j=1}^{n} (0, y_{j}^*) \Big).
\eeqn
Here $x_1, \ldots, x_m$ is a basis of ${\bf Ker} F$, $y_1, \ldots, y_n$ is a basis of ${\bf Coker} F$ and $y_1^*, \ldots, y_n^*$ is the dual basis; $e_1, \ldots, e_k$ is a basis of ${\mb R}^k$ and $e^1, \ldots, e^k$ is its dual basis. We also define $F_1^{(k)}(v, x) = (v, F(x))$ and define $\psi_1^{(k)}: \det F \to \det F^{(k)}_1$ by
\beqn
\psi_1^{(k)} \Big( \bigwedge_{i=1}^{m} x_{i} \wedge \bigwedge_{j=1}^{n} y_{j}^* \Big) = \bigwedge_{i=1}^{m} (0, x_{i}) \otimes \bigwedge_{j=1}^{n} (0, y_{j}^*).
\eeqn
The basic convention is that we regard the orientation of $\theta \in \det F$ as ``the same'' as those of $\psi_0^{(k)}(\theta) \in  F_0^{(k)}$ and $\psi_1^{(k)}(\theta) \in \det F_1^{(k)}$.

Further, let ${\mc G}^{(k)}(X, Y)$ be the space of linear maps $G: {\mb R}^k \oplus X \to {\mb R}^k \oplus Y$ of the form
\beqn
G = \left[ \begin{array}{cc} G_1 & G_2 \\ G_3 & 0 \end{array} \right].
\eeqn
Then let $F_G: {\mb R}^k \oplus X \to {\mb R}^k \oplus Y$ be $F_0^{(k)} + G$. This gives a family of Fredholm operators over the contractible space ${\mc G}^{(k)}(X, Y)$ and $\det F_G$ is trivial over ${\mc G}^{(k)}(X, Y)$. The $F_0^{(k)}$ and $F_1^{(k)}$ above are obtained in this way by special elements in ${\mc G}^{(k)}(X, Y)$
\beqn
G_0^{(k)} = 0,\ G_1^{(k)} = \left[ \begin{array}{cc} I_k & 0 \\ 0 & 0 \end{array}\right] 
\eeqn

Given a continuous family of Fredholm operators $F_T: X \to Y$, $T\in [0, 1]$ such that 
\beqn
k = {\bf dim} \big( {\bf Ker} F_0 \big) - {\bf dim} \big( {\bf Ker} F_1 \big) = \max \big\{ {\bf dim} \big( {\bf Ker} F_{T_1} \big) - {\bf dim}  \big( {\bf Ker} F_{T_2} \big) \ |\ 0 \leq T_1, T_2 \big\},
\eeqn
we can extend the family to a family
\beqn
F_{T, G}: {\mb R}^k \oplus X \to {\mb R}^k \oplus Y,\ T \in [0, 1],\ G \in {\mc G}^{(k)} (X, Y).
\eeqn
Two nonzero elements $\theta_1 \in \det F_1$ and $\theta_0 \in \det F_0$ are called {\it in the same orientation} if there is a curve $G: [0, 1] \to {\mc G}^{(k)}(X, Y)$ such that the following conditions hold.
\begin{enumerate}

\item $G(0) = G_1^{(k)}$, $G(1) = G_0^{(k)}$ and ${\bf dim} ( {\bf Ker} F_{T, G(T)} )$ is a constant.

\item For the family ${\mc F} = \{F_{T, G(T)}\ |\ T \in [0,1]\}$, the elements $\psi_1^{(k)}(\theta_0) \in \det F_{0, G(0)}$ and $\psi_0^{(k)}(\theta_1) \in \det F_{1, G(1)}$ are in the same orientation in $\det {\mc F}$.
\end{enumerate}
With in this section, we denote $\theta_1 \sim \theta_2$ if they are in the same orientation.

We treat the following special case under the above convention. 
\begin{lemma}\label{lemma61}
Let $F_T: X \to Y$, $T\geq 0$ be a family of Fredholm operators of the form $F_T = F + TP$ where $P: X \to Y$ is of rank one. Assume the following. 
\begin{enumerate}

\item $\{ v_1, \ldots, v_m, u \}$ is a basis of ${\bf Ker} F_0$, and $\{ w_1, \ldots, w_n, P(u) \}$ represents a basis of ${\bf Coker} F_0$, with dual basis $\{ w_1^*, \ldots, w_n^*, P(u)^* \}$.

\item For all $T> 0$, $\{ v_1, \ldots, v_m \}$ is a basis of ${\bf Ker} F_T$, $\{ w_1, \ldots, w_n \} \subset Y$ represents a basis of ${\bf Coker} F_T$, and $\{ w_1^*, \ldots, w_n^*\}$ is the dual basis of $({\bf Coker} F_T)^\vee$. 

\end{enumerate}
Denote
\beqn
\begin{split}
\theta_T: &\ = \bigwedge_{i=1}^m v_i \otimes \bigwedge_{j=1}^n w_i^* \in \det F_T,\ (T>0);\\
\theta_0: &\ = \Big( u \wedge \bigwedge_{i=1}^m v_i \Big) \otimes \Big( P(u)^* \wedge \bigwedge_{j=1}^n w_i^* \Big)\in \det F_0.
\end{split}
\eeqn
Then $\theta_1$ and $\theta_0$ are in the same orientation.
\end{lemma}

\begin{proof}
Without loss of generality, assume that $n = 0$. Namely, $F$ has one-dimensional cokernel spanned by $P(u)$. Define $G_{1, T}: {\mb R} \to {\mb R}$, $G_{2, T}: X \to {\mb R}$ and $G_{3, T}: {\mb R} \to Y$ by
\beqn
G_{1, T}(a ) = (1-T) a,\ \ G_{2, T} (x) = P(u)^* F_T(x),\ \ G_{3, T}(a) = (1-T) a P(u).
\eeqn
Define $F_{T, G(T)}:= F_T + G_{1, T} + G_{2, T} + G_{3, T}$. We can check that its kernel is spanned by $\{ ( Te, -(1-T) u) \} \cup \{ (0, v_i)\ |\ i = 1, \ldots, m\}$ and its cokernel is spanned by $(e^*, - P(u)^*)$, where $e \in {\mb R}$ is the standard basis and $e^*$ is its dual basis. Then
\beqn
\theta_{T, G(T)}:= \Big[ (Te, -(1-T) u) \wedge \bigwedge_{i=1}^m (0, v_i) \Big] \otimes \Big[ ( e^*, - P(u)^* ) \Big] \in \det F_{T, G(T)},\ T \in [0, 1].
\eeqn
is a nowhere-vanishing section of $\det F_{T, G(T)}$. Furthermore, notice that 
\beqn
F_{1, G(1)} = \left[ \begin{array}{cc}  0 & P(u)^*  P \\
                                        0 & F + P  \end{array} \right] 
\eeqn
can be connected by deforming the upper-right corner to $F_0^{(k)}$ without changing the dimension of the kernel and the element $\theta_{1, G(1)}$ has the same sign as
\beqn
\psi_0^{(k)} (\theta_1) = \Big[ e \wedge \bigwedge_{i=1}^m (0, v_i)\Big] \otimes \Big[ (e^*, 0) \Big] \in \det F_{1,G(1)};
\eeqn
and
\beqn
F_{0, G(0)} = \left[ \begin{array}{cc} 1 & 0 \\ P(u) & F 
\end{array} \right]
\eeqn
can be connected to $F_1^{(k)}$ by deforming the lower-left corner without changing the dimension of the kernel and the element $\theta_{0, G(0)}$ has the same sign as 
\beqn
\psi_1^{(k)}(\theta_0) = \Big[ (0, u) \wedge \bigwedge_{i=1}^m (0, v_i) \Big] \otimes \Big[ (0, P(u)^* ) \Big] \in \det F_{0, G(0)}.
\eeqn
Then by definition, $\theta_1$ and $\theta_0$ are in the same orientation.
\end{proof}

\subsection{Orientation of BPS solitons and Picard--Lefschetz formula}

We will prove the Picard--Lefschetz formula in a more general situation. Let $X$ be a Hermitian manifold and $F: X \to {\mb C}$ be a holomorphic Morse function. For $\upsilon \in {\bf Crit} F$, its unstable manifold $W^u_\upsilon$ can be identified with the solution space of the ODE
\beq\label{eqn61}
\dot{x} (s) + \nabla F (x (s)) = 0,\ s\in (-\infty, 0],\ \lim_{s \to -\infty} x (s) = \upsilon.
\eeq
It is an $N$-dimensional smooth manifold homeomorphic to an $N$-disk. Choosing an orientation on $W^u_\upsilon$ is the same as choosing, for each solution $\sigma_- (s)$ of \eqref{eqn61}, an orientation of the operator
\beqn
D_{\sigma_-}: W^{1, p}( (-\infty, 0], \sigma_-^* TX ) \to L^p( (-\infty, 0], \sigma_-^* TX),\ D_{\sigma_-}( \xi) = \nabla_s \xi + \nabla^2 F (\sigma_-) \cdot \xi.
\eeqn
Here $p \geq 2$ and the orientation of $D_{\sigma_-}$ is consistent among all choices of Sobolev norms. Similarly, consider $\kappa \in {\bf Crit} F$ and its stable manifold $W^s_\kappa$. It is the solution space of
\beq\label{eqn62}
\dot{x}(s) + \nabla F (x(s)) = 0,\  s\in [0, +\infty),\ \lim_{s \to +\infty} x(s) = \kappa.
\eeq
An orientation on $W^s_\kappa$ is the same as an orientation of the linearized operator of each solution $\sigma_+$ to \eqref{eqn62}, which reads
\beqn
D_{\sigma_+}: W^{1, p}([0, +\infty), \sigma_+^* TX ) \to L^p([0, +\infty), \sigma_+^* TX ),\ D_{\sigma_+} = \nabla_s \xi + \nabla^2 F (\sigma_+) \cdot \xi.
\eeqn

Suppose $\upsilon$ and $\kappa$ are nondegenerate, and ${\bf Im} F (\upsilon) = {\bf Im} F (\kappa)$, ${\bf Re} F (\upsilon) > {\bf Re} F (\kappa)$. Then the equation for BPS soliton connecting $\upsilon$ and $\kappa$ is	
\beqn
\dot{x}(s) +\nabla F (x(s)) = 0,\ s\in (-\infty, +\infty),\ \lim_{s \to -\infty} x(s) = \upsilon,\ \lim_{s \to +\infty} x(s) = \kappa.
\eeqn
Let $\tilde{\mc S}(\upsilon, \kappa)$ be the space of solutions and ${\mc S}(\upsilon, \kappa) = \tilde{\mc S}(\upsilon, \kappa)/ {\mb R}$. By choosing orientations on $W^u_\upsilon$ and $W^s_\kappa$ (independently), one obtains an orientation of the linearization $D_\sigma$ along any solution $\sigma(s)$ as follows. Let $\sigma_\pm = \sigma|_{{\mb R}_\pm}$. Consider the family of operators
\beqn
\begin{split}
\tilde{D}_\sigma^T: W^{1, p}((-\infty, 0], \sigma_-^* TX ) & \oplus W^{1, p}([0, +\infty), \sigma_+^* TX ) \to L^p({\mb R}, \sigma^* TX ) \oplus T_{\sigma(0)} X\\
              \tilde{D}_\sigma^T \big( \xi_-, \xi_+ \big) & = \big( D_{\sigma_-} \xi_- + D_{\sigma_+} \xi_+,\ T(\xi_- (0) - \xi_+ (0)) \big).
\end{split}
\eeqn

\begin{lemma}\label{lemma62}
For any $T>0$, there are canonical isomorphisms
\beq\label{eqn63}
{\bf Ker} \tilde{D}_\sigma^T \simeq {\bf Ker} D_\sigma,
\eeq
\beq\label{eqn64}
{\bf Coker} (\tilde{D}_\sigma^T)^\vee \simeq {\bf Image}( \tilde{D}_\sigma^T )^\bot \simeq {\bf Image} (D_\sigma)^\bot \simeq  {\bf Coker} (D_\sigma)^\vee.
\eeq
\end{lemma}

\begin{proof}
For $T>0$, $(\xi_-, \xi_+) \in {\bf Ker} \tilde{D}_\sigma^T$ if and only if $\xi_-$ and $\xi_+$ are the restrictions of some $\xi\in W^{1, p}({\mb R}, \sigma^* TX)$. This gives the isomorphism ${\bf Ker} \tilde{D}_\sigma^T \simeq {\bf Ker} D_\sigma$. Moreover, if $\eta \in {\bf Image} (D_\sigma)^\bot$, then for $\xi_- \in W^{1, p}( (-\infty, 0], \sigma_-^* TX)$ and $\xi_+ \in W^{1, p}([0, +\infty), \sigma_+^* TX)$,
\beqn
\langle D_{\sigma_\pm} \xi_\pm, \eta \rangle = \mp \langle \xi_\pm (0), \eta(0) \rangle.
\eeqn
Therefore $\eta \mapsto (T \eta, - \eta(0))$ induces an isomorphism ${\bf Image} (D_\sigma)^\bot \simeq {\bf Image} (\tilde{D}_\sigma^T)^\bot$.
\end{proof}

Since $T_{\sigma(0)} X$ is canonically oriented by the complex structure, Lemma \ref{lemma62} induces the following chain of isomorphisms
\begin{multline}\label{eqn65}
\det D_\sigma \simeq \det \tilde{D}_\sigma^T \simeq \det \tilde{D}_\sigma^0 \simeq \det ( D_{\sigma_-}, D_{\sigma_+}) \simeq \det ( {\bf Ker} D_{\sigma_-} \oplus {\bf Ker} D_{\sigma_+}) \\
\simeq  \det {\bf Ker} D_{\sigma_-} \otimes \det {\bf Ker} D_{\sigma_+} \simeq \det D_{\sigma_-} \otimes \det D_{\sigma_+}.
\end{multline}

\begin{rem}
If we regard $\sigma$ as a map from ${\mb R} \times S^1$ which is independent of the second coordinate, then an orientation of $D_\sigma$ also induces an orientation of the linearized operator over the cylinder. By the result of \cite{Floer_Hofer_orientation}, this means choosing orientations on the unstable and stable manifolds induces the so-called coherent orientations on the moduli space of solitons.
\end{rem}

It is convenient to assume that $X$ is K\"ahler and $\sigma$ is maximally transverse. This assumption doesn't affect the generality when discussing orientations. Then for each $\sigma \in \tilde{\mc S}(\upsilon, \kappa)$, $J\dot\sigma$ spans ${\bf Image} (D_\sigma)^\bot$ and we have a distinguished element
\beqn
\dot\sigma \otimes ( J \dot\sigma) \in \det D_\sigma.
\eeqn
\begin{defn}\label{defn64}
For $\sigma \in \tilde{\mc S}(\upsilon, \kappa)$, define ${\bf Sign}(\sigma) \in \{ \pm 1\}$ to be the sign of $\dot \sigma \otimes (J \dot \sigma)$ with respect to the orientation of $D_\sigma$ induced by \eqref{eqn65}. This makes ${\mc S}(\upsilon, \kappa)$ an oriented zero-dimensional manifold. Define
\beqn
\# {\mc S}(\upsilon, \kappa) = \sum_{[\sigma] \in {\mc S}(\upsilon, \kappa)}  {\bf Sign} (\sigma).
\eeqn
\end{defn}

Now consider a smooth family of holomorphic functions $F_\iota: X \to {\mb C}$ with $\iota \in [\iota_-, \iota_+]$. Let the critical points of $F_\iota$ be $\kappa_{k, \iota} \in {\bf Crit} F_\iota$. We make the following assumptions.
\begin{enumerate}

\item For each $\iota$, $F_\iota$ is Morse.

\item For $\iota \neq \iota_0 \in (\iota_-, \iota_+)$, $F_\iota$ is strongly regular, i.e., the imaginary parts ${\bf Im} F_{\iota}( \beta_{k, \iota})$ for all $k$ are distinct. 

\item At $\iota =\iota_0$, there are two critical points $\upsilon_{\iota_0}, \kappa_{\iota_0} \in {\bf Crit} F_{\iota_0}$
such that
\beqn
{\bf Im} F_{\iota_0} (\upsilon_{\iota_0} ) = {\bf Im} F_{\iota_0} (\kappa_{\iota_0} ),\ {\bf Re} F_{\iota_0} (\upsilon_{\iota_0} ) > {\bf Re} F_{\iota_0} (\kappa_{\iota_0} ),
\eeqn
\beqn
\ \left. \frac{d}{d\iota} \right|_{\iota = \iota_0} \Big[ {\bf Im} F_{\iota} (\upsilon_\iota) - {\bf Im} F_\iota (\kappa_\iota ) \Big] \neq 0.
\eeqn
\end{enumerate}

One can identify the unstable manifold of $\upsilon_\iota$ with the space of solutions to the ODE
\beqn
\dot{x}(s) + \nabla F_{\iota} (x(s)) = 0,\ \lim_{s \to -\infty} x(s) = \upsilon_\iota,\ s \in (-\infty, 0].
\eeqn
We choose an orientations of $W_{\upsilon_\iota}^u$ that depends continuously on $\iota$. Then 
\beqn
W_\upsilon^u = \displaystyle \bigsqcup_{\iota_-  \leq \iota \leq \iota_+} W_{\upsilon_\iota}^u
\eeqn
has an induced orientation. It is not compact because the flow lines can break at $\iota = \iota_0$. One can compactify it by adding broken ones. Denote the compactification by $\ov{ W_\upsilon^u }$.

\begin{prop}\label{prop65}
Suppose the moduli space of BPS solitons connecting $\upsilon_{\iota_0}$ and $\kappa_{\iota_0}$ is regular. Then the closure $\ov{ W_\upsilon^u }$ is an oriented manifold with boundary and the boundary is the disjoint union
\beqn
\Big[ W_{\upsilon_{\iota_+}}^u  \Big] \sqcup \Big[ - W_{\upsilon_{\iota_-}}^u \Big] \sqcup \Big[ - (-1)^{\tilde{F}} {\mc S} (\upsilon_{\iota_0}, \kappa_{\iota_0}) \times W_{\kappa_{\iota_0}}^u \Big].
\eeqn
Here ${\mc S}(\upsilon_{\iota_0}, \kappa_{\iota_0} )$ and $W_{\kappa_{\iota_0}}^u$ have their own orientations.
\end{prop}

Notice that ${\mc S}(\upsilon_{\iota_0}, \kappa_{\iota_0})$ is zero-dimensional. Hence this proposition implies the Picard--Lefschetz formula (Proposition \ref{prop47}). Proposition \ref{prop65} is proved in the next subsection. 

\subsection{Proof of Proposition \ref{prop65}}

Consider a Banach manifold $\tilde{\mc B}$ of pairs $(\iota, \psi)$ with $\iota \in (\iota_-, \iota_+)$ and $\psi \in W^{1, p}_{\rm loc}((-\infty, 0], X)$ in which $\psi (s)$ is asymptotic to $\upsilon_\iota$ in the $W^{1, p}$-sense. Consider the Banach bundle $\tilde{\mc E} \to \tilde{\mc B}$ whose fibre over $\tilde\psi = (\iota, \psi)$ is $L^p( (-\infty, 0], \psi^* TX)$. Let $\tilde{\mc F}: \tilde{\mc B} \to \tilde{\mc E}$ be the section $\tilde{\mc F}(\iota, \psi) = \dot\psi + \nabla F_\iota ( \psi)$. Then $W_\upsilon^u \simeq \tilde{\mc F}^{-1}(0)$ and the orientation on $W_\upsilon^u$ is given by the orientation on the linearized operator
\beqn
\tilde{D}_{\iota, \psi}: {\mb R} \times W^{1, p}({\mb R}, \psi^* TX) \to L^p({\mb R}, \psi^* TX).
\eeqn

We need to construct a local chart near any singular flow line. Let $\sigma \in \tilde{\mc S} (\upsilon_{\iota_0}, \kappa_{\iota_0} )$ and $b \in W^u_{\kappa_{\iota_0} }$. For $\iota - \iota_0$ sufficiently small, there are families of vectors
\beqn
\eta^\iota_- \in T_{\upsilon_{\iota_0} } X,\ \eta^\iota_+ \in T_{\kappa_{\iota_0} } X
\eeqn
such that
\beqn
\exp_{\upsilon_{\iota_0} } \eta^\iota_- = \upsilon_\iota,\ \exp_{\kappa_{\iota_0} } \eta^\iota_+ = \kappa_\iota.
\eeqn
For $S>0$ sufficiently large, we can extend $\eta_\pm^\iota$ to vector fields
\beqn
\tilde\eta_-^\iota \in \Gamma((-\infty, -S], \sigma^* TX),\ \tilde\eta_+^\iota \in \Gamma([S, +\infty), \sigma^* TX),
\eeqn
which are asymptotic to $\eta_\pm^\iota$ in the $W^{1, p}$-sense. Choosing cut-off functions $\beta_-$ (supported on $(-\infty, -S+1]$) and $\beta_+$ (supported in $[S-1, +\infty)$), define
\beqn
\tilde\eta^\iota = \beta_- \tilde\eta^\iota_- + \beta_+ \tilde\eta^\iota_+.
\eeqn
Denote
\beqn
\tilde\eta_\pm =  \left. \frac{d}{d\iota } \right|_{\iota = \iota_0} \tilde\eta_\pm^\iota,\ \tilde\eta =  \left. \frac{d}{d\iota} \right|_{\iota = \iota_0} \tilde\eta^\iota.
\eeqn
Then we have the linearized operator
\beqn
\begin{split}
\tilde{D}_\sigma: &\ {\mb R} \{ \partial_\iota \} \times W^{1, p}({\mb R}, \sigma^* TX ) \to L^p( {\mb R}, \sigma^* TX ),\\
               &\ \tilde{D}_\sigma(a \partial_\iota , \xi) = D_\sigma(\xi) + L_\sigma(a \partial_\iota )
\end{split}
\eeqn
where $L_\sigma(a \partial_\iota ) = D_\sigma ( a \tilde\eta )$. Since we assumed that $\sigma$ is maximally transverse, it is easy to see that $\tilde{D}_\sigma$ is surjective.

We can perform the gluing construction in the standard way. The following lemma is left to the reader.
\begin{lemma}\label{lemma66}
Given a singular trajectory represented by $\sigma \in \tilde{\mc S}(\upsilon_{\iota_0}, \kappa_{\iota_0})$ and $b \in W_{\kappa_{\iota_0}}^u$, there exists $\epsilon>0$, such that for each $t \in (0, \epsilon)$ and $\sigma' \in B_\epsilon( \sigma, \tilde{\mc S}(\upsilon_{\iota_0}, \kappa_{\iota_0} ))$, $b' \in B_\epsilon (b, W_{\kappa_{\iota_0} }^u)$, there is a solution
\beqn
\tilde\psi_{t, \sigma', b'} = (\iota_{t, \sigma', b'},  \psi_{t, \sigma', b'}) \in \tilde{\mc F}^{-1}(0).
\eeqn
Moreover, this family satisfies the following conditions.

\begin{enumerate}
\item The map $\Phi: (0, \epsilon) \times \{ \sigma\} \times B_\epsilon(b, W_{\kappa_0}^u) \to W_\upsilon^u$ defined by
\beq\label{eqn66}
\Phi (t, \sigma, b') := \tilde\psi_{t, \sigma, b'}
\eeq
extends continuously as $t \to 0$ to $(\sigma, b')$ and the extension 
\beqn
\Phi: [0, \epsilon) \times \{ \sigma\} \times B_\epsilon (b, W_{\kappa_0}^u ) \to \ov{ W_\upsilon^u}
\eeqn
is a homeomorphism onto a neighborhood of the singular object $(\sigma, b)$.

\item Fix $t\in (0, \epsilon)$. The map $\Psi: \{t\} \times B_\epsilon ( \sigma, \tilde{\mc S}(\upsilon_0, \kappa_0)) \times B_\epsilon( b, W_{\kappa_0}^u ) \to W_\upsilon^u$ defined by
\beqn
\Psi(t, \sigma', b') \mapsto \tilde\psi_{t, \sigma', b'}
\eeqn
is a homeomorphism onto a neighborhood of $\tilde\psi_{t, \sigma, b}$.

\item The orientations on $W_\upsilon^u$ induced from $\Phi$ and $\Psi$ are the same.
\end{enumerate}
\end{lemma}

We call the orientation on $W_\upsilon^u$ in Lemma \ref{lemma66} ``induced from the boundary chart.'' We need to see if it is consistent with the interior orientation. Indeed, define $\tilde{D}_\sigma^T = D_\sigma + T L_\sigma$ and we have isomorphisms
\beqn
\det \tilde{D}_\sigma \simeq \det \tilde{D}_\sigma^T \simeq \det \tilde{D}_\sigma^0.	
\eeqn
Using Lemma \ref{lemma61}, the direction represented by $(0, \dot\sigma) \in {\bf Ker} \tilde{D}_\sigma$ is identified with
\beqn
(\partial_\iota, 0) \wedge (0, \dot\sigma)  \otimes L_\sigma ( \partial_\iota, 0) \sim {\bf Sign} \big\langle L_\sigma( \partial_\iota), J \dot\sigma  \big\rangle\Big[ (\partial_\iota, 0) \wedge (0, \dot\sigma) \otimes ( J \dot \sigma)\Big] =: \theta^0 \in \det \tilde{D}_\sigma^0.
\eeqn
Choose a positive volume form $\omega_b$ on ${\bf Ker} D_b$ and a positive volume form $\omega_\psi$ on ${\bf Ker} D_\psi$ for any $(\iota, \psi)$ in the image of the $\Phi$ of \eqref{eqn66}. Since we have the canonical identification $\det D_\sigma \otimes \det D_b\simeq \det D_\psi$, $\theta^0 \wedge \omega_b \in \det \tilde{D}_\sigma^0 \otimes \det D_b$ is identified canonically with
\beqn
 {\bf Sign} \big\langle L_\sigma(\partial_\iota ), J \dot\sigma \big\rangle \cdot {\bf Sign} (\sigma) \cdot \Big[ ( \partial_\iota, 0) \otimes  \omega_\psi \Big] \in \det \tilde{D}_{\iota, \psi}.
\eeqn
Hence the oriented boundary of $\ov{ W_\upsilon^u }$ is
\beqn
\partial  W_\upsilon^u =  W_{\upsilon_{\iota_+} }^u  \sqcup \Big[ - W_{\upsilon_{\iota_-}}^u \Big]  \sqcup  \Big[ - {\bf Sign}\big\langle L_\sigma (\partial_\iota), J \dot\sigma \big\rangle \big( {\mc S}(\upsilon_{\iota_0}, \kappa_{\iota_0} ) \times W_{\kappa_{\iota_0} }^u \big) \Big].
\eeqn

Lastly, we need to evaluate the sign of $\langle L_\sigma(\partial_\iota), J \dot\sigma \rangle$. We assume that
\beqn
\left. \frac{d}{d\iota}\right|_{\iota = 0} \Big[ {\bf Im} F_\iota ( \upsilon_\iota ) - {\bf Im} F_\iota (\kappa_\iota ) \Big] > 0 \Longleftrightarrow (-1)^{\tilde{F}} = +1.
\eeqn
since the other case is exactly the opposite. Then by definition,
\beqn
\begin{split}
L_\sigma(\partial_\iota ) &\ = \nabla_s \tilde\eta + \nabla^2 F_{\iota_0} (\sigma) \tilde\eta\\
&\ = \nabla_s ( \beta_- \tilde\eta_- + \beta_+ \tilde\eta_+ ) + \nabla^2 F_{\iota_0} (\sigma) (\beta_- \tilde\eta_- + \tilde\beta_+ \eta_+) \\
&\ = \beta_-' \tilde\eta_- + \tilde\beta_+' \eta_+.
\end{split}
\eeqn
Moreover, we know that $J \dot\sigma(s) = - \nabla {\bf Im} F_{\iota_0} ( \sigma(s))$. Therefore,
\beqn
\langle L_\sigma(\partial_\iota), J \dot\sigma \rangle = \langle \beta_-' \tilde\eta_-, - \nabla {\bf Im} F_{\iota_0} \rangle + \langle \beta_+' \tilde\eta_+, - \nabla {\bf Im} F_{\iota_0} \rangle > 0.
\eeqn

\section{Topological Virtual Orbifolds and Virtual Cycles}\label{section7}

We recall the framework of constructing virtual fundamental cycles associated to moduli problems. Such constructions, usually called ``virtual technique'', has a long history since it first appeared in algebraic Gromov--Witten theory by \cite{Li_Tian_2}. The current method is based on the topological approach of \cite{Li_Tian}.	

\subsection{Topological manifolds and transversality}

In this subsection we review the classical theory about topological manifolds and (microbundle) transversality. 

\begin{defn} \label{defn71} {\rm (Topological manifolds and embeddings)}
\begin{enumerate}

\item A {\it topological manifold} is a second countable Hausdorff space $M$ which is locally homeomorphic to an open subset of ${\mb R}^n$. 

\item A subset $S \subset M$ is a {\it submanifold} if $S$ equipped with the subspace topology is a topological manifold. 

\item A map $f: N \to M$ between two topological manifold is called a {\it topological embedding} if $f$ is a homeomorphism onto its image. 

\item A topological embedding $f: N \to M$ is called {\it locally flat} if for any $p \in f(N)$, there is a local coordinate $\varphi_p: U_p \to {\mb R}^m$ where $U_p \subset {\mb R}^n$ is an open neighborhood of $p$ such that 
\beqn
\varphi_p (f(N) \cap U_p) \subset {\mb R}^n \times \{0\}.
\eeqn
\end{enumerate}
\end{defn}

In this paper, without further clarification, all embeddings of topological manifolds are assumed to be locally flat. In fact we will always assume (or prove) the existence of a normal microbundle which implies local flatness.

\subsubsection{Microbundles}

The discussion of topological transversality needs the concept of microbundles, which was introduced by Milnor \cite{Milnor_micro_1}.

\begin{defn}\label{defn72} {\rm (Microbundles)}
Let $B$ be a topological space. 
\begin{enumerate}

\item A {\it microbundle} over a $B$ is a triple $(E, i, p)$ where $E$ is a topological space, $i: M \to E$ (the zero section map) and $p: E \to M$ (the projection) are continuous maps, satisfying the following conditions.
\begin{enumerate}
\item $p \circ i = {\rm Id}_M$.

\item For each $b \in B$ there exist an open neighborhood $U \subset B$ of $b$ and an open neighborhood $V \subset E$ of $i(b)$ with $i(U) \subset V$, $j(V) \subset U$, such that there is a homeomorphism $V \simeq U \times {\mb R}^n$ which makes the following diagram commutes.
\beqn
\vcenter{ \xymatrix{  &  V \ar[dd] \ar[rd]^p & \\
                     U \ar[ru]^i \ar[rd]^i & & U \\
                      & U \times {\mb R}^n \ar[ru]^p & } }.
\eeqn
\end{enumerate}

\item Two microbundles $\xi = (E, i, p)$ and $\xi' = (E', i', p')$ over $B$ are {\it equivalent} if there are open neighborhoods of the zero sections $W \subset E$, $W' \subset E'$ and a homeomorphism $\rho: W \to W'$ which is compatible with the structures of the two microbundles.
\end{enumerate}
\end{defn}

Vector bundles and disk bundles are particular examples of microbundles. More generally, an {\it ${\mb R}^n$-bundle} over a topological manifold $M$ is a fibre bundle over $M$ whose fibres are ${\mb R}^n$ and whose structure group is the group of homeomorphisms of ${\mb R}^n$ which fix the origin. Notice that an ${\mb R}^n$-bundle has a continuous zero section, thus an ${\mb R}^n$-bundle is naturally a microbundle. A very useful fact, which was proved by Kister \cite{Kister_1964} and Mazur \cite{Mazur_1964}, says that microbundles are essentially ${\mb R}^n$-bundles.

\begin{thm} \label{thm73} {\rm (Kister--Mazur Theorem)} 
Let $B$ be a topological manifold (or a weaker space such as a locally finite simplicial complex) and $\xi = (E, i, p)$ be a microbundle over $B$. Then $\xi$ is equivalent to an ${\mb R}^n$-bundle, and the isomorphism class of this ${\mb R}^n$-bundle is uniquely determined by $\xi$. 
\end{thm}

However, ${\mb R}^n$-bundles are essentially different from vector bundles. For example, vector bundles always contain disk bundles, which is not true for ${\mb R}^n$-bundles. 

\begin{defn} \label{defn74} {\rm (Normal microbundles)}
Let $f: S \to M$ be a topological embedding.

\begin{enumerate}

\item  A {\it normal microbundle} of $f$ is a pair $\xi = (N, \nu)$ where $N \subset M$ is an open neighborhood of $f(S)$ and $\nu: N \to S$ is a continuous map such that together with the natural inclusion $S \hookrightarrow N$ they form a microbundle over $N$. A normal microbundle is also called a {\it tubular neighborhood}.

\item Two normal microbundles $\xi_1 = (N_1, \nu_1)$ and $\xi_2 = (N_2, \nu_2)$ are {\it equivalent} if there is another normal microbundle $(N, \nu)$ with $N \subset N_1 \cap N_2$ and 
\beqn
\nu_1|_N = \nu_2|_N = \nu.
\eeqn
An equivalence class is called a {\it germ} of normal microbundles (or tubular neighborhoods). 
\end{enumerate}
\end{defn}

For example, for a smooth submanifold $S \subset M$ in a smooth manifold, there is always a normal microbundle. Its equivalence class is not unique though, as we need to choose the projection map. 

\subsubsection{Transversality}

We first recall the notion of microbundle transversality. Let $Y$ be a topological manifold, $X \subset Y$ be a submanifold and $\xi =  (N, \nu)$ be a normal microbundle of $X$. Let $f: M \to Y$ be a continuous map. 

\begin{defn} \label{defn75} {\rm (Microbundle transversality)}
Let $Y$ be a topological manifold, $X \subset Y$ be a submanifold and $\xi_X$ be a normal microbundle of $X$. Let $f: M \to Y$ be a continuous map. We say that $f$ is {\it transverse to $\xi$} if the following conditions are satisfied.
\begin{enumerate}

\item $f^{-1}(X)$ is a submanifold of $M$.

\item There is a normal microbundle $\xi' = (N', \nu')$ of $f^{-1}(X) \subset M$ such that the following diagram commute
\beqn
\xymatrix{ N' \ar[d] \ar[r]^f &  N \ar[d] \\
           f^{-1}(X) \ar[r]   &  X }
\eeqn
and the inclusion $f: N' \to N$ induces an equivalence of microbundles.
\end{enumerate}
More generally, if $C \subset M$ is any subset, then we say that $f$ is transverse to $X$ {\it near} $C$ if the restriction of $f$ to an open neighborhood of $C$ is transverse to $X$. 
\end{defn}
It is easy to see that the notion of being transverse to $\xi$ only depends on the germ of $\xi$. 

\begin{rem}
The notion of microbundle transversality looks too restrictive at the first glance. For example, the line $x = y$ in ${\mb R}^2$ intersects transversely with the $x$-axis in the smooth category, however, the line is not transverse to the $x$-axis with respect to the natural normal microbundle given by the projection $(x,y) \to (x, 0)$. 
\end{rem}

The following theorem, which is of significant importance in our virtual cycle construction, shows that one can achieve transversality by arbitrary small perturbations. 

\begin{thm} \label{thm77} {\rm (Topological transversality theorem)} Let $Y$ be a topological manifold and $X \subset Y$ be a proper submanifold. Let $\xi$ be a normal microbundle of $X$. Let $C \subset D \subset Y$ be closed sets. Suppose $f: M \to Y$ is a continuous map which is microbundle transverse to $\xi$ near $C$. Then there exists a homotopic map $g: M \to Y$ which is transverse to $\xi$ over $D$ such that the homotopy between $f$ and $g$ is supported in an open neighborhood of $f^{-1}( (D \setminus C) \cap X)$. 
\end{thm}

\begin{rem}
The theorem was proved by Kirby--Siebenmann \cite{Kirby_Siebenmann} with a restriction on the dimensions of $M$, $X$ and $Y$. Then Quinn \cite{Quinn_1982} \cite{Quinn} \cite{Freedman_Quinn} completed the proof of the remaining cases. Notice that in \cite{Quinn}, the transversality theorem is stated for an embedding $i: M \to Y$ and the perturbation can be made through an isotopy. This implies the above transversality result for maps as we can identify a map $f: M \to Y$ with its graph $\tilde f: M \to M \times Y$, and an isotopic embedding of $\tilde f$, written as $\tilde g(x) = (g_1(x), g_2(x))$, can be made transverse to the submanifold $\tilde X = M \times X \subset M \times Y$ with respect to the induced normal microbundle $\tilde \xi$. Then it is easy to see that it is equivalent to $g_2: M \to Y$ being transverse to $\xi$. 
\end{rem}

In most of the situations of this paper, the notion of transversality is about sections of vector bundles. Suppose $f: M \to {\mb R}^n$ is a continuous map. The origin $0 \in {\mb R}^n$ has a canonical normal microbundle. Therefore, one can define the notion of transversality for $f$ as a special case of Definition \ref{defn75}. Now suppose $E \to M$ is an ${\mb R}^n$-bundle and 
\beqn
\varphi_U: E|_U \to U \times {\mb R}^n
\eeqn
is a local trivialization. Each section $s: M \to E$ induces a map $s_U: U \to {\mb R}^n$. Then we say that $s$ is transverse over $U$ if $s_U$ is transverse to the origin of ${\mb R}^n$. This notion is clearly independent of the choice of local trivializations. Then $s$ is said to be transverse if it is transverse over a sufficiently small neighborhood of every point of $M$. 

Notice that the zero section of $E$ has a canonical normal microbundle in the total space, and this notion of transversality for sections never agrees with the notion of transversality for graphs of the sections with respect to this canonical normal microbundle. Hence there is an issue about whether this transversality notion for sections behaves as well as the microbundle transversality. 

\begin{thm}\label{thm79}
Let $M$ be a topological manifold and $E \to M$ be an ${\mb R}^n$-bundle. Let $C\subset D \subset M$ be closed subsets. Let $s: M \to E$ be a continuous section which is transverse near $C$. Then there exists another continuous section which is transverse near $D$ and which agrees with $s$ over a small neighborhood of $C$. 
\end{thm}

\begin{proof}
The difficulty is that the graph of $s$ is not microbundle transverse to the zero section, hence we cannot directly apply the topological transversality theorem (Theorem \ref{thm77}). Hence we need to use local trivializations view the section locally as a map into ${\mb R}^n$. For each $p\in D$, choose a precompact open neighborhood $U_p  \subset M$ of $p$ and a local trivialization 
\beqn
\varphi_p: E|_{U_p} \simeq U_p \times {\mb R}^n.
\eeqn
All $U_p$ form an open cover of $D$. $M$ is paracompact, so is $D$. Hence there exists a locally finite refinement with induced local trivializations. Moreover, $D$ is Lindel\"of, hence this refinement has a countable subcover, denoted by $\{ U_i\}_{i=1}^\infty$. Over each $U_i$ there is an induced trivialization of $E$. 

We claim that there exists precompact open subsets $V_i \sqsubset U_i$ such that $\{ V_i\}_{i=1}^\infty$ still cover $D$. We construct $V_i$ inductively. Indeed, a topological manifold satisfies the $T_4$-axiom, hence we can use open sets to separate the two closed subsets 
\begin{align*}
&\ D \setminus \bigsqcup_{i \neq 1} U_i,\ &\ D \setminus U_1.
\end{align*}
This provides a precompact $V_i \sqsubset U_i$ such that replacing $U_1$ by $V_1$ one still has an open cover of $D$. Suppose we can find $V_1, \ldots, V_k$ so that replacing $U_1, \ldots, U_k$ by $V_1, \ldots, V_k$ still gives an open cover of $D$. Then one can obtain $V_{k+1} \sqsubset U_{k+1}$ to continue the induction. We see that $\{V_i\}_{i=1}^\infty$ is an open cover of $D$ because every point $p\in D$ is contained in at most finitely many $U_i$.

Now take an open neighborhood $U_C \subset M$ of $C$ over which $s$ is transverse. Since $M$ is a manifold, one can separate the two closed subsets $C$ and $M \setminus U_C$ by a cut-off function 
\beqn
\rho_C: M \to [-1, 1]
\eeqn
such that
\beqn
\rho_C^{-1}(-1) = C,\ \rho_C^{-1}(1) = M \setminus U_C. 
\eeqn
Define a sequence of open sets
\beqn
C^k = \rho_C^{-1} ([-1, \frac{1}{k+1})). 
\eeqn
Similarly, we can choose a sequence of shrinkings 
\beqn
V_i \sqsubset \cdots \sqsubset V_i^{k+1} \sqsubset V_i^k \sqsubset \cdots \sqsubset U_i.
\eeqn
Define 
\beqn
W^k = C^k \cup \bigcup_{i=1}^k V_i^k
\eeqn
which is a sequence of open subsets of $M$. 

Now we start an inductive construction. First, over $U_1$, the section can be identified with a map $s_1: U_1 \to {\mb R}^n$. By our assumption, $s_1$ is transverse over $U_C \cap U_1$. Then apply the theorem for the pair of closed subsets $\ov{C^1} \cap U_1 \subset (\ov{C^1} \cap U_1) \cup \ov{V_1^1}$ of $U_1$. Then one can modify it so that it becomes transverse near $(\ov{C^1} \cap U_1) \cup \ov{V_1^1}$, and the change is only supported in a small neighborhood of $\ov{V_1^1} \setminus \ov{C^1}$. This modified section still agrees with the original section near the boundary of $U_1$, hence still defines a section of $E$. It also agrees with the original section over a neighborhood of $\ov{C^2}$. Moreover, the modified section is transverse near
\beqn
\ov{W^1}: = \ov{C^1} \cup \ov{V_1^1}.
\eeqn

Now suppose we have modified the section so that it is transverse near 
\beqn
\ov{W^k}:= \ov{C^k} \cup \bigcup_{i = 1}^k \ov{V_i^k},
\eeqn
and such that $s$ agrees with the original section over an open neighborhood of $\ov{C^{k+1}}$. Then by similar method, one can modify $s$ (via the local trivialization over $U_{k+1}$) to a section which agrees with $s$ over a neighborhood of
\beqn
\ov{C^{k+1}} \cup \bigcup_{i=1}^k \ov{V_i^{k+1}}
\eeqn
and is transverse near $\ov{W^{k+1}}$. In particular this section still agrees with the very original section over $C^{k+1}$. 

We claim that this induction process provides a section $s$ of $E$ which satisfies the requirement. Indeed, since the open cover $\{U_i\}_{i=1}^\infty$ of $D$ is locally finite, the value of the section becomes stabilized after finitely many steps of the induction, hence defines a continuous section. Moreover, in each step the value of the section remains unchanged over the open set $\rho_C^{-1}([-1, 0))$. Transversality also holds by construction.
\end{proof}

\begin{cor}
Let $M$ be a topological manifold with or without boundary, and let $s_1, s_2: M \to E$ be two transverse sections which are homotopic. Then the two submanifolds (with or without boundary) $S_1:= s_1^{-1}(0)$ and $S_2:= s_2^{-1}(0)$ are cobordant. 
\end{cor}

\begin{rem}
In the application of this paper, the target pair $(X, Y)$ in the transversality problem is either a smooth submanifold inside a smooth manifold (or orbifolds), or the zero section of a vector bundle. Hence $X$ admits a tubular neighborhood and a unique equivalence class of normal microbundle. In fact the normal microbundle is equivalent to a disk bundle of the smooth normal bundle. Hence in the remaining discussions, we make the stronger assumption that all normal microbundles are disk bundles of some vector bundle. This does not alter the above discussion. For example, the compositions of two embeddings with disk bundle neighborhoods is still an embedding with a disk bundle neighborhood.
\end{rem}

\subsection{Topological orbifolds and orbibundles}

We use Satake's notion of V-manifolds \cite{Satake_orbifold} instead of groupoids to treat orbifolds, and only discuss it in the topological category. In this paper we only consider effective orbifolds. 

\begin{defn}
Let $M$ be a second countable Hausdorff topological space.

\begin{enumerate}

\item Let $x \in M$ be a point. A topological orbifold chart (with boundary) of $x$ consists of a triple $(\tilde U_x, \Gammait_x, \varphi_x)$, where $\tilde U_x$ is a topological manifold with possibly empty boundary $\partial \tilde U_x$, $\Gammait_x$ is a finite group acting continuously  on $( \tilde U_x, \partial \tilde U_x)$ and
\beqn
\varphi_x: \tilde U_x/\Gammait_x \to M
\eeqn
is a continuous map which is a homeomorphism onto an open neighborhood of $x$. Denote the image $U_x = \varphi_x(\tilde U_x/ \Gammait_x) \subset M$ and denote the composition 
\beqn
\tilde \varphi_x: \xymatrix{ \tilde U_x \ar[r] & \tilde U_x/ \Gammait_x \ar[r]^-{\varphi_x} & M}.
\eeqn

\item If $p \in \tilde U_x$, take $\Gammait_p = (\Gammait_x)_p \subset \Gammait_x$ the stabilizer of $p$. Let $\tilde U_p \subset \tilde U_x$ be a $\Gammait_p$-invariant neighborhood of $p$. Then there is an induced chart (which we call a subchart) $( \tilde U_p, \Gammait_p, \varphi_p)$, where $\varphi_p$ is the composition
\beqn
\varphi_p: \xymatrix{ \tilde U_p/\Gammait_p \ar@{^{(}->}[r] & \tilde U_x/\Gammait_x \ar[r]^-{\varphi_p} & M
}.
\eeqn

\item Two charts $(\tilde U_x, \Gammait_x, \varphi_x)$ and $( \tilde U_y, \Gammait_y, \varphi_y)$ are {\it compatible} if for any $p \in \tilde U_x$ and $q \in \tilde U_y$ with $\varphi_x(p) = \varphi_y (q)\in M$, there exist an isomorphism $\Gammait_p \to \Gammait_q$, subcharts $ \tilde U_p \ni p$, $\tilde U_q \ni q$ and an equivariant homeomorphism $\varphi_{qp}: ( \tilde U_p, \partial \tilde U_p) \simeq ( \tilde U_q, \partial \tilde U_q)$.

\item A {\it topological orbifold atlas} of $M$ is a {\it set} $\{ ( \tilde U_\alpha, \Gammait_\alpha, \varphi_\alpha)\ |\ \alpha \in I \}$ of topological orbifold charts of $M$ such that $M = \bigcup_{\alpha\in I} U_\alpha$ and for each pair $\alpha, \beta \in I$, $( \tilde U_\alpha, \Gammait_\alpha, \varphi_\alpha), ( \tilde U_\beta, \Gammait_\beta, \varphi_\beta)$ are compatible. Two atlases are {\it equivalent} if the union of them is still an atlas. A structure of topological orbifold (with boundary) is an equivalence class of atlases. A topological orbifold (with boundary) is a second countable Hausdorff space with a structure of topological orbifold (with boundary).
\end{enumerate}
\end{defn}
We will often skip the term ``topological'' in the rest of this paper.

Now consider bundles. Let $E$, $B$ be orbifolds and $\pi: E \to B$ be a continuous map.
\begin{defn}
A vector bundle chart (resp. disk bundle chart) of $\pi: E \to B$ is a tuple $( \tilde U, F^n, \Gammait, \hat\varphi, \varphi)$ where $F^n = {\mb R}^n$ (resp. $F^n = {\mb D}^n$), $( \tilde U, \Gammait, \varphi)$ is a chart of $B$ and $( \tilde U \times F^n, \Gammait, \hat \varphi)$ is a chart of $E$, where $\Gammait$ acts on $F^n$ via a representation $\Gammait \to GL({\mb R}^n)$ (resp. $\Gammait \to O(n)$). The compatibility condition is required, namely, the following diagram commutes.
\beqn
\vcenter{\xymatrix{
\tilde U \times F^n/ \Gammait \ar[r]^-{\hat \varphi} \ar[d]_{\tilde \pi} & E \ar[d]^{\pi}\\
\tilde U /\Gammait \ar[r]^{\varphi} & B
}}.
\eeqn
If $(\tilde U_p, \Gammait_p, \varphi_p)$ is a subchart of $(\tilde U, \Gammait, \varphi)$, then one can restrict the bundle chart to $\tilde \pi^{-1}(\tilde U_p)$.
\end{defn}

We can define the notion of compatibility between bundle charts, the notion of orbifold bundle structures and the notion of orbifold bundles in a similar fashion as in the case of orbifolds. We skip the details.

\subsubsection{Embeddings}

Now we consider embeddings for orbifolds and orbifold vector bundles. First we consider the case of manifolds. Let $S$ and $M$ be topological manifolds and $E \to S$, $F \to M$ be continuous vector bundles. Let $\phi: S \to M$ be a topological embedding. A {\it bundle embedding} covering $\phi$ is a continuous map $\wh \phi: E \to F$ which makes the diagram
\beqn
\xymatrix{ E \ar[r]^{\wh\phi} \ar[d] & F \ar[d] \\
           S \ar[r]^\phi         & M }
\eeqn
commute and which is fibrewise a linear injective map. Since $\wh \phi$ determines $\phi$, we also call $\wh \phi: E \to F$ a bundle embedding. 

\begin{defn} {\rm (Orbifold embedding)}
Let $S$, $M$ be orbifolds and $f: S \to M$ is a continuous map which is a homeomorphism onto its image. $\phi$ is called an {\it embedding} if for any pair of orbifold charts, $(\tilde U, \Gammait, \varphi)$ of $S$ and $(\tilde V, \Piit, \psi)$ of $M$, any pair of points $p\in \tilde U$, $q \in  \tilde V$ with $\phi (\varphi(p)) = \psi(q)$, there are subcharts $( \tilde U_p, \Gammait_p, \varphi_p) \subset (\tilde U, \Gammait, \varphi)$ and $(\tilde V_q, \Piit_q, \psi_q) \subset ( \tilde V, \Piit, \psi)$, an isomorphism $\Gammait_{p} \simeq \Piit_q$ and an equivariant locally flat embedding $\tilde \phi_{pq}: \tilde U_p \to \tilde V_q$ such that the following diagram commutes.
\beqn
\vcenter{
\xymatrix{ \tilde U_p \ar[d]_{\tilde \varphi_p} \ar[r]^{\tilde \phi_{pq}} & \tilde V_q \ar[d]^{\tilde \psi_q}\\
S \ar[r]^{\phi } & M } }
\eeqn
\end{defn}

\subsubsection{Multisections and perturbations}

The equivariant feature of the problem implies that transversality can only be achieved by multi-valued perturbations. Here we review basic notions and facts about multisections. Our discussion mainly follows the treatment of \cite{Fukaya_Ono}.

\begin{defn} {\rm (Multimaps)}
Let $A$, $B$ be sets, $l\in {\mb N}$, and ${\mc S}^l(B)$ be the $l$-fold symmetric product of $B$. 
\begin{enumerate}

\item An {\it $l$-multimap} $f$ from $A$ to $B$ is a map $f: A \to {\mc S}^l(B)$. For another $a \in {\mb N}$, there is a natural map
\beq\label{eqn71}
m_a: {\mc S}^l(B) \to {\mc S}^{al} (B)
\eeq
by repeating each component $a$ times.

\item If both $A$ and $B$ are acted by a finite group $\Gammait$, then we say that an $l$-multimap $f: A \to {\mc S}^l(B)$ is $\Gammait$-equivariant if it is equivariant with respect to the $\Gammait$-action on $A$ and the induced $\Gammait$-action on ${\mc S}^l(B)$. 

\item If $A$ and $B$ are both topological spaces, then an $l$-multimap $f: A \to {\mc S}^l(B)$ is called {\it continuous} if it is continuous with respect to the topology on ${\mc S}^l(B)$ induced as a quotient of $B^l$. 

\item A continuous $l$-multimap $f: A \to {\mc S}^l(B)$ is {\it liftable} if there are continuous maps $f_1, \ldots, f_l: A \to B$ such that
\beqn
f(x) = [f^1(x), \ldots, f^l(x)] \in {\mc S}^l(B),\ \forall x \in A.
\eeqn
$f^1, \ldots, f^l$ are called {\it branches} of $f$. 

\item An $l_1$-multimap $f_1: A \to {\mc S}^{l_1} (B)$ and an $l_2$-multimap $f_2: A \to {\mc S}^{l_2}(B)$ are called {\it equivalent} if there exists a common multiple $l = a_1 l_1 = a_2 l_2$ of $l_1$ and $l_2$ such that
\beqn
m_{a_1} \circ f_1 = m_{a_2} \circ f_2
\eeqn
as $l$-multimaps from $A$ to $B$. 

\item Being equivalent is clearly reflexive, symmetric and transitive. A multimap from $A$ to $B$, denoted by $f: A \overset{m}{\to} B$, is an element of 
\beqn
\left( \bigsqcup_{l\geq 1} {\rm Map}( A, {\mc S}^l(B))  \right)/ \sim
\eeqn
where $\sim$ is the above equivalence relation. 
\end{enumerate}
\end{defn}

In the discussions in this paper, we often identify an $l$-multimap with its equivalence class as a multimap.

\begin{defn}\label{defn716} {\rm (Multisections)}
Let $M$ be a topological orbifold and $E \to M$ be a vector bundle.
\begin{enumerate}

\item  A {\it representative} of a (continuous) multisection of $E$ is a collection
\beqn
\Big\{ ( \tilde U_\alpha, {\mb R}^k, \Gammait_\alpha, \hat \varphi_\alpha, \varphi_\alpha; s_\alpha, l_\alpha )\ |\ \alpha \in I \Big\}
\eeqn
where $\{ ( \tilde U_\alpha, {\mb R}^k, \Gammait_\alpha, \hat \varphi_\alpha, \varphi_\alpha)\ |\ \alpha \in I \}$ is a bundle atlas for $E \to M$ and $s_\alpha: \tilde U_\alpha \to {\mc S}^{l_\alpha} ({\mb R}^k)$ is a $\Gammait_\alpha$-equivariant continuous $l_\alpha$-multimap, satisfying the following compatibility condition.
\begin{itemize}
\item For any $p\in \tilde U_\alpha$ and $q \in \tilde U_\beta$ with $\varphi_\alpha(p) = \varphi_\beta(q) \in M$, there exist subcharts $ \tilde U_p \subset \tilde U_\alpha$, $\tilde U_q \subset \tilde U_\beta$, an isomorphism $(\hat \varphi_{pq}, \varphi_{pq})$ of subcharts, a common multiple $l = a_\alpha l_\alpha = a_\beta l_\beta$ of $l_{\alpha}$ and $l_\beta$, such that
\beqn
\hat \varphi_{pq} \circ m_{a_\beta} \circ s_\beta |_{\tilde U_q} = m_{a_\alpha} \circ s_\alpha|_{\tilde U_p} \circ \varphi_{pq}.
\eeqn
\end{itemize}

\item Two representatives are equivalent if their union is also a representative. An equivalence class is called a multisection of $E$, denoted by 
\beqn
s: O \overset{m}{\to} E.
\eeqn

\item A multisection $s: M \overset{m}{\to} E$ is called {\it locally liftable} if for any $p \in M$, there exists a local representative $( \tilde U_p \times {\mb R}^k, \Gammait_p, \hat \varphi_p, \varphi_p; s_p, l_p)$ such that $s_p: \tilde U_p \to {\mc S}^{l_p}({\mb R}^k)$ which is a liftable continuous $l_p$-multimap.

\item A multisection $s: M \overset{m}{\to} E$ is called {\it transverse} if it is locally liftable and for any liftable local representative $s_p: \tilde U_p \to {\mc S}^{l}({\mb R}^k)$, all branches are transverse to the origin of ${\mb R}^k$.
\end{enumerate}
\end{defn}

The space of continuous multisections of $E \to M$, denoted by $C_m^0(M, E)$, is acted by the space of continuous functions $C^0(M)$ on $M$ by pointwise multiplication. $C_m^0(M, E)$ also has the structure of a commutative monoid, but not an abelian group. The additive structure is defined as follows. If $s_1, s_2: M \overset{m}{\to} E$ are multisections, then for liftable local representatives with branches $s_1^a: \tilde U \to {\mb R}^n$, $1 \leq a \leq l$, $s_2^b: \tilde U \to {\mb R}^n$, $1 \leq b \leq k$, define 
\beqn
(s_1 + s_2)^{ab} = [ s_1^a + s_2^b]_{1\leq a \leq l}^{1 \leq b \leq k}.
\eeqn
However there is no inverse to this addition: one can only invert the operation of adding a single valued section. It is enough, though, since we have the notion of being transverse to a single valued section which is not necessarily the zero section. 

We also want to measure the size of multisections. A continuous norm on an orbifold vector bundle $E \to M$ is a continuous function $\| \cdot \|: E \to [0, +\infty)$ which only vanishes on the zero section such that over each local chart, it lifts to an equivariant norm on the fibres. It is easy to construct norms in the relative sense, as one can extend continuous functions defined on closed sets.

The following lemma, which is a generalization of Theorem \ref{thm79}, shows one can achieve transversality for multisections by perturbation relative to a region where transversality already holds.

\begin{lemma}\label{lemma717}
Let $M$ be an orbifold and $E \to M$ be an orbifold vector bundle. Let $C \subset D \subset M$ be closed subsets. Let $S: M \to E$ be a single-valued continuous function and $t_C: M \overset{m}{\to} E$ be a multisection such that $S + t_C$ is transverse over a neighborhood of $C$. Then there exists a multisection $t_D: M \overset{m}{\to} E$ satisfying the following condition. 
\begin{enumerate}
\item $t_C = t_D$ over a neighborhood of $C$.

\item $S + t_D$ is transverse over a neighborhood of $D$. 
\end{enumerate}
Moreover, if $E$ has a continuous norm $\| \cdot \|$, then for any $\epsilon>0$, one can require that 
\beqn
\| t_D \|_{C^0} \leq \| t_C \|_{C^0} + \epsilon.
\eeqn
\end{lemma}

\begin{proof}
Similar to the proof of Theorem \ref{thm79}, one can choose a countable locally finite open cover $\{ U_i\}_{i=1}^\infty$ of $D$ satisfying the following conditions. 
\begin{enumerate}

\item Each $\ov{U_i}$ is compact.

\item There are a collection of precompact open subsets $V_i \sqsubset U_i$ such that $\{V_i\}_{i=1}^\infty$ is still an open cover of $D$.

\item Over each $U_i$ there is a local representative of $t_C$, written as 
\beqn
(\tilde U_i \times {\mb R}^k, \Gammait_i, \hat \varphi_i, \varphi_i; t_{i}, l_i)
\eeqn
where $t_i: \tilde U_i \to {\mc S}^{l_i}({\mb R}^k)$ is a $\Gammait_i$-equivariant $l_i$-multimap which is liftable. Write
\beqn
t_i (x) = [ t_i^1(x), \ldots, t_i^{l_i}(x)].
\eeqn
In this chart also write $S$ as a map $S_i: \tilde U_i \to {\mb R}^k$. 
\end{enumerate}

The transversality assumption implies that there is an open neighborhood $U_C \subset M$ of $C$ such that for each $i$, each $a\in \{1, \ldots, l_i\}$, $s_i^a$ is transverse to the origin over 
\beqn
\tilde U_{i, C}:= \varphi_i^{-1}(U_C) \subset \tilde U_C.
\eeqn
Using an inductive construction which is very similar to that in the proof of Theorem \ref{thm79}, one can construct a valid perturbation. We only sketch the construction for the first chart. Indeed, one can perturb each $t_1^a$ over $\tilde U_1$ to a function $\dot t_1^a: \tilde U_1 \to {\mb R}^k$, such that $S_1 + \dot t_1^a$ is transverse over a neighborhood of the closure of $\tilde V_1:= \varphi_1^{-1}(V_1)$ inside $\tilde U_1$, but $\dot t_1^a  = t_1^a$ over a neighborhood of $\tilde U_{1, C}$ and the near the boundary of $\tilde U_1$. Moreover, given $\epsilon>0$ we may require that 
\beq\label{eqn72}
\sup_{x \in \tilde U_1} \| \dot t_1^a(x) \| \leq \sup_{x \in \tilde U_1} \| t_1^a(x) \| + \frac{\epsilon}{2}.
\eeq
Then we obtain a continuous $l_1$-multimap
\beqn
\dot t_1(x) = [ \dot t_1^1(x), \ldots, \dot t_1^{l_1}(x)].
\eeqn
This multimap may not be $\Gammait_1$-transverse. We reset 
\beqn
\dot t_1 (x):= [ t_1^{ab}]_{1\leq a \leq l_1}^{1 \leq b \leq n_1}:= [ g_b^{-1} \dot t_1^a (g_b x)]_{1 \leq a \leq l_1}^{1 \leq b \leq n_1}
\eeqn
where $\Gammait_1 = \{ g_1, \ldots, g_{n_1}\}$. It is easy to verify that this is $\Gammait_1$-invariant, and agrees with $t_1$ over a neighborhood of $\tilde U_{1, C}$ and near the boundary of $\tilde U_1$. There still holds
\beqn
\sup_{a, b} \sup_{x \in \tilde U_1} \| \dot t_1^{ab}(x) \| \leq \sup_a \sup_{x \in \tilde U_1} \| t_1^a(x) \| + \frac{\epsilon}{2}.
\eeqn
Therefore, together with the original multisection over the complement of $U_1$, $\dot t_1$ defines a continuous multisection of $E$. Moreover, it agrees with the original one over a neighborhood of $C$ and is transverse near $C \cup \ov{V_1}$. In this way we can continue the induction to perturb over all $\tilde U_i$. At the $k$-th step of the induction, we replace $\frac{\epsilon}{2}$ by $\frac{\epsilon}{2^k}$ in the $C^0$ bound \eqref{eqn72}. Since $U_i$ is locally finite, near each point, the value of the perturbation stabilizes after finitely many steps of this induction. This results in a multisection $t_D$ which satisfies our requirement.
\end{proof}

\subsection{Virtual orbifold atlases}

Now we introduce the notion of virtual orbifold atlases. This notion plays a role as a general structure of moduli spaces we are interested in, and is a type of intermediate objects in concrete constructions. The eventual objects we would like to construct are {\it good coordinate systems}, which are special types of virtual orbifold atlases. 

\begin{defn}\label{defn718}
Let $X$ be a compact and Hausdorff space.
\begin{enumerate}

\item A {\it virtual orbifold chart} (chart for short) is a tuple 
\beqn
C:= (U, E, S, \psi, F)
\eeqn
where
\begin{enumerate}

\item $U$ is a topological orbifold (with boundary).

\item $E \to U$ is a continuous orbifold vector bundle.

\item $S: U \to E$ is a continuous section.

\item $F \subset X$ is an open subset. 

\item $\psi: S^{-1}(0) \to F$ is a homeomorphism.
\end{enumerate}
$F$ is call the {\it footprint} of this chart $C$, and the integer ${\rm dim} U - {\rm rank} E$ is called the {\it virtual dimension} of $C$.

\item Let $C = (U, E, S, \psi, F)$ be a chart and $U' \subset U$ be an open subset. The {\it restriction} of $C$ to $U'$ is the chart 
\beqn
C' = C|_{U'} = (U', E', S', \psi', F')
\eeqn
where $E' = E|_{U'}$, $S' = S|_{U'}$, $\psi' = \psi|_{(S')^{-1}(0)}$, and $F' = {\bf Image} \psi'$. Any such chart $C'$ induced from an open subset $U' \subset U$ is called a {\it shrinking} of $C$. A shrinking $C' = C|_{U'}$ is called a {\it precompact shrinking} if $U' \sqsubset U$, denoted by $C' \sqsubset C$.
\end{enumerate}
\end{defn}

A very useful lemma about shrinkings is the following, whose proof is left to the reader.

\begin{lemma}
Suppose $C = (U, E, S, \psi, F)$ is a virtual orbifold atlas and let $F' \subset F$ be an open subset. Then there exists a shrinking $C'$ of $C$ whose footprints is $F'$. Moreover, if $F' \sqsubset F$, then $C'$ can be chosen to be a precompact shrinking. 
\end{lemma}

\begin{defn}\label{defn720}
Let $C_i:= (U_i, E_i, S_i, \psi_i, F_i)$, $i=1,2$ be two charts of $X$. An {\it embedding} of $C_1$ into $C_2$ consists of a bundle embedding $\wh{\phi}_{21}$ satisfying the following conditions.
\begin{enumerate}

\item The following diagrams commute;
\begin{align*}
&\ \xymatrix{E_1 \ar[r]^{\wh\phi_{21}}   \ar[d]^{\pi_1} & E_2 \ar[d]_{\pi_2} \\
          U_1 \ar@/^1pc/[u]^{S_1} \ar[r]^{\phi_{21}} & U_2 \ar@/_1pc/[u]_{S_2}       }\ &\  \xymatrix{  S_1^{-1}(0) \ar[r]^{\phi_{21}} \ar[d]^{\psi_1} & S_2^{-1}(0) \ar[d]^{\psi_2}  \\     X \ar[r]^{{\rm Id}} & X }
\end{align*}

\item {\bf (Tangent Bundle Condition)} There exists an open neighborhood $N_{21} \subset U_2$ of $\phi_{21}(U_1)$ and a subbundle $E_{1;2}\subset E_2|_{N_{21}}$ which extends $\wh\phi_{21}(E_1)$ such that $S_2|_{N_2}: N_2 \to E_2|_{N_2}$ is transverse to $E_{1;2}$ and $S_2^{-1}(E_{1;2}) = \phi_{21}(U_1)$. 
\end{enumerate}
\end{defn}

The following lemma is left to the reader. 

\begin{lemma}
The composition of two embeddings is still an embedding.
\end{lemma}

\begin{defn}\label{defn722}
Let $C_i = (U_i, E_i, S_i, \psi_i, F_i)$, $(i=1, 2)$ be two charts. A {\bf coordinate change} from $C_1$ to $C_2$ is a triple $T_{21} = (U_{21}, \phi_{21}, \wh\phi_{21})$, where $U_{21}\subset U_1$ is an open subset and $(\phi_{21}, \wh\phi_{21})$ is an embedding from $C_1|_{U_{21}}$ to $C_2$. They should satisfy the following conditions.
\begin{enumerate}

\item $\psi_1(U_{21} \cap S_1^{-1}(0)) = F_1 \cap F_2$.

\item If $x_k \in U_{21}$ converges to $x_\infty \in U_1$ and $y_k = \phi_{21}(x_k)$ converges to $y_\infty \in U_2$, then $x_\infty \in U_{21}$ and $y_\infty = \phi_{21}(y_\infty)$. 
\end{enumerate}
\end{defn}

\begin{lemma}\label{lemma723}
Let $C_i = (U_i, E_i, S_i, \psi_i, F_i)$, $(i=1, 2)$ be two charts and let $T_{21} = (U_{21}, \wh\phi_{21})$ be a coordinate change from $C_1$ to $C_2$. Suppose $C_i' = C_i|_{U_i'}$ be a shrinking of $C_i$. Then the restriction $T_{21}':= T_{21}|_{U_1' \cap \phi_{21}^{-1}(U_2')}$ is  a coordinate change from $C_1'$ to $C_2'$. 
\end{lemma}

\begin{proof}
Left to the reader.
\end{proof}

We call $T_{21}'$ in the above lemma the {\it induced} coordinate change from the shrinking.

Now we introduce the notion of atlases.

\begin{defn}\label{defn724}
Let $X$ be a compact metrizable space. A {\it virtual orbifold atlas} of virtual dimension $d$ on $X$ is a collection
\beqn
{\mf A}:= \Big( \Big\{ C_I:= (U_I, E_I, S_I, \psi_I, F_I)\ |\ I \in {\mc I} \Big\},\ \Big\{ T_{JI} = \big( U_{JI}, \phi_{JI}, \wh\phi_{JI}\big) \ |\ I \preq J \Big\}\Big),
\eeqn
where
\begin{enumerate}

\item $({\mc I}, \preq)$ is a finite, partially ordered set.

\item For each $I\in {\mc I}$, $C_I$ is a virtual orbifold chart of virtual dimension $d$ on $X$.

\item For $I \preq J$, $T_{JI}$ is a coordinate change from $C_I$ to $C_J$.
\end{enumerate}
They are subject to the following conditions.
\begin{itemize}
\item {\bf (Covering Condition)} $X$ is covered by all the footprints $F_I$. 

\item {\bf (Cocycle Condition)} For $I \preq J \preq K \in {\mc I}$, denote $U_{KJI} = U_{KI} \cap \phi_{JI}^{-1} (U_{KJ}) \subset U_I$. Then we require that 
\beqn
\wh\phi_{KI}|_{U_{KJI}} = \wh\phi_{KJ} \circ \wh \phi_{JI}|_{U_{KJI}}
\eeqn
as bundle embeddings.

\item {\bf (Overlapping Condition)} For $I, J \in {\mc I}$, we have
\beqn
\ov{F_I} \cap \ov{F_J} \neq \emptyset \Longrightarrow I \preq J\ {\rm or}\ J \preq I.
\eeqn
\end{itemize}
\end{defn}

All virtual orbifold atlases considered in this paper have definite virtual dimensions, although sometimes we do not explicitly mention it.

\subsubsection{Orientations}

Now we discuss orientation. When $M$ is a topological manifold, there is an orientation bundle ${\mc O}_M \to M$ which is a double cover of $M$ (or a ${\mb Z}_2$-principal bundle). $M$ is orientable if and only if ${\mc O}_M$ is trivial. If $E \to M$ is a continuous vector bundle, then $E$ also has an orientation bundle ${\mc O}_E \to M$ as a double cover. Since ${\mb Z}_2$-principal bundles over a base $B$ are classified by $H^1(B; {\mb Z}_2)$, the orientation bundles can be multiplied. We use $\otimes$ to denote this multiplication.

\begin{defn} {\rm (Orientability)}
\begin{enumerate}
\item A virtual orbifold chart $C = (U, E, S, \psi, F)$ is {\it locally orientable} if for any bundle chart $(\tilde U, {\mb R}^n, \Gammait, \wh \varphi, \varphi)$ of $E$, if we denote $\tilde E = \tilde U \times {\mb R}^n$, then for any $\gamma \in \Gammait$, the map
\beqn
\gamma: {\mc O}_{\tilde U} \otimes {\mc O}_{\tilde E^*} \to {\mc O}_{\tilde U} \otimes {\mc O}_{\tilde E^*}
\eeqn
is the identity over all fixed points of $\gamma$. 

\item If $C$ is locally orientable, then ${\mc O}_{\tilde U} \otimes {\mc O}_{\tilde E^*}$ for all local charts glue together a double cover ${\mc O}_C \to U$. If ${\mc O}_C$ is trivial (resp. trivialized), then we say that $C$ is {\it orientable} (resp. {\it oriented}).

\item A coordinate change $T_{21}= (U_{21}, \phi_{21}, \wh\phi_{21})$ between two oriented charts $C_1 = (U_1, E_1, S_1, \psi_1, F_1)$ and $C_2 = (U_2, E_2, S_2, \psi_2, F_2)$ is called {\it oriented} if the embeddings $\phi_{21}$ and $\wh \phi_{21}$ are compatible with the orientations on ${\mc O}_{C_1}$ and ${\mc O}_{C_2}$.

\item An atlas ${\mf A}$ is oriented if all charts are oriented and all coordinate changes are oriented.
\end{enumerate}
\end{defn}

\subsection{Good coordinate systems}

Now we introduce the notion of shrinkings of virtual orbifold atlases. 

\begin{defn}\label{defn726}
Let ${\mf A} = ( \{ C_I | I \in {\mc I} \}, \{ T_{JI} | I \preq J \})$ be a virtual orbifold atlas on $X$.
\begin{enumerate}

\item  A {\it shrinking} of ${\mf A}$ is another virtual orbifold atlas ${\mf A}' = ( \{ C_I' | I \in {\mc I} \}, \{ T_{JI}' | I \preq J \})$ indexed by elements of the same partially ordered set ${\mc I}$ such that for each $I \in {\mc I}$, $C_I'$ is a shrinking $C_I|_{U_I'}$ of $C_I$ and for each $I \preq J$, $T_{JI}'$ is the induced shrinking of $T_{JI}$ given by Lemma \ref{lemma723}.

\item If for every $I \in {\mc I}$, $U_I'$ is a precompact subset of $U_I$, then we say that ${\mf A}'$ is a {\it precompact shrinking} of ${\mf A}$ and denote ${\mf A}' \sqsubset {\mf A}$. 

\end{enumerate}
\end{defn}

Given a virtual orbifold atlas ${\mf A} = ( \{ C_I | I \in {\mc I} \}, \{ T_{JI} | I \preq J \} )$, 
we define a relation $\curlyvee$ on the disjoint union $\bigsqcup_{I \in {\mc I}} U_I$ as follows. $U_I \ni x \curlyvee y\in U_J$ if one of the following holds.
\begin{enumerate}
\item $I = J$ and $x = y$;

\item $I \preq J$, $x \in U_{JI}$ and $y = \phi_{JI}(x)$;

\item $J \preq I$, $y \in U_{IJ}$ and $x = \phi_{IJ}(y)$.
\end{enumerate}
If ${\mf A}'$ is a shrinking of ${\mf A}$, then it is easy to see that the relation $\curlyvee'$ on $\bigsqcup_{I \in {\mc I}} U_I'$ defined as above is induced from the relation $\curlyvee$ for ${\mf A}$ via restriction. 

For an atlas ${\mf A}$, if $\curlyvee$ is an equivalence relation, we can form the quotient space
\beqn
|{\mf A}|:= \Big( \bigsqcup_{I \in {\mc I}} U_I \Big)/ \curlyvee.
\eeqn
with the quotient topology. There is a natural injective map 
\beqn
X \hookrightarrow |{\mf A}|.
\eeqn
We call $|{\mf A}|$ the {\it virtual neighborhood} of $X$ associated to the atlas ${\mf A}$. Denote the quotient map by 
\beq\label{eqn73}
\pi_{\mf A}: \bigsqcup_{I \in {\mc I}} U_I \to |{\mf A}|
\eeq
which induces continuous injections $U_I \hookrightarrow |{\mf A}|$. A point in $|{\mf A}|$ is denoted by $|x|$, which has certain representative $x \in U_I$ for some $I$. 

\begin{defn}\label{defn727}
A virtual orbifold atlas ${\mf A}$ on $X$ is called a {\it good coordinate system} if the following conditions are satisfied.
\begin{enumerate}

\item $\curlyvee$ is an equivalence relation.

\item The virtual neighborhood $|{\mf A}|$ is a Hausdorff space.

\item For all $I\in {\mc I}$, the natural maps $U_I \to |{\mf A}|$ are homeomorphisms onto their images.
\end{enumerate}
\end{defn}

The conditions for good coordinate systems are very useful for later constructions (this is the same as in the Kuranishi approach, see \cite{FOOO_2016}), for example, the construction of suitable multisection perturbations. In these constructions, the above conditions are often implicitly used without explicit reference. Therefore, an important step is to construct good coordinate systems.

\begin{thm}\label{thm728} {\rm (Constructing good coordinate system)}
Let ${\mf A}$ be a virtual orbifold atlas on $X$ with the collection of footprints $F_I$ indexed by $I \in {\mc I}$. Let $F_I^\square \sqsubset F_I$ for all $I\in {\mc I}$ be a collection of precompact open subsets such that 
\beqn
X = \bigcup_{I \in {\mc I}} F_I^\square.
\eeqn
Then there exists a shrinking ${\mf A}' $ of ${\mf A}$ such that the collection of shrunk footprints $F_I'$ contains $\ov{F_I^\square}$ for all $I \in {\mc I}$ and ${\mf A}'$ is a good coordinate system. 

Moreover, if ${\mf A}$ is already a good coordinate system, then any shrinking of ${\mf A}$ remains a good coordinate system. 
\end{thm}

We give a proof of Theorem \ref{thm728} in the next subsection. A similar result is used in the Kuranishi approach while our argument potentially differs from that of \cite{FOOO_2016}.

\begin{rem}\label{rem729}
If ${\mf A}$ is a good coordinate system, and ${\mf A}'$ is a shrinking of ${\mf A}$, then the shrinking induces a natural map 
\beqn
|{\mf A}'| \hookrightarrow |{\mf A}|.
\eeqn
If we equip both $|{\mf A}'|$ and $|{\mf A}|$ with the quotient topologies, then the natural map is continuous. However there is another topology on $|{\mf A}'|$ by viewing it as a subset of $|{\mf A}|$. We denote this topology by $\|{\mf A}'\|$ and call it the {\it subspace topology}. In most cases, the quotient topology is strictly stronger than the subspace topology. Hence it is necessary to distinguish the two different topologies. 
\end{rem}

\subsection{Shrinking good coordinate systems}

In this subsection we prove Theorem \ref{thm728}. First we show that by precompact shrinkings one can make the relation $\curlyvee$ an equivalence relation. 

\begin{lemma}\label{lemma730}
Let ${\mf A}$ be a virtual orbifold atlas on $X$ with the collection of footprints $\{F_I\ |\ I \in {\mc I}\}$. Let $F_I^\square \sqsubset F_I$ be precompact open subsets such that 
\beqn
X = \bigcup_{I \in {\mc I}} F_I^\square.
\eeqn
Then there exists a precompact shrinking ${\mf A}'$  of ${\mf A}$ whose collection of footprints $F_I'$ contains $\ov{F_I^\square}$ for all $I \in {\mc I}$ such that the relation $\curlyvee$ on ${\mf A}'$ is an equivalence relation.
\end{lemma}

\begin{proof}
By definition, the relation $\curlyvee$ is reflexive and symmetric. By the comments above, any shrinking of ${\mf A}$ will preserve reflexiveness and symmetry. Hence we only need to shrink the atlas to make the induced relation transitive. 

For any subset ${\mc I}' \in {\mc I}$, we say that $\curlyvee$ is transitive in ${\mc I}'$ if for $x, y, z \in \bigsqcup_{I \in {\mc I}'} U_I$, $x\curlyvee y$, $y \curlyvee z$ imply $x \curlyvee z$. Being transitive in any subset ${\mc I}'$ is a condition that is preserved under shrinking. Hence it suffices to construct shrinkings such that $\curlyvee$ is transitive in $\{I, J, K\}$ for any three distinct elements $I, J, K \in {\mc I}$. Let $x\in U_I$, $y \in U_J$, $z \in U_K$ be general elements.

Since $U_I$ (resp. $U_J$ resp. $U_K$) is an orbifold and hence metrizable, we can choose a sequence of precompact open subsets $U_I^n \sqsubset U_I$ (resp. $U_J^n \sqsubset U_J$ resp. $U_K^n \sqsubset U_K$) containing $\ov{F_I^\square}$ (resp. $\ov{F_J^\square}$ resp. $\ov{F_K^\square}$) such that 
\beqn
U_I^{n+1} \sqsubset U_I^n,\ \ \ U_J^{n+1} \sqsubset U_J^n,\ \ \ U_K^{n+1} \sqsubset U_K^n,
\eeqn
and
\beqn
\bigcap_{n} U_I^n = \psi_I^{-1}(\ov{F_I^\square}),\ \ \ \bigcap_n U_J^n = \psi_J^{-1}(\ov{F_J^\square}),\ \ \ \bigcap_{n} U_K^n = \psi_K^{-1}(\ov{F_K^\square}).
\eeqn
Then for each $n$, $U_I^n, U_J^n, U_K^n$ induce a shrinking of the atlas ${\mf A}$, denoted by ${\mf A}^n$. Let the induced binary relation on $U_I^n \sqcup U_J^n \sqcup U_K^n$ still by $\curlyvee$. We claim that for $n$ large enough, $\curlyvee$ is an equivalence relation on this triple disjoint union. Denote the domains of the shrunk coordinate changes by $U_{JI}^n$, $U_{KJ}^n$ and $U_{KI}^n$ respectively.

If this is not true, then without loss of generality, we may assume that for all large $n$, there exist points $x^n \in U_I^n$, $y^n \in  U_J^n$, $z^n \in U_K^n$ such that  
\beq\label{eqn74}
x^n \curlyvee y^n,\ y^n \curlyvee z^n,\ {\rm but}\  (x^n, z^n) \notin \curlyvee;
\eeq
Then for some subsequence (still indexed by $n$), $x^n$, $y^n$ and $z^n$ converge to $x^\infty \in \psi_I^{-1}(\ov{F_I^\square}) \subset U_I$, $y^\infty \in \psi_J^{-1}( \ov{F_J^\square}) \subset U_J$ and $z^\infty \in \psi_K^{-1}(\ov{F_K^\square}) \subset U_K$ respectively. Then by the definition of coordinate changes (Definition \ref{defn722}), one has
\beqn
x^\infty \curlyvee y^\infty,\ y^\infty \curlyvee z^\infty \Longrightarrow \psi_I(x^\infty) = \psi_J(y^\infty) = \psi_K(z^\infty) \in \ov{F_I^\square} \cap \ov{F_J^\square} \cap \ov{F_K^\square}.
\eeqn
By the {\bf (Overlapping Condition)} of Definition \ref{defn724}, $\{I, J, K\}$ is totally ordered. Since the roles of $K$ and $I$ are symmetric, we may assume that $I \preq K$. Then since
\beqn
x^{\infty} \in \psi_I^{-1} \big( \ov{F_I^\square} \cap \ov{F_K^\square} \big) \subset \psi_I^{-1}( F_I \cap F_K) = \psi_I^{-1} ( F_{KI}) \subset U_{KI}
\eeqn
and $U_{KI} \subset U_I$ is an open set, for $n$ large enough one has
\beqn
x^n \in U_{KI}. 
\eeqn

\begin{enumerate}
\item If $I \preq J \preq K$, then by {\bf (Cocycle Condition)} of Definition \ref{defn724},
\beqn
\phi_{KI}(x^n) =\phi_{KJ} (\phi_{JI}(x^n)) = \phi_{KJ}( y^n) = z^n.
\eeqn
So $x^n \curlyvee z^n$, which contradicts \eqref{eqn74}.

\item If $J \preq I \preq K$, then {\bf (Cocycle Condition)} of Definition \ref{defn724},
\beqn
z^n = \phi_{KJ} (y^n) = \phi_{KI} \big( \phi_{IJ}(y^n) \big) = \phi_{KI}(x^n).
\eeqn
So $x^n \curlyvee z^n$, which contradicts \eqref{eqn74}.

\item If $I \preq K \preq J$, then since $\phi_{KI}(x^{\infty}) \in \psi_K^{-1}( F_J \cap F_K) \subset U_{JK}$, for large $n$, $x^n\in \phi_{KI}^{-1}(U_{JK})$. Then by {\bf (Cocycle Condition)} of Definition \ref{defn724},
\beqn
\phi_{JK}(z^n) = y^n = \phi_{JI}(x^n) = \phi_{JK} \big( \phi_{KI}(x^n) \big).
\eeqn
Since $\phi_{JK}$ is an embedding, we have $\phi_{KI}(x^n) = z^n$. Therefore, $x^n \curlyvee z^n$, which contradicts \eqref{eqn74}.
\end{enumerate}
Therefore, $\curlyvee^n$ is an equivalence relation on $U_I^n \sqcup U_J^n \sqcup U_K^n$ for large enough $n$. We can perform the shrinking for any triple of elements of ${\mc I}$, which eventually makes $\curlyvee$ an equivalence relation. By the construction the shrunk footprints $F_I'$ still contain $\ov{F_I^\square}$. 
\end{proof}

\begin{lemma}\label{lemma731}
Suppose ${\mf A}$ is a virtual orbifold atlas on $X$ such that the relation $\curlyvee$ is an equivalence relation. Suppose there is a collection of precompact subsets $F_I^\square \sqsubset F_I$ of footprints of ${\mf A}$ such that 
\beqn
X = \bigcup_{I \in {\mc I}} F_I^\square.
\eeqn
Then there exists a precompact shrinking ${\mf A}' \sqsubset {\mf A}$ satisfying 
\begin{enumerate}
\item The shrunk footprints $F_I'$ all contain $F_I^\square$.

\item The virtual neighborhood $|{\mf A}'|$ is a Hausdorff space.
\end{enumerate}
\end{lemma}

Before proving Lemma \ref{lemma731}, we need some preparations. Order the finite set ${\mc I}$ as $\{ I_1, \ldots, I_m\}$ such that for $k = 1, \ldots, m$, 
\beqn
I_k \preq J \Longrightarrow J \in \{I_k, I_{k+1}, \ldots, I_m\}.
\eeqn
For each $k$, $\curlyvee$ induces an equivalence relation on $\bigsqcup_{i \geq k} U_{I_i}$ and denote the quotient space by $|{\mf A}_k|$. Then the map $\pi_{\mf A}$ of \eqref{eqn73} induces a natural continuous map 
\beqn
\pi_k: \bigsqcup_{i \geq k} U_{I_i} \to |{\mf A}_k|.
\eeqn

\begin{lemma}\label{lemma732}
For $k = 1, \ldots, m$, if $|{\mf A}_k|$ is Hausdorff and ${\mf A}'$ is a shrinking of ${\mf A}$, then $|{\mf A}_k'|$ is also Hausdorff.
\end{lemma}

\begin{proof}
Left to the reader. A general fact is that the quotient topology is always stronger than (or homeomorphic to) the subspace topology (see Remark \ref{rem729}).
\end{proof}

\begin{lemma}\label{lemma733}
The natural map 
\beq\label{eqn75}
|{\mf A}_{k+1}| \to |{\mf A}_k|
\eeq
is a homeomorphism onto an open subset.
\end{lemma}

\begin{proof}
The map is clearly continuous and injective. To show that it is a homeomorphism onto an open set, consider any open subset $O_{k+1}$ of its domain. Its preimage under the quotient map 
\beqn
U_{I_{k+1}} \sqcup \cdots \sqcup U_{I_m} \to |{\mf A}_{k+1}|
\eeqn
is denoted by 
\beqn
\tilde O_{k+1} = O_{I_{k+1}} \sqcup \cdots \sqcup O_{I_m},
\eeqn
where $O_{I_{k+1}}, \ldots, O_{I_m}$ are open subsets of $U_{I_{k+1}}, \ldots, U_{I_m}$ respectively. Define 
\beqn
O_{I_k}:= \bigcup_{i \geq k+1,\ I_i \geq I_k} \phi_{I_i I_k}^{-1}(O_{I_i}).
\eeqn
This is an open subset of $U_{I_k}$. Then the image of $O_{k+1}$ under the map \eqref{eqn75}, denoted by $O_k$, is the quotient of 
\beqn
\tilde O_k:= O_{I_k} \sqcup O_{I_{k+1}} \sqcup \cdots \sqcup O_{I_m} \subset U_{I_k} \sqcup \cdots \sqcup U_{I_m}
\eeqn
On the other hand, $\tilde O_k$ is exactly the preimage of $O_k$ under the quotient map. Hence by the definition of the quotient topology, $O_k$ is open. This show that \eqref{eqn75} is a homeomorphism onto an open subset. 
\end{proof}

\begin{proof}[Proof of Lemma \ref{lemma731}]
For each $k$, we would like to construct shrinkings $U_{I_i}'\sqsubset U_{I_i}$ for all $i \geq k$ such that $|{\mf A}_k'|$ is Hausdorff and the shrunk footprints $F_{I_i}'$ contains $\ov{F_{I_i}^\square}$ for all $i \geq k$. Our construction is based on a top-down induction. First, for $k = m$, $|{\mf A}_m| \simeq U_{I_m}$ and hence is Hausdorff. Suppose after shrinking $|{\mf A}_{k+1}|$ is already Hausdorff. 

Choose open subsets $F_{I_i}' \sqsubset F_{I_i}$ for all $i \geq k$ such that 
\begin{align*}
&\ \ov{F_{I_i}^\square} \subset F_{I_i}',\ &\ X =  \bigcup_{i \geq k} F_{I_i}' \cup \bigcup_{i \leq k-1} F_{I_i}.
\end{align*}
Choose precompact open subsets $U_{I_i}' \sqsubset U_{I_i}$ for all $i \geq k$ such that 
\begin{align*}
&\ \psi_{I_i}( U_{I_i}' \cap S_{I_i}^{-1}(0)) = F_{I_i}',\ &\ \psi_{I_i}( \ov{U_{I_i}'} \cap S_{I_i}^{-1}(0)) = \ov{F_{I_i}'}.
\end{align*}
Then $U_{I_i}'$ for $i \geq k$ and $U_{I_i}$ for $i<k$ provide a shrinking ${\mf A}'$ of ${\mf A}$. We claim that $|{\mf A}_k'|$ is Hausdorff. 

Indeed, pick any two different points $|x|, |y| \in |{\mf A}_k'|$. If $|x|, |y| \in |{\mf A}_{k+1}'| \subset |{\mf A}_k'|$, then by the induction hypothesis and Lemma \ref{lemma732}, $|x|$ and $|y|$ can be separated by two open subsets in $|{\mf A}_{k+1}'|$. Then by Lemma \ref{lemma733}, these two open sets are also open sets in $|{\mf A}_k'|$. Hence we assume that one or both of $|x|$ and $|y|$ are in $|{\mf A}_k'| \setminus |{\mf A}_{k+1}'|$. 

\vspace{0.2cm}

\noindent {\bf Case 1.} Suppose $|x|$ and $|y|$ are represented by $x, y \in U_{I_k}'$. Choose a distance function on $U_{I_k}$ which induces the same topology. Let $O_x^\epsilon$ and $O_y^\epsilon$ be the open $\epsilon$-balls in $U_{I_k}'$ centered at $x$ and $y$ respectively. Then for $\epsilon$ small enough, $\ov{O_x^\epsilon} \cap \ov{O_y^\epsilon} = \emptyset$. 

\vspace{0.2cm}

\noindent {\it Claim.} For $\epsilon$ sufficiently small, for all $I_k \preq I_i$ and $I_k \preq I_j$, one has
\beqn
\pi_{k+1} \big( \ov{\phi_{I_i I_k}(O_x^\epsilon \cap U_{I_i I_k}' )} \big) \cap \pi_{k+1} \big( \ov{\phi_{I_j I_k}(O_y^\epsilon \cap U_{I_j I_k}' )} \big) = \emptyset.
\eeqn
Here the closures are the closures in $U_{I_i}$ and $U_{I_j}$ respectively. 

\vspace{0.2cm}

\noindent {\it Proof of the claim.} Suppose it is not the case, then there exist a sequence $\epsilon_n \to 0$, $I_k \preq I_i$, $I_k \preq I_j$, and a sequence of points 
\beqn
|z^n| \in \pi_{k+1} \big( \ov{\phi_{I_i I_k} (O_x^{\epsilon_n}  \cap U_{I_i I_k}' )} \big) \cap \pi_{k+1} \big( \ov{\phi_{I_j I_k}(O_y^\epsilon \cap U_{I_j I_k}' )} \big) \subset |{\mf A}_{k+1}|.
\eeqn
Then $|z^n|$ has its representative $p^n \in \ov{\phi_{I_i I_k} (O_x^{\epsilon_n}  \cap U_{I_k I_i}' )} \subset U_{I_i}$ and its representative $q^n \in \ov{\phi_{I_j I_k} (O_y^{\epsilon_n}  \cap U_{I_j I_k}' )} \subset U_{I_j}$. Then $p^n \curlyvee q^n$ and without loss of generality, assume that $I_i \preq I_j$. Then $p^n \in U_{I_j I_i}$ and $q^n = \phi_{I_j I_i}(p^n)$. 

Choose distance functions $d_i$ on $U_{I_i}$ and $d_j$ on $U_{I_j}$ which induce the same topologies. Then one can choose $x^n \in O_x^{\epsilon_n} \cap U_{I_i I_k}' $ and $y^n \in O_y^{\epsilon_n} \cap U_{I_j I_k}'$ such that 
\begin{align}\label{eqn76}
&\ d_i ( p^n, \phi_{I_i I_k}(x^n) ) \leq \epsilon_n,\ &\ d( q^n, \phi_{I_i I_k}(y^n)) \leq \epsilon_n.
\end{align}
Since $\ov{U_{I_i}'}$ and $\ov{U_{I_j}'}$ are compact and $p^n \in \ov{U_{I_i}'}$, $q^n \in \ov{U_{I_j}'}$, for some subsequence (still indexed by $n$), $p^n$ converges to some $p^\infty \in \ov{U_{I_i}'}$ and $q^n$ converges to some $q^\infty \in \ov{U_{I_j}'}$. Then $p^\infty \curlyvee q^\infty$. Moreover, by \eqref{eqn76}, one has 
\begin{align*}
&\ \lim_{n \to\infty} \phi_{I_i I_k}(x^n) = p^\infty,\ &\ \lim_{n \to \infty} \phi_{I_j I_k}(y^n) = q^\infty.
\end{align*}
On the other hand, $x^n$ converges to $x$ and $y^n$ converges to $y$. By the property of coordinate changes, one has that $x \in U_{I_i I_k}$, $y \in U_{I_j I_k}$ and 
\beqn
x \curlyvee p^\infty \curlyvee q^\infty \curlyvee  y. 
\eeqn
Since $\curlyvee$ is an equivalence relation and it remains an equivalence relation after shrinking, $x \curlyvee y$, which contradicts $x \neq y$.  \hfill {\it End of the proof of the claim.}

Now choose such an $\epsilon$ and abbreviate $O_x = O_x^\epsilon$, $O_y = O_y^\epsilon$. Denote 
\begin{align*}
&\ P_{I_i} = \ov{\phi_{I_i I_k}(O_x \cap U_{I_i I_k}' )} \subset \ov{U_{I_i}'},\ &\ Q_{I_i} =  \ov{\phi_{I_i I_k}(O_y \cap U_{I_i I_k}' )} \subset \ov{U_{I_i}'}
\end{align*}
(which could be empty). They are all compact, hence  
\begin{align*}
&\ P_{k+1}:= \pi_{k+1}( \bigsqcup_{i \geq k+1} P_{I_i} )  \subset  |{\mf A}_{k+1}|,\ &\  Q_{k+1}:= \pi_{k+1}( \bigsqcup_{i \geq k+1} Q_{I_i})  \subset |{\mf A}_{k+1}|
\end{align*}
are both compact. The above claim implies that $P_{k+1} \cap Q_{k+1} = \emptyset$. Then by the induction hypothesis which says that $|{\mf A}_{k+1}|$ is Hausdorff, they can be separated by open sets $V_{k+1}, W_{k+1} \subset |{\mf A}_{k+1}|$. Write 
\begin{align*}
&\ \pi_{k+1}^{-1} \big( V_{k+1}\big) = \bigsqcup_{i \geq k+1} V_{I_i},\ &\ \pi_{k+1}^{-1} \big( W_{k+1} \big) = \bigsqcup_{i \geq k+1} W_{I_i}.
\end{align*}
Define $V_{I_i}' = V_{I_i} \cap U_{I_i}'$, $W_{I_i}' = W_{I_i} \cap U_{I_i}'$ and 
\begin{align*}
&\ V_{I_k}':= O_x \cup \bigcup_{I_k \preq I_i} \phi_{I_i I_k}^{-1}(V_{I_i}') \cap U_{I_k}',\ &\ W_{I_k}':= O_y \cup \bigcup_{I_k \preq I_i} \phi_{I_i I_k}^{-1}(W_{I_i}') \cap U_{I_k}'
\end{align*}
and 
\begin{align*}
&\ V_k':= \pi_k \big( \bigsqcup_{i \geq k} V_{I_i}' \big)\subset |{\mf A}_{k}'|,\ &\ W_k':= \pi_k \big( \bigsqcup_{i \geq k} W_{I_i}' \big) \subset |{\mf A}_{k}'|.
\end{align*}
It is easy to check that $V_k'$ and $W_k'$ are disjoint open subsets containing $|x|$ and $|y|$ respectively. Therefore $|x|$ and $|y|$ are separated in $|{\mf A}_k'|$.

\vspace{0.2cm}

\noindent {\bf Case 2.} Now suppose $|x|$ is represented by $x \in U_{I_k}'$ and $|y|\in |{\mf A}_{k+1}'|$. Similar to {\bf Case 1}, we claim that for $\epsilon$ sufficiently small, for all $I_k \preq I_i$, one has 
\beqn
|y| \notin \pi_{k+1} \big( \ov{\phi_{I_i I_k}(O_x^\epsilon) \cap U_{I_i I_k}'} \big)=: P_{I_i} \subset \ov{U_{I_i}'}.
\eeqn
The proof is similar and is omitted. Then choose such an $\epsilon$ and abbreviate $O_x = O_x^\epsilon$. $P_{I_i}$ are compact sets and so is  
\beqn
P_{k+1}:= \pi_{k+1} \big( \bigsqcup_{i \geq k+1} P_{I_i} \big) \subset |{\mf A}_{k+1}|.
\eeqn
The above claim implies that $|y| \notin P_{k+1}$. Then by the induction hypothesis, $|y|$ and $P_{k+1}$ can be separated by open sets $V_{k+1}$ and $W_{k+1}$ of $|{\mf A}_{k+1}|$. By similar procedure as in {\bf Case 1} above, one can produce two open subsets of $|{\mf A}_k'|$ which separate $|x|$ and $|y|$.  

Therefore, we can finish the induction and construct a shrinking such that $|{\mf A}'|$ is Hausdorff. Eventually the shrunk footprints still contain $\ov{F_I^\square}$. 
\end{proof}

Now we can finish proving Theorem \ref{thm728}. Suppose $|{\mf A}|$ is Hausdorff. By the definition of the quotient topology, the natural map $U_I \hookrightarrow |{\mf A}|$ is continuous. Since $U_I$ is locally compact and $|{\mf A}|$ is Hausdorff, a further shrinking can make this map a homeomorphism onto its image. This uses the fact that a continuous bijection from a compact space to a Hausdorff space is necessarily a homeomorphism. Hence the third condition for a good coordinate system is satisfied by a precompact shrinking of ${\mf A}$, and this condition is preserved for any further shrinking. This establishes Theorem \ref{thm728}.

\subsection{Perturbations}\label{subsection76}

Now we define the notion of perturbations. 

\begin{defn}\label{defn734}
Let ${\mf A}$ be a good coordinate system on $X$. 

\begin{enumerate}

\item A {\it multi-valued perturbation} of ${\mf A}$, simply called a {\it perturbation}, denoted by ${\mf t}$, consists of a collection of multi-valued continuous sections 
\beqn
t_I: U_I \overset{m}{\to} E_I
\eeqn
satisfying (as multisections)
\beqn
t_J \circ \phi_{JI} = \wh \phi_{JI} \circ t_I|_{U_{JI}}.
\eeqn

\item Given a multi-valued perturbation ${\mf t}$, the object
\beqn
\tilde {\mf s} = \Big( \tilde s_I = S_I + t_I: U_I \overset{m}{\to} E_I \Big)
\eeqn
satisfies the same compatibility condition with respect to coordinate changes. The perturbation ${\mf t}$ is called {\it transverse} if every $\tilde s_I$ is a transverse multisection. 

\item Suppose ${\mf A}$ is thickened by ${\mc N} = \{(N_{JI}, E_{I;J})\ |\ I \preq J\}$. We say that ${\mf t}$ is {\it ${\mc N}$-normal} if for all $I \preq J$, one has
\beq\label{eqn77}
t_J(N_{JI}) \subset E_{I; J}|_{N_{JI}}. 
\eeq

\item The zero locus of a perturbed $\tilde {\mf s}$ gives objects in various different categories. Denote
\beqn
{\mc Z} = \bigsqcup_{I \in {\mc I}}  \tilde s_I^{-1}(0).
\eeqn
It is naturally equipped with the topology induced from the disjoint union of $U_I$. Denote by
\beqn
|{\mc Z}|:= {\mc Z}/ \curlyvee. 
\eeqn
the quotient of ${\mc Z}$, which is equipped with the quotient topology. Furthermore, there is a natural injection $|{\mc Z}|  \hookrightarrow |{\mf A}|$. Denote by $\| {\mc Z}\|$ the same set as $|{\mc Z}|$ but equipped with the topology as a subspace of $|{\mf A}|$. 
\end{enumerate}
\end{defn}

In order to construct suitable perturbations of a good coordinate system, we need certain tubular neighborhood structures with respect to coordinate changes. In our topological situation, it is sufficient to have some weaker structure near the embedding images.

\begin{defn}\label{defn735}
Let ${\mf A}$ be a good coordinate system with charts indexed by elements in a finite partially ordered set $({\mc I}, \preq)$ and coordinate changes indexed by pairs $I \preq J \in {\mc I}$. A {\it thickening} of ${\mf A}$ is a collection of objects 
\beqn
\big\{ (N_{JI}, E_{I; J})\ |\ I \preq J \big\}
\eeqn
where $N_{JI} \subset U_J$ is an open neighborhood of $\phi_{JI}(U_{JI})$ and $E_{I; J}$ is a subbundle of $E_J|_{N_{JI}}$. They are required to satisfy the following conditions. 
\begin{enumerate}

\item If $I \preq K$, $J \preq K$ but there is no partial order relation between $I$ and $J$, then 
\beq\label{eqn78}
N_{KI} \cap N_{KJ} = \emptyset.
\eeq

\item For all triples $I \preq J \preq K$, 
\beqn
E_{I;J}|_{\phi_{KJ}^{-1}(N_{KI}) \cap N_{JI} } = \wh\phi_{KJ}^{-1}(E_{I;K})|_{\phi_{KJ}^{-1}(N_{KI}) \cap N_{JI}}.
\eeqn

\item For all triples $I \preq J \preq K$, one has
\beqn
E_{I; K}|_{N_{KI} \cap N_{KJ}} \subset E_{J; K}|_{N_{KI} \cap N_{KJ}}.
\eeqn

\item Each $(N_{JI}, E_{I;J})$ satisfies the {\bf (Tangent Bundle Condition)} of Definition \ref{defn720}.
\end{enumerate}
\end{defn}

\begin{rem}\label{rem736}
The above setting is slightly more general than what we need in our application in this paper and the companion \cite{Tian_Xu_4}. In this paper we will see the following situation in the concrete cases.
\begin{enumerate}

\item The index set ${\mc I}$ consists of certain nonempty subsets of a finite set $\{1, \ldots, m \}$, which has a natural partial order given by inclusions. 

\item For each $i \in I$, $\Gammait_i$ is a finite group and $\Gammait_I =  \mathit{\Pi}_{i\in I} \Gammait_i$. $U_I = \tilde U_I/ \Gammait_I$ where $\tilde U_I$ is a topological manifold acted by $\Gammait_I$. Moreover, ${\bm E}_1, \ldots, {\bm E}_m$ are vector spaces acted by $\Gammait_i$ and the orbifold bundle $E_I \to U_I$ is the quotient 
\beqn
E_I:= (\tilde U_I \times {\bm E}_I ) / \Gammait_I,\ {\rm where}\ {\bm E}_I:= \bigoplus_{i \in I} {\bm E}_i.
\eeqn

\item For $I \preq J$, $U_{JI} = \tilde U_{JI} / \Gammait_I$ where $\tilde U_{JI} \subset \tilde U_I$ is a $\Gammait_I$-invariant open subset and the coordinate change is induced from the following diagram
\beq\label{eqn79}
\vcenter{ \xymatrix{ \tilde V_{JI} \ar[r] \ar[d] & \tilde U_J \\
                     \tilde U_{JI}              & }   }
\eeq
Here $\tilde V_{JI} \to \tilde U_{JI}$ is a covering space with group of deck transformations identical to $\Gammait_{J-I} = \mathit{\Pi}_{j \in J - I} \Gammait_j$; then $\Gammait_J$ acts on $\tilde V_{JI}$ and $\tilde V_{JI} \to \tilde U_J$ is a $\Gammait_J$-equivariant embedding of manifolds, which induces an orbifold embedding $U_{JI} \to U_J$ and an orbibundle embedding $E_I|_{U_{JI}} \to E_J$. 

\end{enumerate}

In this situation, one naturally has subbundles $E_{I; J} \subset E_J$ for all pairs $I \preq J$. Hence a thickening of such a good coordinate system is essentially only a collection of neighborhoods $N_{JI}$ of $\phi_{JI}(U_{JI})$ which satisfy \eqref{eqn78} and 
\beqn
S_J^{-1}(E_{I;J}) \cap N_{JI} = \phi_{JI}(U_{JI}).
\eeqn
\end{rem}

\begin{thm}\label{thm737}
Let $X$ be a compact Hausdorff space and have a good coordinate system 
\beqn
{\mf A} = \Big( \big\{ C_I = (U_I, E_I, S_I, \psi_I, F_I) \ |\ I \in {\mc I} \big\},\ \big\{ T_{JI}  \ |\ I \preq J \big\} \Big).
\eeqn
Let ${\mf A}' \sqsubset {\mf A}$ be any precompact shrinking. Let ${\mc N} = \{(N_{JI}, E_{I;J})\}$ be a thickening of ${\mf A}$ and $N_{JI}'\subset U_J'$ be a collection of open neighborhoods of $\phi_{JI}(U_{JI}')$ such that $\ov{N_{JI}'} \subset N_{JI}$. Then they induce a thickening ${\mc N}'$ of ${\mf A}'$ by restriction. Let $d_I: U_I \times U_I \to [0, +\infty)$ be a distance function on $U_I$ which induces the same topology as $U_I$. Let $\epsilon>0$ be a constant. Then there exist a collection of multisections $t_I: U_I \overset{m}{\to} E_I$ satisfying the following conditions.
\begin{enumerate}

\item For each $I\in {\mc I}$, $\tilde s_I:= S_I + t_I$ is transverse.

\item For each $I \in {\mc I}$, 
\beq\label{eqn710}
d_I \big( \tilde s_I^{-1}(0) \cap \ov{U_I'}, S_I^{-1}(0) \cap \ov{U_I'} \big) \leq \epsilon.
\eeq

\item For each pair $I \preq J$, we have 
\beqn
\wh \phi_{JI} \circ t_I|_{U_{JI}'} = t_J \circ \phi_{JI}|_{U_{JI}'}.
\eeqn
Hence the collection of restrictions $t_I':= t_I|_{U_I'}$ defines a perturbation ${\mf t}'$ of ${\mf A}'$. 

\item ${\mf t}'$ is ${\mc N}'$-normal. 

\end{enumerate}
\end{thm}

\begin{proof}
To simplify the proof, we assume that we are in the situation described by Remark \ref{rem736}. The general case requires minor modifications in a few places. Then the subbundles $E_{I_b, I_a}$ are naturally define over $U_{I_b}$. 

We use the inductive construction. Order the set ${\mc I}$ as $I_1, \ldots, I_m$ such that 
\beqn
I_k \preq I_l \Longrightarrow k \leq l.
\eeqn
For each $k, l$ with $k < l$, define open sets $N_{I_l, k}^-$ by 
\beqn
N_{I_l, k}^- = \bigcup_{a \leq k, I_a \prec I_l} N_{I_l I_a}.
\eeqn
Define open sets $N_{I_l, k}^+$ inductively. First $N_{I_m,k}^+ =\emptyset$. Then 
\begin{align*}
&\ N_{I_l, k}^+ := \bigcup_{I_l \prec I_b } \phi_{I_b I_l}^{-1}(N_{I_b, k}),\ &\ N_{I_l, k} = N_{I_l,k}^- \cup N_{I_l, k}^+.
\end{align*}
Replacing $N_{JI}$ by $N_{JI}'$ in the above definitions, we obtain $N_{I_l, k}'\subset U_{I_l}'$ with 
\beqn
\ov{N_{I_l, k}'} \subset N_{I_l, k}.
\eeqn
If $k \geq l$, define 
\beqn
N_{I_l, k} = U_{I_l},\ N_{I_l, k}' = U_{I_l}'.
\eeqn

On the other hand, it is not hard to inductively choose a system of continuous norms on $E_I$ such that, for all pairs $I_a \preq I_b$, the bundle embedding $\wh\phi_{I_b I_a}$ is isometric. Given such a collection of norms, choose $\delta>0$ such that for all $I$, 
\beq\label{eqn711}
d_I\big( x, S_I^{-1}(0) \cap \ov{U_I'} \big) > \epsilon,\ x \in \ov{U_I'} \Longrightarrow \| S_I(x)\| \geq (2m+1) \delta.
\eeq

Now we reformulate the problem in an inductive fashion. We would like to verify the following induction hypothesis.

\vspace{0.2cm}

\noindent {\it Induction Hypothesis.} For $a, k = 1, \ldots, m$, there exists an open subset $O_{I_a, k} \subset N_{I_a, k}$ which contains $\ov{N_{I_a, k}'}$ and multisections
\beqn
t_{I_a, k}: O_{I_a, k} \overset{m}{\to} E_{I_a}.
\eeqn
They satisfy the following conditions. 

\begin{enumerate}
\item For all pairs $I_a \preq I_b$, over a neighborhood of the compact subset
\beqn
\ov{N_{I_a,k}'} \cap \phi_{I_b I_a}^{-1}( \ov{N_{I_b, k}'}) \subset \ov{U_{I_b I_a}'} \subset U_{I_b I_a}
\eeqn
one has
\beq\label{eqn712}
t_{I_b, k} \circ \phi_{I_b I_a} = \wh\phi_{I_b I_a} \circ t_{I_a, k}.
\eeq

\item In a neighborhood of $\ov{N_{I_b I_a}'}$, the value of $t_{I_b, k}$ is contained in the subbundle $E_{I_a; I_b}$. 

\item $S_{I_a} + t_{I_a, k}$ is transverse.

\item $\| t_{I_a, k}\|_{C^0} \leq 2 k \delta$. 
\end{enumerate}
\vspace{0.2cm}

It is easy to see that the $k = m$ case implies this theorem. Indeed, \eqref{eqn710} follows from \eqref{eqn711} and the bound $\| t_{I_l, m}\| \leq m \delta$. Now we verify the conditions of the induction hypothesis. For the base case, apply Lemma \ref{lemma717} to 
\beqn
M = U_{I_1},\ \ C = \emptyset,\ \ D = U_{I_1},
\eeqn
we can construct a multisection $t_{I_1, 1}: U_{I_1} \overset{m}{\to} E_{I_1}$ making $S_{I_1} + t_{I_1}$ transverse with $\| t_{I_1}\| \leq  \delta$. Now we construct $t_{I_a, 1}$ for $a = 2, \ldots, m$ via a backward induction. Define 
\beqn
O_{I_a, 1} = N_{I_a, 1} = N_{I_a I_1} \cup \bigcup_{I_a \prec I_b} \phi_{I_b I_a}^{-1}(N_{I_b, 1}),\ a = 1, \ldots, m.
\eeqn
Then \eqref{eqn712} determines the value of $t_{I_m, 1}$ over the set
\beqn
\phi_{I_m I_1}(U_{I_m I_1}) \subset N_{I_m, 1} = N_{I_m I_1}.
\eeqn
It is a closed subset of $N_{I_m, 1}$, hence one can extend it to a continuous section of $E_{I_1; I_m}$ which can be made satisfy the bound 
\beqn
\| t_{I_m, 1}\|\leq  \left( 1 + \frac{1}{m} \right) \delta.
\eeqn
Suppose we have constructed $t_{I_a, 1}: O_{I_a, 1} \overset{m}{\to} E_{I_a}$ for all $a \geq l+1$ such that together with $t_{I_1, 1}$ they satisfy the induction hypothesis for $k = 1$ with the bound
\beqn
\| t_{I_a, 1}\| \leq \left( 1 + \frac{m-l}{m} \right) \delta.
\eeqn
Then we construct $t_{I_l, 1}: O_{I_l, 1} \overset{m}{\to } E_{I_l}$ as follows. Given
\beqn
z_{I_l} \in \phi_{I_l I_1}(U_{I_l I_1}) \cup N_{I_l, 1}^+ = \phi_{I_l I_1}(U_{I_l I_1}) \cup  \bigcup_{I_l \prec I_b} \phi_{I_b I_l}^{-1}(N_{I_b, 1}), 
\eeqn
if $z_{I_l}$ is in the first component, then define $t_{I_l, 1}(z_{I_l})$ by the formula \eqref{eqn712} for $b=l$, $a=1$. If $z_{I_l} \in N_{I_b I_l} \cap \phi_{I_b I_l}^{-1}(N_{I_b, 1})$ for some $b$, then define
\beqn
t_{I_l, 1}(z_{I_l}) = \wh\phi_{I_b I_l}^{-1}( t_{I_b, 1}(\phi_{I_b I_l}(z_{I_l}))).
\eeqn
It is easy to verify using the {\bf (Cocycle Condition)} that these definitions agree over some closed neighborhood of 
\beqn
\phi_{I_l I_1}( \ov{U_{I_l I_1}'}) \cup \bigcup_{I_l \prec I_b} \phi_{I_b I_l}^{-1}( \ov{N_{I_b I_l}'}).
\eeqn
Then one can extend it to a continuous multisection of $E_{I_1; I_l}$ satisfying the bound
\beqn
\| t_{I_l, 1} \| \leq \left( 1 + \frac{m-l+1}{m} \right) \delta.
\eeqn
Then the induction can be carried on and stops until $l = 2$, for which one has the bound
\beqn
\| t_{I_2, 1} \| \leq \left( 1 + \frac{m-1}{m} \right) \delta \leq 2 \delta.
\eeqn
The transversality of $S_{I_a} + t_{I_a, 1}$ for $a \geq 2$ follows from the fact that $t_{I_a, 1}$ takes value in $E_{I_1; I_a}$, the fact that $S_{I_a}|_{N_{I_a I_1}}$ intersects with $E_{I_1; I_a}$ transversely along $\phi_{I_a I_1}(U_{I_a I_1})$, and the fact that $S_{I_1} + t_{I_1, 1}$ is transverse. Hence we have verified the $k=1$ case of the induction hypothesis.

Suppose we have verified the induction hypothesis for $k-1$. For all $a \leq k-1$, define
\begin{align*}
&\ O_{I_a, k} = O_{I_a, k-1},\ &\ t_{I_a, k} = t_{I_a, k-1}.
\end{align*}
The induction hypothesis implies that we have a multisection 
\beqn
t_{I_k, k-1}: O_{I_k, k-1}\overset{m}{\to} E_{I_k}
\eeqn
such that $S_{I_k} + t_{I_k, k-1}$ is transverse and $\| t_{I_k, k-1}\| \leq (2k-2)\delta$. Then apply Lemma \ref{lemma717}, one can obtain a multisection $t_{I_k, k}$ defined over a neighborhood of $\ov{U_{I_k}'}  = \ov{N_{I_k, k}'}$ contained in $U_{I_k} = N_{I_k, k}$ such that $S_{I_k} + t_{I_k, k}$ is transverse, $\| t_{I_k, k}\| \leq (2k-1) \delta$, and $t_{I_k, k} = t_{I_k, k-1}$ over a neighborhood of $\ov{N_{I_k, k-1}'}$ which is smaller than $O_{I_k,k-1}$. Then by the similar backward induction as before, using the extension property of continuous multi-valued functions, one can construct perturbations with desired properties. The remaining details are left to the reader.
\end{proof}

In our argument, condition \eqref{eqn710} is crucial in establishing the compactness of the perturbed zero locus. In the situation of Theorem \ref{thm737}, suppose a perturbation ${\mf t}'$ is constructed over the shrinking ${\mf A}'\sqsubset {\mf A}''$. Then for a further precompact shrinking ${\mf A}'' \sqsubset {\mf A}'$, \eqref{eqn710} remains true (with $\ov{U_I'}$ replaced by $\ov{U_I''}$).

\begin{prop}\label{prop738}
Let ${\mf A}$ be a good coordinate system on $X$ and let ${\mf A}' \sqsubset {\mf A}$ be a precompact shrinking. Let ${\mc N}$ be a thickening of ${\mf A}$. Equip each chart $U_I$ a distance function $d_I$ which induces the same topology. Then there exists $\epsilon>0$ satisfying the following conditions. Let ${\mf t}$ be a multi-valued perturbation of ${\mf s}$ which is ${\mc N}$-normal. Suppose
\beq\label{eqn713}
d_I \big( \tilde s_I^{-1}(0) \cap \ov{U_I'}, S_I^{-1}(0)\cap \ov{U_I'} \big) \leq \epsilon,\ \forall I \in {\mc I}.
\eeq
Then the zero locus $\| (\tilde {\mf s}')^{-1} (0)\|$ is sequentially compact with respect to the subspace topology induced from $|{\mf A}'|$.
\end{prop}

\begin{proof}
Now for each $x \in |{\mf A}'|$, define $I_x \in {\mc I}$ to be the minimal element for which $x$ can be represented by a point $\tilde x \in U_{I_x}'$. 

\vspace{0.2cm}

\noindent {\it Claim.} Given $I$, there exists $\epsilon>0$ such that, for any perturbations $\tilde {\mf s}$ that satisfy \eqref{eqn713}, if $x_i \in \| (\tilde {\mf s}')^{-1}(0)\|$ with $I_{x_i} = I$ for all $i$, then $x_i$ has a convergent subsequence.

\vspace{0.2cm}

\noindent {\it Proof of the claim.} Suppose this is not true, then there exist a sequence $\epsilon_k >0$ that converge to zero, and a sequence of multi-valued perturbations ${\mf t}_k$ satisfying 
\beq\label{eqn714}
d_I \big( \tilde s_{k, J}^{-1}(0) \cap \ov{U_J'}, S_J^{-1}(0) \cap \ov{U_J'} \big) \leq \epsilon_k,\ \forall J \in {\mc I}
\eeq
and sequences of points $\tilde x_{k, i} \in \tilde s_{k,I}^{-1}(0) \cap U_I'$ such that the sequence $\{ x_{k, i} = \pi_{\mf A}(\tilde x_{k, i}) \}_{i=1}^\infty$ does not have a convergent subsequence. Since $\tilde x_{k, i} \in \ov{U_I'}$ which is a compact subset of $U_I$, for all $k$ the sequence $\tilde x_{k, i}$ has subsequential limits, denoted by $\tilde x_k \in \tilde s_{k, I}^{-1}(0) \cap \ov{U_I'}$. Then since $\epsilon_k \to 0$, the sequence $\tilde x_k$ has a subsequential limit $\tilde x_\infty \in \ov{U_I'} \cap S_I^{-1}(0)$. Denote 
\beqn
x_\infty = \psi_I( \tilde x_\infty) \in \psi_I ( \ov{U_I'} \cap S_I^{-1}(0)) =  \ov{F_I'} \subset F_I \subset X.
\eeqn
Since all $F_I'$ cover $X$, there exists $J \in {\mc I}$ such that $x_\infty \in F_J$. Then by the {\bf (Overlapping Condition)} of the atlas ${\mf A}'$, we have either $I \preq J$ or $J \preq I$ but $J \neq I$. We claim that the latter is impossible. Indeed, if $x_\infty = \psi_J (\tilde y_\infty)$ with $\tilde y_\infty \in U_J'$, then we have $\tilde y_\infty \in U_{IJ}'$ and $\tilde x_\infty \in \varphi_{IJ}'(U_{IJ}')$. Then for $k$ sufficiently large, we have $\tilde x_k$ in $N_{JI}$. Fix such a large $k$, then for $i$ sufficiently large, we have $\tilde x_{k, i} \in N_{JI}$. However, since the perturbation ${\mf t}$ is ${\mc N}$-normal, it follows that $\tilde x_{k, i} \in \varphi_{IJ}'(U_{IJ}')$. This contradicts the assumption that $I_{x_{k, i}} = I$. Therefore $I \preq J$. Then there is a unique $\tilde y_\infty \in U_J' \cap S_J^{-1}(0)$ such that $\psi_J'( \tilde y_\infty) = x_\infty$ and $\tilde y_\infty = \varphi_{JI} ( \tilde x_\infty)$. Then $\tilde x_\infty \in U_{JI}$. Therefore, for $k$ sufficiently large, we have $\tilde x_k \in U_{JI}$ and we have the convergence
\beqn
\tilde y_k:= \varphi_{JI}(\tilde x_k) \to \tilde y_\infty
\eeqn
since $\varphi_{JI}$ is continuous. Since $\tilde y_\infty \in U_J'$ which is an open subset of $U_J$, for $k$ sufficiently large, we have $\tilde y_k \in U_J'$. Fix such a large $k$. Then for $i$ sufficiently large, we have
\beqn
\tilde x_{k, i} \in U_{JI} \cap \varphi_{JI}^{-1}(U_J') \cap U_I'.
\eeqn
Hence we have $\tilde y_{k, i}:= \varphi_{JI}'(\tilde x_{k, i}) \in U_J'$ and 
\beqn
\lim_{i \to \infty} \tilde y_{k, i} = \tilde y_k.
\eeqn
Since the map $U_J' \to |{\mf A}'|$ is continuous, we have the convergence 
\beqn
\lim_{i \to \infty} x_{k, i} = \lim_{i \to \infty} \pi_{{\mf A}'} ( \tilde y_{k, i}) = \pi_{{\mf A}'} (\tilde y_\infty).  
\eeqn
This contradicts the assumption that $x_{k, i}$ does not converge for all $k$. Hence the claim is prove. \hfill {\it End of the proof of the claim.}

\vspace{0.2cm}

Now for all $I \in {\mc I}$, choose the smallest $\epsilon$ such that the condition of the above claim hold. We claim this $\epsilon$ satisfies the condition of this proposition. Indeed, let ${\mf t}$ be such a perturbation and let $x_k$ be a sequence of points in $\| (\tilde s')^{-1}(0)\|$. Then there exists an $I \in {\mc I}$ and a subsequence (still indexed by $k$) with $I_{x_k} = I$. Then by the above claim, $x_k$ has a subsequential limit. Therefore $\| (\tilde s')^{-1}(0)\|$ is sequentially compact. 
\end{proof}

\subsection{The virtual cardinality and invariance}

In the application of the virtual technique in this paper, we only consider two situations: 1) the index of the problem is zero; 2) the index of the problem is zero or one and the charts are all manifolds. Then the discussion of the topology of the perturbed zero locus is very simple. In the general situation, for a transverse multi-valued perturbation, the perturbed zero locus has the structure of a {\it weighted branched manifold}. 

Indeed, let ${\mf A}$ be a good coordinate system of virtual dimension zero and let ${\mf t}$ be a transverse multi-valued perturbation for which the perturbed zero locus is sequentially compact. Let $|x| \in |{\mf A}|$ be a zero. Then over each chart $C_I$ which contains a representative $x_I\in U_I$ of $|x|$, we have a local representative
\beqn
\tilde s_I(x) = [ s_I^1(x), \ldots, s_I^l(x)].
\eeqn
The orientation of the atlas provides a number $\epsilon_i \in \{1, -1, 0\}$, such that if $s_I^i(x_I) \neq 0$, then $\epsilon_i = 0$. Then define the multiplicity of this zero $x_I$ by 
\beqn
m (|x|) = \frac{1}{l} \sum_{i=1}^l \epsilon_i.
\eeqn
One can verify that $m(x_I)$ is independent of the local representative, and is independent of the representative $x_I$ of $|x|$. Then we define the {\it virtual cardinality} of ${\mf A}$ by
\beqn
\# {\mf A} = \sum_{|x|\in |{\mc Z}_{\mf t}|} m(|x|) \in {\mb Q}.
\eeqn
The compactness of $|{\mc Z}_{\mf t}|$ implies that this sum is finite and hence $\#{\mf A}$ is a finite number. Moreover, if the charts are all manifolds, then transversality can be achieved by single-valued perturbations. Therefore, in that case, the virtual cardinality is an integer. 

Lastly we expect that the virtual cardinality is independent of various choices. This requires to consider virtual dimension one case, and we assume that the charts are all manifolds. Consider a good coordinate system ${\mf A}$ of virtual dimension one with boundary on a space $X$. Then there is a natural closed subset $\partial X \subset X$ and a boundary atlas $\partial {\mf A}$ on $\partial X$ which has virtual dimension zero without boundary. If ${\mf A}$ is oriented, then $\partial {\mf A}$ has an induced orientation. 

\begin{prop}
Let ${\mf A}$ be an oriented good coordinate system of dimension 1 with boundary on a space $X$. Suppose all charts of ${\mf A}$ are manifolds. Then $\# (\partial {\mf A}) = 0$. 
\end{prop}

\begin{proof}
Over $\partial {\mf A}$ one can find a transverse single-valued perturbation whose perturbed zero locus is a compact zero-dimensional manifold. The counting with signs of the zeroes is equal to $\# (\partial {\mf A})$. Moreover, the topological transversality theorem (Theorem \ref{thm77}) allows us to extend the perturbation to a single-valued transverse perturbation on ${\mf A}$, whose perturbed zero locus is still compact. Then the perturbed zero locus of ${\mf A}$ is a compact oriented one-dimensional manifold whose boundary is the perturbed zero locus of $\partial {\mf A}$, with the induced orientation. Hence $\# (\partial {\mf A}) = 0$. 
\end{proof}

The general method to prove the virtual cardinality is independent of choices is to construct a good coordinate system on the product $X \times [0, 1]$, such that two different systems of choices can be made homotopic to each other. Then the above proposition implies the independence.

\bibliography{mathref}

\newcommand{\etalchar}[1]{$^{#1}$}
\providecommand{\bysame}{\leavevmode\hbox to3em{\hrulefill}\thinspace}
\providecommand{\MR}{\relax\ifhmode\unskip\space\fi MR }
\providecommand{\MRhref}[2]{%
  \href{http://www.ams.org/mathscinet-getitem?mr=#1}{#2}
}
\providecommand{\href}[2]{#2}
\begin{thebibliography}{{Mun}03}

\bibitem[CFG{\etalchar{+}}]{CFGKS}
Ionu\c{t} {Ciocan-Fontaine}, David Favero, J\'r\'emy Gu\'er\'e, Bumsig Kim, and
  Mark Shoemaker, \emph{Fundamental factorization of a {GLSM}, part {I}:
  construction}, \url{https://arxiv.org/abs/1802.05247}.

\bibitem[CGS00]{Cieliebak_Gaio_Salamon_2000}
Kai Cieliebak, Ana Gaio, and Dietmar Salamon, \emph{${J}$-holomorphic curves,
  moment maps, and invariants of {H}amiltonian group actions}, International
  Mathematics Research Notices \textbf{16} (2000), 831--882.

\bibitem[CLLL15]{CLLL_15}
Huai-Liang Chang, Jun Li, Wei-Ping Li, and Chiu-Chu Liu, \emph{Mixed-spin-p
  fields of {F}ermat quintic polynomials},
  \url{http://arxiv.org/abs/1505.07532}, 2015.

\bibitem[CLW]{Chen_Li_Wang_1}
Bohui Chen, {An-Min} Li, and {Bai-Ling} Wang, \emph{Virtual neighborhood
  technique for pseudo-holomorphic spheres},
  \url{https://arxiv.org/abs/1306.3276}.

\bibitem[FH93]{Floer_Hofer_orientation}
Andreas Floer and Helmut Hofer, \emph{Coherent orientations for periodic orbit
  problems in symplectic geometry}, Mathematische Zeitschrift \textbf{212}
  (1993), no.~1, 13--38.

\bibitem[FJR]{FJR3}
Huijun Fan, Tyler Jarvis, and Yongbin Ruan, \emph{The {W}itten equation and its
  virtual fundamental cycle}, \url{http://arxiv.org/abs/0712.4025}.

\bibitem[FJR08]{FJR1}
\bysame, \emph{Geometry and analysis of spin equations}, Communications on Pure
  and Applied Mathematics \textbf{61} (2008), no.~6, 745--788.

\bibitem[FJR13]{FJR2}
\bysame, \emph{The {W}itten equation, mirror symmetry and quantum singularity
  theory}, Annals of {M}athematics \textbf{178} (2013), 1--106.

\bibitem[FJR18]{FJR_GLSM}
\bysame, \emph{A mathematical theory of gauged linear $\sigma$-model}, Geometry
  and Topology \textbf{22} (2018), 235--303.

\bibitem[FO99]{Fukaya_Ono}
Kenji Fukaya and Kaoru Ono, \emph{Arnold conjecture and {G}romov--{W}itten
  invariants for general symplectic manifolds}, Topology \textbf{38} (1999),
  933--1048.

\bibitem[FOOOa]{FOOO_2015}
Kenji Fukaya, Yong-Geun Oh, Hiroshi Ohta, and Kaoru Ono, \emph{Kuranishi
  structure, pseudo-holomorphic curve, and virtual fundamental chain: part 1},
  \url{http://arxiv.org/abs/1503.07631}.

\bibitem[FOOOb]{FOOO_2017}
\bysame, \emph{Kuranishi structure, pseudo-holomorphic curve, and virtual
  fundamental chain: part 2}, \url{https://arxiv.org/abs/1704.01848}.

\bibitem[FOOO12]{FOOO_2012}
\bysame, \emph{Technical details on {K}uranishi structure and virtual
  fundamental chain}, \url{https://arxiv.org/abs/1209.4410}, 2012.

\bibitem[FOOO16]{FOOO_2016}
\bysame, \emph{Shrinking good coordinate systems associated to {K}uranishi
  structures}, Journal of Symplectic Geometry \textbf{14} (2016), no.~4,
  1295--1310.

\bibitem[FQ90]{Freedman_Quinn}
Michael Freedman and Frank Quinn, \emph{Topology of 4-manifolds}, Princeton
  University Press, 1990.

\bibitem[HV00]{Hori_Vafa}
Kentaro Hori and Cumrun Vafa, \emph{Mirror symmetry},
  \url{http://arxiv.org/abs/hep-th/0002222}, 2000.

\bibitem[HWZ07]{HWZ1}
Helmut Hofer, Krzysztof Wysocki, and Eduard Zehnder, \emph{A general {F}redholm
  theory. {I}. {A} splicing-based differential geometry}, Journal of European
  Mathematical Society \textbf{9} (2007), 841--876.

\bibitem[HWZ09a]{HWZ2}
\bysame, \emph{A general {F}redholm theory {I}{I}: implicit function theorems},
  Geometric and Functional Analysis \textbf{19} (2009), 206--293.

\bibitem[HWZ09b]{HWZ3}
\bysame, \emph{A general {F}redholm theory {I}{I}{I}: {F}redholm functors and
  polyfolds}, Geometry and Topology \textbf{13} (2009), 2279--2387.

\bibitem[Joya]{Joyce_07}
Dominic Joyce, \emph{Kuranishi homology and {K}uranishi cohomology},
  \url{http://arxiv.org/abs/0707.3572}.

\bibitem[Joyb]{Joyce_14}
\bysame, \emph{A new definition of {K}uranishi space},
  \url{https://arxiv.org/abs/1409.6908}.

\bibitem[Kis64]{Kister_1964}
James Kister, \emph{Microbundles are fibre bundles}, Annals of Mathematics
  \textbf{80} (1964), no.~1, 190--199.

\bibitem[KS77]{Kirby_Siebenmann}
Robion Kirby and Laurence Siebenmann, \emph{Foundational essays on topological
  manifolds, smoothings, and triangulations}, Princeton University Press and
  University of Tokyo Press, 1977.

\bibitem[LT98a]{Li_Tian_2}
Jun Li and Gang Tian, \emph{Virtual moduli cycles and {G}romov--{W}itten
  invariants of algebraic varieties}, Journal of American Mathematical Society
  \textbf{11} (1998), no.~1, 119--174.

\bibitem[LT98b]{Li_Tian}
\bysame, \emph{Virtual moduli cycles and {G}romov--{W}itten invariants of
  general symplectic manifolds}, Topics in symplectic $4$-manifolds (Irvine,
  CA, 1996), First International Press Lecture Series., no.~I, International
  Press, Cambridge, MA, 1998, pp.~47--83.

\bibitem[LT98c]{Liu_Tian_Floer}
Gang Liu and Gang Tian, \emph{Floer homology and {A}rnold conjecture}, Journal
  of Differential Geometry \textbf{49} (1998), 1--74.

\bibitem[Maz64]{Mazur_1964}
Barry Mazur, \emph{The method of infinite repetition in pure topology: {I}},
  Annals of Mathematics \textbf{80} (1964), no.~2, 201--226.

\bibitem[Mil64]{Milnor_micro_1}
John Milnor, \emph{Microbundles {Part I}}, Topology \textbf{3} (1964),
  no.~Suppl. 1, 53--80.

\bibitem[MS04]{McDuff_Salamon_2004}
Dusa McDuff and Dietmar Salamon, \emph{${J}$-holomorphic curves and symplectic
  topology}, Colloquium Publications, vol.~52, American Mathematical Society,
  2004.

\bibitem[MT09]{Mundet_Tian_2009}
Ignasi {Mundet i Riera} and Gang Tian, \emph{A compactification of the moduli
  space of twisted holomorphic maps}, Advances in Mathematics \textbf{222}
  (2009), 1117--1196.

\bibitem[{Mun}99]{Mundet_thesis}
Ignasi {Mundet i Riera}, \emph{Yang--{M}ills--{H}iggs theory for symplectic
  fibrations}, Ph.D. thesis, Universidad Aut\'onoma de Madrid, 1999.

\bibitem[{Mun}03]{Mundet_2003}
\bysame, \emph{Hamiltonian {G}romov--{W}itten invariants}, Topology \textbf{43}
  (2003), no.~3, 525--553.

\bibitem[MW15]{MW_1}
Dusa McDuff and Katrin Wehrheim, \emph{The fundamental class of smooth
  {K}uranishi atlases with trivial isotropy}, Journal of Topology and Analysis
  \textbf{Published Online} (2015), 1--173.

\bibitem[MW17a]{MW_2}
\bysame, \emph{Smooth {K}uranishi atlases with isotropy}, Geometry and Topology
  \textbf{21} (2017), 2725--2809.

\bibitem[MW17b]{MW_3}
\bysame, \emph{The topology of {K}uranishi atlases}, Proceedings of London
  Mathematical Society \textbf{115} (2017), no.~3, 221--292.

\bibitem[Par16]{Pardon_virtual}
John Pardon, \emph{An algebraic approach to virtual fundamental cycles on
  moduli spaces of pseudo-holomorphic curves}, Geometry and Topology
  \textbf{20} (2016), 779--1034.

\bibitem[Qui82]{Quinn_1982}
Frank Quinn, \emph{Ends of maps. {III.} dimension 4 and 5}, Journal of
  Differential Geometry (1982).

\bibitem[Qui88]{Quinn}
\bysame, \emph{Topological transversality holds in all dimensions}, Bulletin of
  the American Mathematical Society \textbf{18} (1988), 145--148.

\bibitem[RT95]{Ruan_Tian}
Yongbin Ruan and Gang Tian, \emph{A mathematical theory of quantum cohomology},
  Journal of Differential Geometry \textbf{42} (1995), 259--367.

\bibitem[RT97]{Ruan_Tian_97}
\bysame, \emph{Higher genus symplectic invariants and sigma model coupled with
  gravity}, Inventiones Mathematicae \textbf{130} (1997), 455--516.

\bibitem[Sat56]{Satake_orbifold}
Ichiro Satake, \emph{On a generalization of the notion of manifold},
  Proceedings of National Academy of Sciences \textbf{42} (1956), 359--363.

\bibitem[Sol05]{Solomon_preprint}
Jake Solomon, \emph{The {W}itten complex for non-{M}orse functions}, preprint,
  2005.

\bibitem[SX14]{Lagrange_multiplier}
Stephen Schecter and Guangbo Xu, \emph{Morse theory for {L}agrange multipliers
  and adiabatic limits}, Journal of Differential Equations \textbf{257} (2014),
  4277--4318.

\bibitem[TX]{Tian_Xu_4}
Gang Tian and Guangbo Xu, \emph{A wall-crossing formula for the correlation
  function of gauged linear $\sigma$-model}, preprint.

\bibitem[TX2]{Tian_Xu_geometric}
\bysame, \emph{Gauged linear sigma model in geometric phases},
  \url{https://arxiv.org/abs/1809.00424}, 2.

\bibitem[TX15]{Tian_Xu}
\bysame, \emph{Analysis of gauged {W}itten equation}, Journal f\"ur die reine
  und angewandte Mathematik \textbf{Published Online} (2015), 1--88.

\bibitem[TX16]{Tian_Xu_2}
\bysame, \emph{Correlation functions in gauged linear $\sigma$-model}, Science
  China. Mathematics \textbf{59} (2016), 823--838.

\bibitem[TX17]{Tian_Xu_2017}
\bysame, \emph{The symplectic approach of gauged linear $\sigma$-model},
  Proceedings of the G\"okova Geometry-Topology Conference 2016 (Selman
  Akbulut, Denis Auroux, and Turgut \"Onder, eds.), 2017, pp.~86--111.

\bibitem[Ven15]{Venugopalan_quasi}
Sushmita Venugopalan, \emph{Vortices on surfaces with cylindrical ends},
  Journal of Geometry and Physics \textbf{98} (2015), 575--606.

\bibitem[Wit93]{Witten_LGCY}
Edward Witten, \emph{Phases of ${N}=2$ theories in two dimensions}, Nuclear
  Physics \textbf{B403} (1993), 159--222.

\end{thebibliography}

\bibliographystyle{amsalpha}

\end{document}